\documentclass[Body text point size,12pt]{article}
\usepackage{amssymb}
\usepackage{amsmath}
\usepackage{amsfonts}

\newtheorem{theorem}{Theorem}

\newtheorem{corollary}[theorem]{Corollary}

\newtheorem{definition}[theorem]{Definition}

\newtheorem{lemma}[theorem]{Lemma}

\newtheorem{proposition}[theorem]{Proposition}
\newtheorem{remark}[theorem]{Remark}

\newenvironment{proof}[1][Proof]{\noindent\textbf{#1.} }{\ \rule{0.5em}{0.5em}}
\input{tcilatex}
\begin{document}

\title{Asymptotic analysis of pollution filtration through thin random
fissures between two porous media\thanks{%
This work has been supported by the Comit\'{e} Mixte Franco-Marocain under
the PHC MA/08/183.}}
\author{A. Brillard$^{\text{a}}$\thanks{%
Corresponding author. Tel.: +33 3 89 33 63 10; fax: +33 3 89 33 63 19.%
\newline
\textit{Email addresses}: Alain.Brillard@uha.fr (A.\ Brillard),
m.eljarroudi@uae.ma (M.\ El Jarroudi).}, M. El Jarroudi$^{\text{b}}$, M. El
Merzguioui$^{\text{b}}$ \\
%EndAName
$^{\text{a}}${\small Universit\'{e} de Haute-Alsace,}\\
{\small Laboratoire de Gestion des Risques et Environnement,}\\
{\small 25\ rue de Chemnitz, F-68200\ Mulhouse, France}\\
$^{\text{b}}${\small Universit\'{e} Abdelmalek Essa\^{a}di, FST\ Tanger}\\
{\small D\'{e}partement de Math\'{e}matiques, B.P. 416, Tanger, Maroc}}
\date{ }
\maketitle

\begin{abstract}
We describe the asymptotic behavior of a filtration problem from a
contaminated porous medium to a non-contaminated porous medium through thin
vertical fissures of fixed height $h>0$, of random thinness of order $%
\varepsilon $ and which are $\varepsilon $-periodically distributed. We
compute the limit velocity of the flow and the limit flux of pollutant at
the interfaces between the two porous media and the intermediate one.
\end{abstract}

\section{Introduction}

We consider a porous medium which is contaminated by some pollutant and
which communicates with another porous and non-contaminated medium through
vertical fissures of height $h>0$, of random thinness of order $\varepsilon
>0$, and which are periodically disposed. Each of these porous media is
given a $\varepsilon $-periodic structure.\ Instead of considering a Stokes
problem in each of these porous media, we assume that the velocity of the
fluid is governed by a Darcy law with periodic permeability matrix.

The purpose of this work is to determine the influence of the fissures on
the transport of the contaminant into the non-contaminated medium, computing
the global flux of pollutant which penetrates in this non-contaminated
medium and the asymptotic velocity of the fluid which flows through the
fissures.

Let $\Omega $ be a bounded, smooth and open subset of $\mathbb{R}^{3}$, with
boundary $\Gamma $, such that

\begin{equation*}
\left\{ 
\begin{array}{rll}
\Omega ^{+} & = & \Omega \cap \left\{ x_{3}>0\right\} \neq \varnothing , \\ 
\Omega _{h}^{-} & = & \Omega \cap \left\{ x_{3}<-h\right\} \neq \varnothing .%
\end{array}%
\right.
\end{equation*}

Let $\Sigma \times \left\{ 0\right\} =\partial \Omega \cap \left\{
x_{3}=0\right\} $.$\ \Sigma $ is a bounded and smooth subset of $\mathbb{R}%
^{2}$. We define

\begin{equation*}
\left\{ 
\begin{array}{rll}
\Gamma _{0}^{+} & = & \Sigma \times \left\{ 0\right\} , \\ 
\Gamma _{h}^{-} & = & \Sigma \times \left\{ -h\right\} , \\ 
Y_{h} & = & \left\{ x\in \Omega \mid \left( x_{1},x_{2}\right) \in \Sigma 
\text{, }-h<x_{3}<0\right\} , \\ 
\Gamma ^{+} & = & \partial \Omega ^{+}\backslash \Gamma _{0}^{+}, \\ 
\Gamma ^{-} & = & \partial \Omega _{h}^{-}\backslash \Gamma _{h}^{-}.%
\end{array}%
\right.
\end{equation*}

The domain $\Omega $ is thus equal to $\Omega ^{+}\cup \Gamma _{0}^{+}\cup
Y_{h}\cup \Gamma _{h}^{-}\cup \Omega _{h}^{-}$.

Let $\left( \Pi ,\Upsilon ,P\right) $ be some probability space and $\left(
T\left( t\right) \right) _{t\in \mathbb{R}}$ be a group of transformations
on $\left( \Pi ,\Upsilon \right) $, that is satisfying%
\begin{equation*}
\left\{ 
\begin{array}{rlll}
T\left( 0\right) & = & Id_{\Pi }, &  \\ 
T\left( t_{1}+t_{2}\right) & = & T\left( t_{1}\right) \circ T\left(
t_{2}\right) & \forall t_{1},t_{2}\in \mathbb{R}, \\ 
P\left( T^{-1}\left( t\right) A\right) & = & P\left( A\right) & \forall A\in
\Upsilon \text{, }\forall t\in \mathbb{R},%
\end{array}%
\right.
\end{equation*}%
where $Id_{\Pi }$ is the identity map on $\Pi $ and the set $\left\{ \left(
t,\omega \right) \in 
%TCIMACRO{\U{211d} }%
%BeginExpansion
\mathbb{R}
%EndExpansion
\times \Pi \mid T\left( t\right) \omega \in A\right\} $ is $dt\times dP$
measurable, for every $A\in \Upsilon $. We suppose that $T$ is ergodic (or
metrically transitive), which means that every $A\in \Upsilon $ such that $%
T\left( t\right) A=A$, for every $t\in \mathbb{R}$, has a probability $%
P\left( A\right) $ equal to $0$ or $1$.

We introduce some random processes $q$ and $r$ defined on $\mathbb{R}\times
\Pi $ and satisfying the following conditions:

\begin{enumerate}
\item $q\left( t,\omega \right) $ is a stationary random process, that is,
for every positive integer $n$, every points $t_{1},\cdots ,t_{n}$, every $t$
in $\mathbb{R}$, and every $B\in \mathcal{B}\left( \mathbb{R}\right) $, one
has%
\begin{equation*}
\begin{array}{l}
P\left( \left\{ \omega \mid q\left( t+t_{1},\omega \right) ,\cdots ,q\left(
t+t_{n},\omega \right) \in B\right\} \right)  \\ 
\quad =P\left( \left\{ \omega \mid q\left( t_{1},T\left( t\right) \omega
\right) ,\cdots ,q\left( t_{n},T\left( t\right) \omega \right) \in B\right\}
\right) ,%
\end{array}%
\end{equation*}

where $\mathcal{B}\left( \mathbb{R}\right) $ is the Borel $\sigma $-algebra
on $\mathbb{R}$. Since $T$ preserves the measure $P$, the above equality
implies that the joint distribution of $\left\{ q\left( t_{1}\right) ,\cdots
,q\left( t_{n}\right) \right\} $ is the same as the joint distribution of $%
\left\{ q\left( t_{1}+t\right) ,\cdots ,q\left( t_{n}+t\right) \right\} $,
for every $t$ in $\mathbb{R}$.

\item The derivatives $\frac{d^{m}q}{dt^{m}}$ and $\frac{d^{m}r}{dt^{m}}$
exist for $m=1,2,3$ and there exist non-random constants $c_{1}$, $c_{2}$
and $c_{3}$ such that the following bounds hold true with probability 1%
\begin{equation}
0<c_{1}\leq q\left( t,\omega \right) \leq c_{2}<1\text{ ; }\left\vert
r\left( t,\omega \right) \right\vert \leq 1\text{ ; }\left\vert \dfrac{d^{m}q%
}{dt^{m}}\right\vert ,\left\vert \dfrac{d^{m}r}{dt^{m}}\right\vert \leq
c_{3},  \label{equ5}
\end{equation}
\end{enumerate}

From the properties of $T$ and $q$, we derive the ergodic theorem (see, for
example, \cite{Gikh} and \cite{Vent})%
\begin{equation}
\forall n\in \mathbb{Z}^{\ast }:\left\langle q^{n}\left( 0\right)
\right\rangle =\underset{\mathcal{T}\rightarrow \infty }{\lim }\frac{1}{2%
\mathcal{T}}\dint_{-\mathcal{T}}^{\mathcal{T}}q^{n}\left( t,\omega \right)
dt,  \label{Ergo}
\end{equation}%
almost surely, where the symbol $\left\langle .\right\rangle $ stands for
the mathematical expectation with respect to the measure $P$.

Let $\left( \alpha _{i}\left( \omega \right) \right) _{i\in 
%TCIMACRO{\U{2124} }%
%BeginExpansion
\mathbb{Z}
%EndExpansion
}$ and $\left( \beta _{i}\left( \omega \right) \right) _{i\in 
%TCIMACRO{\U{2124} }%
%BeginExpansion
\mathbb{Z}
%EndExpansion
}$ be sequences of random variables satisfying%
\begin{equation}
\left\vert \alpha _{i}\left( \omega \right) \right\vert \leq c_{4}\text{ ; }%
\left\vert \beta _{i}\left( \omega \right) \right\vert \leq c_{4}\text{, }%
\forall i\in 
%TCIMACRO{\U{2124} }%
%BeginExpansion
\mathbb{Z}
%EndExpansion
,  \label{random}
\end{equation}%
with probability 1, where $c_{4}$ is a non-random constant. We define, for
every $i,j\in 
%TCIMACRO{\U{2124} }%
%BeginExpansion
\mathbb{Z}
%EndExpansion
$, the fissure $Y_{\varepsilon ,ij}\left( \omega \right) $ as%
\begin{equation*}
Y_{\varepsilon ,ij}\left( \omega \right) =\left\{ 
\begin{array}{l}
x\in 
%TCIMACRO{\U{211d} }%
%BeginExpansion
\mathbb{R}
%EndExpansion
^{3}\mid \varepsilon a_{i}^{-}\left( -\varepsilon ^{-\theta }x_{3}\right)
<x_{1}-i\varepsilon <\varepsilon a_{i}^{+}\left( -\varepsilon ^{-\theta
}x_{3}\right) , \\ 
\quad \varepsilon a_{j}^{-}\left( -\varepsilon ^{-\theta }x_{3}\right)
<x_{2}-j\varepsilon <\varepsilon a_{j}^{+}\left( -\varepsilon ^{-\theta
}x_{3}\right) \text{, }x_{3}\in \left] -h,0\right[%
\end{array}%
\right\} ,
\end{equation*}%
with $0<\varepsilon <1$, $0<\theta <2/3$ and $a_{i}^{\pm }\left( z\right)
=r\left( z+\beta _{i}\left( \omega \right) ,\omega \right) \pm q\left(
z+\alpha _{i}\left( \omega \right) ,\omega \right) /2$. Let $I_{\varepsilon
}\left( \omega \right) =\left\{ \left( i,j\right) \in 
%TCIMACRO{\U{2124} }%
%BeginExpansion
\mathbb{Z}
%EndExpansion
^{2}\mid Y_{\varepsilon ,ij}\left( \omega \right) \subset Y_{h}\right\} $.\
We also define the sets%
\begin{equation*}
\left\{ 
\begin{array}{rllrll}
\Gamma _{0,\varepsilon ,ij}^{+}\left( \omega \right) & = & \partial
Y_{\varepsilon ,ij}\left( \omega \right) \cap \Gamma _{0}^{+}, & \Gamma
_{0,\varepsilon }^{+}\left( \omega \right) & = & \underset{\left( i,j\right)
\in I_{\varepsilon }\left( \omega \right) }{\cup }\Gamma _{0,\varepsilon
,ij}^{+}\left( \omega \right) , \\ 
\Gamma _{h,\varepsilon ,ij}^{-}\left( \omega \right) & = & \partial
Y_{\varepsilon ,ij}\left( \omega \right) \cap \Gamma _{h}^{-}, & \Gamma
_{h,\varepsilon }^{-}\left( \omega \right) & = & \underset{\left( i,j\right)
\in I_{\varepsilon }\left( \omega \right) }{\cup }\Gamma _{h,\varepsilon
,ij}^{-}\left( \omega \right) , \\ 
\Lambda _{\varepsilon }\left( \omega \right) & = & \partial Y_{\varepsilon
}\left( \omega \right) \backslash \left( \Gamma _{0,\varepsilon }^{+}\cup
\Gamma _{h,\varepsilon }^{-}\right) \left( \omega \right) & Y_{\varepsilon
}\left( \omega \right) & = & \underset{\left( i,j\right) \in I_{\varepsilon
}\left( \omega \right) }{\cup }Y_{\varepsilon ,ij}\left( \omega \right) .%
\end{array}%
\right.
\end{equation*}

Let $Z=\left] -1/2,1/2\right[ ^{3}$ be the unit cube of $%
%TCIMACRO{\U{211d} }%
%BeginExpansion
\mathbb{R}
%EndExpansion
^{3}$ and assume that it can be decomposed as $Z=Z^{1}\cup S\cup Z^{2}$,
where $Z^{1}$ and $Z^{2}$ are two disjoint, open and connected sets
separated by the smooth surface $S$ (Fig.\ 1).\FRAME{fhFU}{3.6354cm}{2.6755cm%
}{0pt}{\Qcb{A\ 2D view of the periodic structure of the porous media.}}{}{%
porous.eps}{\special{language "Scientific Word";type
"GRAPHIC";maintain-aspect-ratio TRUE;display "USEDEF";valid_file "F";width
3.6354cm;height 2.6755cm;depth 0pt;original-width 2.8072in;original-height
2.0531in;cropleft "0";croptop "0.9983";cropright "1.0004";cropbottom
"0";filename 'porous.eps';file-properties "XNPEU";}}

We assume that $\Omega ^{+}$ and $\Omega _{h}^{-}$ are two $\varepsilon $%
-periodic porous media which communicate through the fissures $%
Y_{\varepsilon ,ij}\left( \omega \right) $.\FRAME{fhFU}{6.2955cm}{1.8561cm}{%
0pt}{\Qcb{The porous media and the fissures.}}{}{channels1.eps}{\special%
{language "Scientific Word";type "GRAPHIC";maintain-aspect-ratio
TRUE;display "USEDEF";valid_file "F";width 6.2955cm;height 1.8561cm;depth
0pt;original-width 4.9018in;original-height 1.4062in;cropleft "0";croptop
"1";cropright "1";cropbottom "0";filename 'channels1.eps';file-properties
"XNPEU";}}

We set%
\begin{equation}
\left\{ 
\begin{array}{rllrll}
\Omega _{f}^{+,\varepsilon } & = & \Omega ^{+}\cap \left( \underset{k\in 
\mathbb{Z}^{3}}{\cup }\left( \varepsilon Z^{1}+k\varepsilon \right) \right) ,
& \Omega _{h,f}^{-,\varepsilon } & = & \Omega _{h}^{-}\cap \left( \underset{%
k\in \mathbb{Z}^{3}}{\cup }\left( \varepsilon Z^{1}+k\varepsilon \right)
\right) , \\ 
\Omega _{s}^{+,\varepsilon } & = & \Omega ^{+}\cap \left( \underset{k\in 
\mathbb{Z}^{3}}{\cup }\left( \varepsilon Z^{2}+k\varepsilon \right) \right) ,
& \Omega _{h,s}^{-,\varepsilon } & = & \Omega _{h}^{-}\cap \left( \underset{%
k\in \mathbb{Z}^{3}}{\cup }\left( \varepsilon Z^{2}+k\varepsilon \right)
\right) , \\ 
S_{\varepsilon }^{+} & = & \Omega ^{+}\cap \left( \underset{k\in \mathbb{Z}%
^{3}}{\cup }\left( \varepsilon S+k\varepsilon \right) \right) , & 
S_{\varepsilon }^{-} & = & \Omega _{h}^{-}\cap \left( \underset{k\in \mathbb{%
Z}^{3}}{\cup }\left( \varepsilon S+k\varepsilon \right) \right) .%
\end{array}%
\right.  \label{str-perio}
\end{equation}

We suppose that $\Omega _{f}^{+,\varepsilon }$ (resp. $\Omega
_{h,f}^{-,\varepsilon }$) is the portion of $\Omega ^{+}$ (resp. $\Omega
_{h}^{-}$) consisting of the pores which are filled in with some fluid and $%
\Omega _{s}^{+,\varepsilon }$ (resp. $\Omega _{h,s}^{-,\varepsilon }$) is
the portion of $\Omega ^{+}$ (resp. $\Omega _{h}^{-}$) consisting of the
non-porous rocks. We suppose that%
\begin{equation*}
\left\{ 
\begin{array}{ccc}
\partial \Omega _{f}^{+,\varepsilon }\cap \partial Y_{\varepsilon }\left(
\omega \right) & = & \Gamma _{0,\varepsilon }^{+}\left( \omega \right) , \\ 
\partial \Omega _{h,f}^{-,\varepsilon }\cap \partial Y_{\varepsilon }\left(
\omega \right) & = & \Gamma _{h,\varepsilon }^{-}\left( \omega \right) .%
\end{array}%
\right.
\end{equation*}

We define the fluid part of the domain as%
\begin{equation*}
\Omega _{f}^{\varepsilon }\left( \omega \right) =\Omega _{f}^{+,\varepsilon
}\cup \left( \Gamma _{0,\varepsilon }^{+}\cup Y_{\varepsilon }\cup \Gamma
_{h,\varepsilon }^{-}\right) \left( \omega \right) \cup \Omega
_{h,f}^{-,\varepsilon }.
\end{equation*}

Let $f\in C\left( \Omega ^{+}\right) $, $g^{+}\in \mathbf{L}^{2}\left(
\Omega ^{+};%
%TCIMACRO{\U{211d} }%
%BeginExpansion
\mathbb{R}
%EndExpansion
^{3}\right) \ $and $g^{-}\in \mathbf{L}^{2}\left( \Omega _{h}^{-};%
%TCIMACRO{\U{211d} }%
%BeginExpansion
\mathbb{R}
%EndExpansion
^{3}\right) $ be functions satisfying%
\begin{equation*}
supp\left( f\right) \subset \Omega ^{+}\text{ and }f\geq 0\text{ \ in }%
\Omega ^{+}\text{ ; }supp\left( g^{+}\right) \subset \Omega ^{+}\text{ ; }%
supp\left( g^{-}\right) \subset \Omega _{h}^{-}.
\end{equation*}

We consider in $\Omega _{f}^{\varepsilon }\left( \omega \right) $ the
reaction-diffusion problem with first-order reaction%
\begin{equation}
\left\{ 
\begin{array}{rlll}
-D\Delta u_{\varepsilon }+v_{\varepsilon }\cdot \nabla u_{\varepsilon }+%
\mathcal{R}u_{\varepsilon } & = & f & \text{in }\Omega _{f}^{\varepsilon
}\left( \omega \right) , \\ 
u_{\varepsilon } & = & 0 & \text{on }\Gamma ^{+}\cup \Gamma ^{-}, \\ 
\dfrac{\partial u_{\varepsilon }}{\partial n} & = & 0 & \text{on }%
S_{\varepsilon }^{+}\cup S_{\varepsilon }^{-}\cup \Lambda _{\varepsilon
}\left( \omega \right) ,%
\end{array}%
\right.  \label{eq9}
\end{equation}%
where\ $u_{\varepsilon }$ is the concentration of the pollutant, $D=D_{mol}$
is the molecular diffusion coefficient, $\mathcal{R}$ is a nonnegative
reaction coefficient, $n$ is the unit outer normal and $v_{\varepsilon }$ is
the velocity of the fluid, which is the solution of the Darcy-Stokes problems%
\begin{equation}
\left\{ 
\begin{array}{rlll}
\mu ^{+}\left( K_{\varepsilon }^{+}\right) ^{-1}v_{\varepsilon ,d}-\nabla
p_{\varepsilon ,d} & = & g^{+} & \text{in }\Omega _{f}^{+,\varepsilon }, \\ 
\mu ^{-}\left( K_{\varepsilon }^{-}\right) ^{-1}v_{\varepsilon ,d}-\nabla
p_{\varepsilon ,d} & = & g^{-} & \text{in }\Omega _{h,f}^{-,\varepsilon },
\\ 
\func{div}\left( v_{\varepsilon ,d}\right) & = & 0 & \text{in }\Omega
_{f}^{+,\varepsilon }\cup \Omega _{h,f}^{-,\varepsilon }, \\ 
v_{\varepsilon ,d}\cdot n & = & 0 & \text{on }\partial \Omega
_{f}^{+,\varepsilon }\cup \partial \Omega _{h,f}^{-,\varepsilon }\cup \Gamma
, \\ 
-\mu \varepsilon ^{2}\Delta v_{\varepsilon ,s}+\nabla p_{\varepsilon ,s} & =
& 0 & \text{in }Y_{\varepsilon }\left( \omega \right) , \\ 
\func{div}\left( v_{\varepsilon ,s}\right) & = & 0 & \text{in }%
Y_{\varepsilon }\left( \omega \right) , \\ 
v_{\varepsilon ,s} & = & 0 & \text{on\ }\Lambda _{\varepsilon }\left( \omega
\right) ,%
\end{array}%
\right.  \label{equ10}
\end{equation}%
with the following interface conditions

\begin{equation}
\left\{ 
\begin{array}{rlll}
\left( v_{\varepsilon ,s}\right) _{3} & = & \left( v_{\varepsilon ,d}\right)
_{3} & \text{on }\Gamma _{0,\varepsilon }^{+}\left( \omega \right) \cup
\Gamma _{h,\varepsilon }^{-}\left( \omega \right) , \\ 
\mu \varepsilon ^{2}\dfrac{\partial \left( v_{\varepsilon ,s}\right) _{3}}{%
\partial x_{3}}\mid _{x_{3}=0,-h} & = & p_{\varepsilon ,d}-p_{\varepsilon ,s}
& \text{on }\Gamma _{0,\varepsilon }^{+}\left( \omega \right) \cup \Gamma
_{h,\varepsilon }^{-}\left( \omega \right) , \\ 
\mu \varepsilon ^{2}\dfrac{\partial \left( v_{\varepsilon ,s}\right) _{\tau }%
}{\partial x_{3}} & = & -\gamma \left( K_{\varepsilon }^{+}\right)
^{-1/2}\left( v_{\varepsilon ,s}\right) _{\tau } & \text{on }\Gamma
_{0,\varepsilon }^{+}\left( \omega \right) , \\ 
\mu \varepsilon ^{2}\dfrac{\partial \left( v_{\varepsilon ,s}\right) _{\tau }%
}{\partial x_{3}} & = & \gamma \left( K_{\varepsilon }^{-}\right)
^{-1/2}\left( v_{\varepsilon ,s}\right) _{\tau } & \text{on }\Gamma
_{h,\varepsilon }^{-}\left( \omega \right) ,%
\end{array}%
\right.  \label{equ11}
\end{equation}%
where:

\begin{itemize}
\item $v_{\varepsilon ,d}$ and $p_{\varepsilon ,d}$ are respectively Darcy's
velocity and pressure in $\Omega _{f}^{+,\varepsilon }$ and $\Omega
_{h,f}^{-,\varepsilon }$,

\item $K_{\varepsilon }^{+}$ and $K_{\varepsilon }^{-}$ are the absolute
permeability matrices in $\Omega _{f}^{+,\varepsilon }$ and $\Omega
_{h,f}^{-,\varepsilon }$ respectively,

\item $v_{\varepsilon ,s}$ and $p_{\varepsilon ,s}$ are respectively the
velocity and the pressure of the Stokes flow in the fissures,

\item $\mu ^{+}$ (resp. $\mu ^{-}$, $\mu $) is the viscosity coefficient in $%
\Omega _{f}^{+,\varepsilon }$ (resp. in $\Omega _{h,f}^{-,\varepsilon }$, $%
Y_{\varepsilon }\left( \omega \right) $),

\item $\left( v_{\varepsilon ,s}\right) _{\tau }=\left( \left(
v_{\varepsilon ,s}\right) _{1},\left( v_{\varepsilon ,s}\right) _{2}\right) $
is the tangential velocity.
\end{itemize}

We suppose that the $3\times 3$ matrices $K_{\varepsilon }^{+}$ and $%
K_{\varepsilon }^{-}$ are defined through the $Z$-periodic construction: $%
K_{\varepsilon }^{+}=K^{+}\left( x/\varepsilon \right) $ and $K_{\varepsilon
}^{-}=K^{-}\left( x/\varepsilon \right) $, where $K^{+}$ and $K^{-}$ are
bounded, symmetric and positive definite, and that $\mu ^{\pm }$ and $\mu $
are positive constants. In the above interface conditions, (\ref{equ11})$%
_{1} $ means the continuity of the mass flux through the interfaces $\Gamma
_{0,\varepsilon }^{+}\left( \omega \right) $ and $\Gamma _{h,\varepsilon
}^{-}\left( \omega \right) $, (\ref{equ11})$_{2}$ represents the continuity
of the normal stress through the corresponding interface, and the two last
equalities of (\ref{equ11}) represent the Beavers-Joseph-Saffman conditions
on the tangential stress, with some nonnegative slippage coefficient $\gamma 
$ (see \cite{Arbo}, \cite{Beav} and \cite{Saff}).

We first describe the asymptotic behaviour of the solution $v_{\varepsilon }$
of (\ref{equ10})-(\ref{equ11}) using $\Gamma $-convergence methods (see \cite%
{Att} and \cite{Dal}, for the definition and the properties of this
variational convergence). We prove that the asymptotic velocities $%
v_{0,d}^{+}$ (in $\Omega ^{+}$) and $v_{0,d}^{-}$ (in $\Omega _{h}^{-}$) and
the asymptotic pressures $p_{0}^{+}$ (in $\Omega ^{+}$) and $p_{0}^{-}$ (in $%
\Omega _{h}^{-}$) are linked through the Darcy laws (\ref{Darcy}) in $\Omega
^{+}$ and $\Omega _{h}^{-}$ respectively.\ We also describe the asymptotic
behaviour of the velocity $v_{\varepsilon ,s}$ (see Corollary \ref%
{corollary1}). We then describe the asymptotic behaviour of the solution $%
u_{\varepsilon }$ of (\ref{eq9}) using the energy method. We prove that the
flux of pollutant through $\Gamma _{0}^{+}$ is given through (\ref{asympt})$%
_{3}$, while the flux through $\Gamma _{h}^{-}$ is given through (\ref%
{asympt})$_{5}$.

Homogenization theory introduced in the few past decades (see for instance 
\cite{Ben-Lio-Pap} and \cite{Jik-Koz-Ole} and the references therein) gives
powerful tools for the description of equivalent media which are
heterogeneous at a microscopic level. The homogenization of transport
problems of chemical products through porous media has been studied by many
authors (see for example \cite{Horn} and \cite{Mik}). A model of random
fissures has already been studied in \cite{Khrus} for a problem of
radiophysics posed in $%
%TCIMACRO{\U{211d} }%
%BeginExpansion
\mathbb{R}
%EndExpansion
_{2}^{+}\cup \left( \cup _{j\in 
%TCIMACRO{\U{2124} }%
%BeginExpansion
\mathbb{Z}
%EndExpansion
}Q_{j}^{\varepsilon }\left( \omega \right) \right) \cup 
%TCIMACRO{\U{211d} }%
%BeginExpansion
\mathbb{R}
%EndExpansion
_{2,h}^{-}$, where $Q_{j}^{\varepsilon }\left( \omega \right) $ is the $j$%
-th fissure and $%
%TCIMACRO{\U{211d} }%
%BeginExpansion
\mathbb{R}
%EndExpansion
_{2,h}^{-}=\left\{ x\in 
%TCIMACRO{\U{211d} }%
%BeginExpansion
\mathbb{R}
%EndExpansion
^{2}\mid x_{2}<-h\right\} $, $%
%TCIMACRO{\U{211d} }%
%BeginExpansion
\mathbb{R}
%EndExpansion
_{2}^{+}=\left\{ x\in 
%TCIMACRO{\U{211d} }%
%BeginExpansion
\mathbb{R}
%EndExpansion
^{2}\mid x_{2}>0\right\} $. We here adopt the shape of the fissures used in 
\cite{Khrus}, that we extend to a 3D case, in order to model in a quite
realistic way the constitution of the soil.

The paper is organized as follows. In the following section, we introduce
and study the appropriate local coordinates inside the fissures. In the
third section, we study the convergence of the velocities. The fourth part
is concerned with the asymptotic analysis of the contaminant transport
problem.\ We first deal with the case where $\mathcal{R}=0$.\ We build the
solutions of local problems in the neighborhood of the fissures in order to
pass to the limit in the original problem and study their asymptotic
properties. The fluxes across $\Gamma _{0}^{+}$ and $\Gamma _{h}^{-}$ given
in this section are those obtained in (\ref{asympt})$_{3,4}$, respectively,
with $\mathcal{R}=0$.\ In the last part of this section, we also consider
the case of a dispersive contaminant ($\mathcal{R}>0$) with random
dispersion in the fissures (see Remark \ref{Remark2}). We finally give the
asymptotic behavior of the fully reaction-diffusion problem. We here again
introduce the solutions of local problems.

\section{Local coordinates in the fissures}

In the fissure $Y_{\varepsilon ,ij}\left( \omega \right) $, $\left(
i,j\right) \in I_{\varepsilon }\left( \omega \right) $, we define, for a
fixed event $\omega $ for which the conditions (\ref{equ5}) and (\ref{random}%
) are satisfied, $\xi _{1}=x_{1}-i\varepsilon $, $\xi
_{2}=x_{2}-j\varepsilon $, $z=-x_{3}$, and introduce the curvilinear
coordinates $t\in \left( 0,h\right) $ and $y_{1},y_{2}\in \left(
-\varepsilon /2,\varepsilon /2\right) $.\ Thus doing, the lateral boundary
of the fissure coincides with the planes $y_{1},y_{2}=\pm \varepsilon /2$.
These coordinates are described through%
\begin{equation}
\left\{ 
\begin{array}{lll}
\Phi _{1,\varepsilon }\left( \xi _{1},\xi _{2},z,y_{1},y_{2},t\right) & = & 
\xi _{1}-a_{i}^{+}\left( \varepsilon ^{-\theta }z\right) \left( \dfrac{%
\varepsilon }{2}+y_{1}\right) -a_{i}^{-}\left( \varepsilon ^{-\theta
}z\right) \left( \dfrac{\varepsilon }{2}-y_{1}\right) =0, \\ 
\Phi _{2,\varepsilon }\left( \xi _{1},\xi _{2},z,y_{1},y_{2},t\right) & = & 
\xi _{2}-a_{j}^{+}\left( \varepsilon ^{-\theta }z\right) \left( \dfrac{%
\varepsilon }{2}+y_{2}\right) -a_{j}^{-}\left( \varepsilon ^{-\theta
}z\right) \left( \dfrac{\varepsilon }{2}-y_{2}\right) =0, \\ 
\Phi _{3,\varepsilon }\left( \xi _{1},\xi _{2},z,y_{1},y_{2},t\right) & = & 
z-\varepsilon ^{\theta }\psi _{\varepsilon }\left( \dfrac{\xi _{1}}{%
\varepsilon },\dfrac{\xi _{2}}{\varepsilon },\varepsilon ^{-\theta }t\right)
=0.%
\end{array}%
\right.  \label{coor-loc1}
\end{equation}

Defining $\zeta _{1}=\xi _{1}/\varepsilon ,\zeta _{2}=\xi _{2}/\varepsilon $
and $\tau =\varepsilon ^{-\theta }t$, the orthogonality conditions of the
coordinates given through%
\begin{equation*}
\left\{ 
\begin{array}{ccl}
\dfrac{\partial \Phi _{1,\varepsilon }}{\partial \zeta _{1}}\dfrac{\partial
\Phi _{3,\varepsilon }}{\partial \zeta _{1}}+\dfrac{\partial \Phi
_{1,\varepsilon }}{\partial z}\dfrac{\partial \Phi _{3,\varepsilon }}{%
\partial z} & = & 0, \\ 
\dfrac{\partial \Phi _{2,\varepsilon }}{\partial \zeta _{2}}\dfrac{\partial
\Phi _{3,\varepsilon }}{\partial \zeta _{2}}+\dfrac{\partial \Phi
_{2,\varepsilon }}{\partial z}\dfrac{\partial \Phi _{3,\varepsilon }}{%
\partial z} & = & 0%
\end{array}%
\right.
\end{equation*}%
and the condition $\psi _{\varepsilon }\left( 0,0,\varepsilon ^{-\theta
}t\right) =\tau $ imply the following Cauchy system

\begin{equation}
\left\{ 
\begin{array}{rll}
\dfrac{\partial \psi _{\varepsilon }}{\partial \zeta _{1}} & = & \varepsilon
^{2\left( 1-\theta \right) }\left( 
\begin{array}{l}
\left( a_{i}^{-}\left( \psi _{\varepsilon }\right) \right) ^{\prime }\dfrac{%
\zeta _{1}-a_{i}^{+}\left( \psi _{\varepsilon }\right) }{a_{i}^{+}\left(
\psi _{\varepsilon }\right) -a_{i}^{-}\left( \psi _{\varepsilon }\right) }
\\ 
\quad -\left( a_{i}^{+}\left( \psi _{\varepsilon }\right) \right) ^{\prime }%
\dfrac{\zeta _{1}-a_{i}^{-}\left( \psi _{\varepsilon }\right) }{%
a_{i}^{+}\left( \psi _{\varepsilon }\right) -a_{i}^{-}\left( \psi
_{\varepsilon }\right) }%
\end{array}%
\right) , \\ 
\dfrac{\partial \psi _{\varepsilon }}{\partial \zeta _{2}} & = & \varepsilon
^{2\left( 1-\theta \right) }\left( 
\begin{array}{l}
\left( a_{j}^{-}\left( \psi _{\varepsilon }\right) \right) ^{\prime }\dfrac{%
\zeta _{2}-a_{j}^{+}\left( \psi _{\varepsilon }\right) }{a_{j}^{+}\left(
\psi _{\varepsilon }\right) -a_{j}^{-}\left( \psi _{\varepsilon }\right) }
\\ 
\quad -\left( a_{j}^{+}\left( \psi _{\varepsilon }\right) \right) ^{\prime }%
\dfrac{\zeta _{2}-a_{j}^{-}\left( \psi _{\varepsilon }\right) }{%
a_{j}^{+}\left( \psi _{\varepsilon }\right) -a_{j}^{-}\left( \psi
_{\varepsilon }\right) }%
\end{array}%
\right) , \\ 
\psi _{\varepsilon }\left( 0,0,\varepsilon ^{-\theta }t\right) & = & \tau ,%
\end{array}%
\right.  \label{coor-loc2}
\end{equation}%
which has to be solved.

\begin{lemma}
\label{lemma1}

\begin{enumerate}
\item The system (\ref{coor-loc2}) has a unique solution $\psi _{\varepsilon
}\left( \zeta _{1},\zeta _{2},t\right) =\tau +\psi _{\varepsilon }^{1}\left(
\zeta _{1},\tau \right) +\psi _{\varepsilon }^{2}\left( \zeta _{2},\tau
\right) $, with $\psi _{\varepsilon }^{1}\left( 0,\tau \right) =\psi
_{\varepsilon }^{2}\left( 0,\tau \right) =0$.

\item For every $k\in 
%TCIMACRO{\U{2115} }%
%BeginExpansion
\mathbb{N}
%EndExpansion
^{\ast }$, every $\zeta _{1}$, $\zeta _{2}\in \left[ -k,k\right] $, and
every $\tau \in 
%TCIMACRO{\U{211d} }%
%BeginExpansion
\mathbb{R}
%EndExpansion
$, one has, when $\varepsilon $ is close to $0$%
\begin{equation*}
\left\{ 
\begin{array}{rlll}
\psi _{\varepsilon }\left( \zeta _{1},\zeta _{2},t\right) & = & \tau
+O\left( \varepsilon ^{2\left( 1-\theta \right) }\right) , &  \\ 
\dfrac{\partial \psi _{\varepsilon }}{\partial \tau } & = & 1+O\left(
\varepsilon ^{2\left( 1-\theta \right) }\right) , &  \\ 
\dfrac{\partial \psi _{\varepsilon }}{\partial \zeta _{\alpha }},\dfrac{%
\partial ^{2}\psi _{\varepsilon }}{\partial \tau ^{2}},\dfrac{\partial
^{2}\psi _{\varepsilon }}{\partial \zeta _{\alpha }\partial \tau } & = & 
O\left( \varepsilon ^{2\left( 1-\theta \right) }\right) & \alpha =1,2.%
\end{array}%
\right.
\end{equation*}
\end{enumerate}
\end{lemma}

\begin{proof}
1. Observe that $a_{i}^{+}\left( \psi _{\varepsilon }\right)
-a_{i}^{-}\left( \psi _{\varepsilon }\right) =q\left( \psi _{\varepsilon
}+\alpha _{i}\left( \omega \right) ,\omega \right) \geq c_{1}>0$ and $%
\left\vert \left( a_{i}^{\pm }\left( \psi _{\varepsilon }\right) \right)
^{\prime }\right\vert \leq c_{3}$, thanks to (\ref{equ5}). One deduces that,
for every $k\in 
%TCIMACRO{\U{2115} }%
%BeginExpansion
\mathbb{N}
%EndExpansion
^{\ast }$ and every $\zeta _{1}$, $\zeta _{2}\in \left[ -k,k\right] $, the
functions%
\begin{equation*}
\begin{array}{lll}
\left( \zeta _{1},\psi _{\varepsilon }\right) & \longmapsto & \left(
a_{i}^{-}\left( \psi _{\varepsilon }\right) \right) ^{\prime }\dfrac{\zeta
_{1}-a_{i}^{+}\left( \psi _{\varepsilon }\right) }{a_{i}^{+}\left( \psi
_{\varepsilon }\right) -a_{i}^{-}\left( \psi _{\varepsilon }\right) }-\left(
a_{i}^{+}\left( \psi _{\varepsilon }\right) \right) ^{\prime }\dfrac{\zeta
_{1}-a_{i}^{-}\left( \psi _{\varepsilon }\right) }{a_{i}^{+}\left( \psi
_{\varepsilon }\right) -a_{i}^{-}\left( \psi _{\varepsilon }\right) }, \\ 
\left( \zeta _{2},\psi _{\varepsilon }\right) & \longmapsto & \left(
a_{j}^{-}\left( \psi _{\varepsilon }\right) \right) ^{\prime }\dfrac{\zeta
_{2}-a_{j}^{+}\left( \psi _{\varepsilon }\right) }{a_{j}^{+}\left( \psi
_{\varepsilon }\right) -a_{j}^{-}\left( \psi _{\varepsilon }\right) }-\left(
a_{j}^{+}\left( \psi _{\varepsilon }\right) \right) ^{\prime }\dfrac{\zeta
_{2}-a_{j}^{-}\left( \psi _{\varepsilon }\right) }{a_{j}^{+}\left( \psi
_{\varepsilon }\right) -a_{j}^{-}\left( \psi _{\varepsilon }\right) },%
\end{array}%
\end{equation*}%
are bounded and Lipschitz continuous with respect to $\left( \zeta _{1},\psi
_{\varepsilon }\right) $ and $\left( \zeta _{2},\psi _{\varepsilon }\right) $
respectively. The Cauchy problem (\ref{coor-loc2}) thus has a unique
solution $\psi _{\varepsilon }$. The condition $\psi _{\varepsilon }\left(
0,0,\tau \right) =\tau $ implies $\psi _{\varepsilon }\left( \zeta
_{1},\zeta _{2},t\right) =\tau +\psi _{\varepsilon }^{1}\left( \zeta
_{1},\tau \right) +\psi _{\varepsilon }^{2}\left( \zeta _{2},\tau \right) $,
with $\psi _{\varepsilon }^{1}\left( 0,\tau \right) =\psi _{\varepsilon
}^{2}\left( 0,\tau \right) =0$.

2. Thanks to the preceding point and to the hypothesis (\ref{equ5}), the
quantities $\psi _{\varepsilon }^{1}$, $\psi _{\varepsilon }^{2}$, $\frac{%
\partial ^{2}\psi _{\varepsilon }}{\partial ^{2}\tau }$, $\frac{\partial
\psi _{\varepsilon }}{\partial \zeta _{\alpha }}$ and $\frac{\partial
^{2}\psi _{\varepsilon }}{\partial \zeta _{\alpha }\partial \tau }$ are
equal to some $O\left( \varepsilon ^{2\left( 1-\theta \right) }\right) $, $%
\alpha =1,2$.
\end{proof}

\begin{lemma}
\label{lemma2}The system (\ref{coor-loc1}) has a unique solution $\left( 
\bar{\xi}_{1}\left( y_{1,}y_{2},t\right) ,\bar{\xi}_{2}\left(
y_{1,}y_{2},t\right) ,\bar{z}\left( y_{1},y_{2},t\right) \right) $.
\end{lemma}

\begin{proof}
The Jacobian $\overline{\Delta }$ of the matrix $\frac{\partial \left( \Phi
_{1,\varepsilon },\Phi _{2,\varepsilon },\Phi _{3,\varepsilon }\right) }{%
\partial \left( \xi _{1},\xi _{2},z\right) }$ can be computed as%
\begin{equation*}
\begin{array}{l}
\overline{\Delta }=1+\dfrac{\partial \psi _{\varepsilon }}{\partial \zeta
_{1}}\left( \left( a_{i}^{-}\right) ^{\prime }\dfrac{\zeta _{1}-a_{i}^{+}}{%
a_{i}^{+}-a_{i}^{-}}-\left( a_{i}^{+}\right) ^{\prime }\dfrac{\zeta
_{1}-a_{i}^{-}}{a_{i}^{+}-a_{i}^{-}}\right) \\ 
\quad +\dfrac{\partial \psi _{\varepsilon }}{\partial \zeta _{2}}\left(
\left( a_{j}^{-}\right) ^{\prime }\dfrac{\zeta _{2}-a_{j}^{+}}{%
a_{j}^{+}-a_{j}^{-}}-\left( a_{j}^{+}\right) ^{\prime }\dfrac{\zeta
_{2}-a_{j}^{-}}{a_{j}^{+}-a_{j}^{-}}\right) .%
\end{array}%
\end{equation*}

Thanks to Lemma \ref{lemma1}, one has $\overline{\Delta }=1+O\left(
\varepsilon ^{2\left( 1-\theta \right) }\right) $, when $\varepsilon $ is
close to $0$. One deduces, using the implicit function theorem, that the
system (\ref{coor-loc1}) has a unique solution $\left( \bar{\xi}_{1},\bar{\xi%
}_{2},\bar{z}\right) $.
\end{proof}

Thanks to the implicit function theorem, one has%
\begin{equation}
\left\{ 
\begin{array}{lll}
\dfrac{\partial \bar{\xi}_{1}}{\partial y_{1}} & = & \dfrac{1}{\overline{%
\Delta }}\left( a_{i}^{+}-a_{i}^{-}\right) \left( 1-\dfrac{\partial \psi
_{\varepsilon }}{\partial \zeta _{2}}\widetilde{\Phi }_{\varepsilon ,j}+%
\dfrac{\partial \psi _{\varepsilon }}{\partial \zeta _{1}}\widetilde{\Phi }%
_{\varepsilon ,i}\right) , \\ 
\dfrac{\partial \bar{\xi}_{1}}{\partial y_{2}} & = & \dfrac{2}{\overline{%
\Delta }}\left( a_{j}^{+}-a_{j}^{-}\right) \dfrac{\partial \psi
_{\varepsilon }}{\partial \zeta _{2}}\widetilde{\Phi }_{\varepsilon ,i}, \\ 
\dfrac{\partial \bar{\xi}_{1}}{\partial t} & = & \dfrac{2}{\overline{\Delta }%
}\varepsilon ^{\left( 1-\theta \right) }\dfrac{\partial \psi _{\varepsilon }%
}{\partial \tau }\widetilde{\Phi }_{\varepsilon ,i},%
\end{array}%
\right.  \label{coor-loc3}
\end{equation}%
where%
\begin{equation*}
\begin{array}{lll}
\widetilde{\Phi }_{\varepsilon ,i} & = & \left( a_{i}^{-}\right) ^{\prime }%
\dfrac{\zeta _{1}-a_{i}^{+}}{a_{i}^{+}-a_{i}^{-}}-\left( a_{i}^{+}\right)
^{\prime }\dfrac{\zeta _{1}-a_{i}^{-}}{a_{i}^{+}-a_{i}^{-}}, \\ 
\widetilde{\Phi }_{\varepsilon ,j} & = & \left( a_{j}^{-}\right) ^{\prime }%
\dfrac{\zeta _{2}-a_{j}^{+}}{a_{j}^{+}-a_{j}^{-}}-\left( a_{j}^{+}\right)
^{\prime }\dfrac{\zeta _{2}-a_{j}^{-}}{a_{j}^{+}-a_{j}^{-}}.%
\end{array}%
\end{equation*}

One also has%
\begin{equation}
\left\{ 
\begin{array}{lllrll}
\dfrac{\partial \bar{\xi}_{2}}{\partial y_{1}} & = & \dfrac{2\left(
a_{i}^{+}-a_{i}^{-}\right) }{\overline{\Delta }}\dfrac{\partial \psi
_{\varepsilon }}{\partial \zeta _{2}}\widetilde{\Phi }_{\varepsilon ,j}, & 
\dfrac{\partial \bar{z}}{\partial y_{1}} & = & \dfrac{2}{\overline{\Delta }}%
\varepsilon ^{\theta -1}\dfrac{\partial \psi _{\varepsilon }}{\partial \zeta
_{1}}\left( a_{i}^{+}-a_{i}^{-}\right) , \\ 
\dfrac{\partial \bar{\xi}_{2}}{\partial y_{2}} & = & \dfrac{%
a_{j}^{+}-a_{j}^{-}}{\overline{\Delta }}\left( 
\begin{array}{l}
1-\dfrac{\partial \psi _{\varepsilon }}{\partial \zeta _{1}}\widetilde{\Phi }%
_{\varepsilon ,i} \\ 
\quad +\dfrac{\partial \psi _{\varepsilon }}{\partial \zeta _{2}}\widetilde{%
\Phi }_{\varepsilon ,j}%
\end{array}%
\right) , & \dfrac{\partial \bar{z}}{\partial y_{2}} & = & \dfrac{2}{%
\overline{\Delta }}\varepsilon ^{\theta -1}\dfrac{\partial \psi
_{\varepsilon }}{\partial \zeta _{2}}\left( a_{j}^{+}-a_{j}^{-}\right) , \\ 
\dfrac{\partial \bar{\xi}_{2}}{\partial t} & = & \dfrac{2}{\overline{\Delta }%
}\varepsilon ^{1-\theta }\dfrac{\partial \psi _{\varepsilon }}{\partial \tau 
}\widetilde{\Phi }_{\varepsilon ,j}, & \dfrac{\partial \bar{z}}{\partial t}
& = & \dfrac{1}{\overline{\Delta }}\dfrac{\partial \psi _{\varepsilon }}{%
\partial \tau }\left( 
\begin{array}{l}
1+\dfrac{\partial \psi _{\varepsilon }}{\partial \zeta _{1}}\widetilde{\Phi }%
_{\varepsilon ,i} \\ 
\quad +\dfrac{\partial \psi _{\varepsilon }}{\partial \zeta _{2}}\widetilde{%
\Phi }_{\varepsilon ,j}%
\end{array}%
\right) .%
\end{array}%
\right.  \label{coor-loc4}
\end{equation}

Let $\left( g_{\varepsilon ,\alpha \beta }\right) _{\alpha ,\beta =1,2,3}$
be the metric tensor associated to the local basis defined through the
vectors (\ref{coor-loc3}) and (\ref{coor-loc4}). One has the following
result.

\begin{lemma}
\label{lemme3}

\begin{enumerate}
\item The metric tensor $\left( g_{\varepsilon ,\alpha \beta }\right)
_{\alpha ,\beta =1,2,3}$ satisfies a symmetry property and the following
behaviour%
\begin{equation*}
\begin{array}{rllrll}
g_{\varepsilon ,11} & = & \left( q_{i}\right) ^{2}\left( \varepsilon
^{-\theta }t\right) +O\left( \varepsilon ^{2\left( 1-\theta \right) }\right)
, & g_{\varepsilon ,22} & = & \left( q_{j}\right) ^{2}\left( \varepsilon
^{-\theta }t\right) +O\left( \varepsilon ^{2\left( 1-\theta \right) }\right)
, \\ 
g_{\varepsilon ,12} & = & O\left( \varepsilon ^{2\left( 1-\theta \right)
}\right) , & g_{\varepsilon ,23} & = & O\left( \varepsilon ^{2\left(
1-\theta \right) }\right) , \\ 
g_{\varepsilon ,13} & = & O\left( \varepsilon ^{2\left( 1-\theta \right)
}\right) , & g_{\varepsilon ,33} & = & 1+O\left( \varepsilon ^{2\left(
1-\theta \right) }\right) ,%
\end{array}%
\end{equation*}%
where $q_{i}\left( \varepsilon ^{-\theta }t\right) =q\left( \varepsilon
^{-\theta }t+\alpha _{i}\left( \omega \right) ,\omega \right) $, $\forall
i\in 
%TCIMACRO{\U{2124} }%
%BeginExpansion
\mathbb{Z}
%EndExpansion
$.

\item The contravariant components $\left( g_{\varepsilon }^{\alpha \beta
}\right) _{\alpha ,\beta =1,2,3}$ of $\left( g_{\varepsilon ,\alpha \beta
}\right) _{\alpha ,\beta =1,2,3}$ satisfy a symmetry property and%
\begin{equation*}
\begin{array}{rllrll}
g_{\varepsilon }^{11} & = & \dfrac{1}{\left( q_{i}\right) ^{2}\left(
\varepsilon ^{-\theta }t\right) }+O\left( \varepsilon ^{2\left( 1-\theta
\right) }\right) , & g_{\varepsilon }^{23} & = & O\left( \varepsilon
^{2\left( 1-\theta \right) }\right) , \\ 
g_{\varepsilon }^{12} & = & O\left( \varepsilon ^{2\left( 1-\theta \right)
}\right) , & g_{\varepsilon }^{13} & = & O\left( \varepsilon ^{2\left(
1-\theta \right) }\right) , \\ 
g_{\varepsilon }^{22} & = & \dfrac{1}{\left( q_{j}\right) ^{2}\left(
\varepsilon ^{-\theta }t\right) }+O\left( \varepsilon ^{2\left( 1-\theta
\right) }\right) , & g_{\varepsilon }^{33} & = & 1+O\left( \varepsilon
^{2\left( 1-\theta \right) }\right) .%
\end{array}%
\end{equation*}
\end{enumerate}
\end{lemma}

\begin{proof}
Observing that%
\begin{equation*}
\begin{array}{rll}
\left\vert \dfrac{\partial }{\partial t}q_{i}\left( \varepsilon ^{-\theta
}z\right) -\dfrac{\partial }{\partial t}q_{i}\left( \varepsilon ^{-\theta
}t\right) \right\vert & = & O\left( \varepsilon ^{2-3\theta }\right) , \\ 
\left\vert q_{i}\left( \varepsilon ^{-\theta }z\right) -q_{i}\left(
\varepsilon ^{-\theta }t\right) \right\vert & = & O\left( \varepsilon
^{2\left( 1-\theta \right) }\right) , \\ 
\det \left( g_{\varepsilon ,\alpha \beta }\right) & = & \left( q_{i}\right)
^{2}\left( \varepsilon ^{-\theta }t\right) \left( q_{j}\right) ^{2}\left(
\varepsilon ^{-\theta }t\right) +O\left( \varepsilon ^{2\left( 1-\theta
\right) }\right) ,%
\end{array}%
\end{equation*}%
these formulas are direct consequences of (\ref{coor-loc3})-(\ref{coor-loc4}%
).
\end{proof}

One deduces from the preceding results that the gradient of a function $u$
expressed in the local coordinates $\left( y_{1},y_{2},t\right) $ inside the
fissure $Y_{\varepsilon ,ij}\left( \omega \right) $ is of the form%
\begin{equation*}
\nabla u=\left( Id+O\left( \varepsilon ^{2\left( 1-\theta \right) }\right)
\right) \left( \dfrac{1}{q_{i}}\dfrac{\partial u}{\partial y_{1}},\dfrac{1}{%
q_{j}}\dfrac{\partial u}{\partial y_{2}},\dfrac{\partial u}{\partial t}%
\right) ,
\end{equation*}%
for some non-diagonal matrix $O\left( \varepsilon ^{2\left( 1-\theta \right)
}\right) $.

Hence, using the formula $\Delta u=\frac{1}{\sqrt{\left\vert g\right\vert }}%
\partial _{\alpha }\left( \sqrt{\left\vert g\right\vert }g^{\alpha \beta
}\partial _{\beta }u\right) $, for $\alpha ,\beta =1,2,3$, with $\left\vert
g\right\vert =\left\vert \det \left( g_{\alpha \beta }\right) \right\vert $,
one has%
\begin{equation*}
\Delta u=\left( \frac{1}{\left( q_{i}\right) ^{2}}\frac{\partial ^{2}u}{%
\partial \left( y_{1}\right) ^{2}}+\frac{1}{\left( q_{j}\right) ^{2}}\frac{%
\partial ^{2}u}{\partial \left( y_{2}\right) ^{2}}+\frac{1}{q_{i}q_{j}}\frac{%
\partial }{\partial t}\left( q_{i}q_{j}\frac{\partial u}{\partial t}\right)
\right) \left( 1+O\left( \varepsilon ^{2\left( 1-\theta \right) }\right)
\right) .
\end{equation*}

\section{Study of the fluid flow}

\subsection{Existence of a weak solution and a priori estimates}

We define the functional space%
\begin{equation*}
\mathbf{V}_{\varepsilon }=\left\{ 
\begin{array}{l}
v\in \mathbf{L}^{2}\left( \Omega _{f}^{\varepsilon }\left( \omega \right) ;%
%TCIMACRO{\U{211d} }%
%BeginExpansion
\mathbb{R}
%EndExpansion
^{3}\right) \mid \func{div}\left( v\right) =0\text{ in }\Omega
_{f}^{\varepsilon }\left( \omega \right) \text{, }v\mid _{Y_{\varepsilon
}\left( \omega \right) }\in \mathbf{H}^{1}\left( Y_{\varepsilon }\left(
\omega \right) ;%
%TCIMACRO{\U{211d} }%
%BeginExpansion
\mathbb{R}
%EndExpansion
^{3}\right) \text{,} \\ 
\quad v=0\text{ on }\Lambda _{\varepsilon }\left( \omega \right) \text{, }%
v\cdot n=0\text{ on }\partial \Omega _{f}^{+,\varepsilon }\cup \partial
\Omega _{h,f}^{-,\varepsilon }\cup \Gamma%
\end{array}%
\right\} .
\end{equation*}

$\mathbf{V}_{\varepsilon }$ is a complete space when endowed with the norm%
\begin{equation*}
\left\Vert v\right\Vert _{\mathbf{V}_{\varepsilon }}=\left( \left\Vert
v\right\Vert _{\mathbf{L}^{2}\left( \Omega _{f}^{+,\varepsilon };%
%TCIMACRO{\U{211d} }%
%BeginExpansion
\mathbb{R}
%EndExpansion
^{3}\right) }^{2}+\left\Vert v\right\Vert _{\mathbf{L}^{2}\left( \Omega
_{h,f}^{-,\varepsilon };%
%TCIMACRO{\U{211d} }%
%BeginExpansion
\mathbb{R}
%EndExpansion
^{3}\right) }^{2}+\left\Vert \nabla v\right\Vert _{\mathbf{L}^{2}\left(
Y_{\varepsilon }\left( \omega \right) ;%
%TCIMACRO{\U{211d} }%
%BeginExpansion
\mathbb{R}
%EndExpansion
^{9}\right) }^{2}\right) ^{1/2}.
\end{equation*}

Multiplying (\ref{equ10}) by $\Phi \in \mathbf{V}_{\varepsilon }$, using
Green's formula and the conditions (\ref{equ11}), we obtain the following
variational formulation%
\begin{equation}
\begin{array}{l}
\mu ^{+}\dint\nolimits_{\Omega _{f}^{+,\varepsilon }}\left( K_{\varepsilon
}^{+}\right) ^{-1}v_{\varepsilon ,d}\cdot \Phi dx+\mu
^{-}\dint\nolimits_{\Omega _{h,f}^{-,\varepsilon }}\left( K_{\varepsilon
}^{-}\right) ^{-1}v_{\varepsilon ,d}\cdot \Phi dx+\mu \varepsilon
^{2}\dint\nolimits_{Y_{\varepsilon }\left( \omega \right) }\nabla
v_{\varepsilon ,s}\cdot \nabla \Phi dx \\ 
\quad +\gamma \dint\nolimits_{\Gamma _{0,\varepsilon }^{+}}\left(
K_{\varepsilon }^{+}\right) ^{-1/2}\left( v_{\varepsilon ,s}\right) _{\tau
}\cdot \left( \Phi \right) _{\tau }dx^{\prime }+\gamma
\dint\nolimits_{\Gamma _{h,\varepsilon }^{-}}\left( K_{\varepsilon
}^{-}\right) ^{-1/2}\left( v_{\varepsilon ,s}\right) _{\tau }\cdot \left(
\Phi \right) _{\tau }dx^{\prime } \\ 
\quad =\dint\nolimits_{\Omega _{f}^{+,\varepsilon }}g^{+}\cdot \Phi
dx+\dint\nolimits_{\Omega _{h,f}^{-,\varepsilon }}g^{-}\cdot \Phi dx,%
\end{array}
\label{equ13}
\end{equation}%
where $x^{\prime }=\left( x_{1},x_{2}\right) $. Using standard arguments, we
immediately deduce from this variational formulation that the system (\ref%
{equ10})-(\ref{equ11}) has a unique weak solution $\left( v_{\varepsilon
},p_{\varepsilon }\right) \in \mathbf{V}_{\varepsilon }\times L^{2}\left(
\Omega _{\varepsilon }\left( \omega \right) \right) /\mathbb{R}$.

\begin{lemma}
\label{estim-vites}There exists a non-random constant $C$ independent of $%
\varepsilon $ such that%
\begin{equation*}
\begin{array}{rclrll}
\dint\nolimits_{Y_{\varepsilon }\left( \omega \right) }\left\vert
v_{\varepsilon }\right\vert ^{2}dx & \leq & C\varepsilon
^{2}\dint\nolimits_{Y_{\varepsilon }\left( \omega \right) }\left\vert \nabla
v_{\varepsilon }\right\vert ^{2}dx, & \dint\nolimits_{\Omega
_{f}^{\varepsilon }\left( \omega \right) }\left\vert v_{\varepsilon
}\right\vert ^{2}dx & \leq & C, \\ 
\dint\nolimits_{\Gamma _{0,\varepsilon }^{+}}\left\vert v_{\varepsilon
}\right\vert ^{2}dx^{\prime }+\dint\nolimits_{\Gamma _{h,\varepsilon
}^{-}}\left\vert v_{\varepsilon }\right\vert ^{2}dx^{\prime } & \leq & 
C\varepsilon ^{2}, & \varepsilon ^{2}\dint\nolimits_{Y_{\varepsilon }\left(
\omega \right) }\left\vert \nabla v_{\varepsilon }\right\vert ^{2}dx & \leq
& C.%
\end{array}%
\end{equation*}
\end{lemma}

\begin{proof}
Define the normalized fissure%
\begin{equation}
Z_{\varepsilon ,ij}\left( \omega \right) =\left\{ 
\begin{array}{l}
\left( z_{1},z_{2},x_{3}\right) \in 
%TCIMACRO{\U{211d} }%
%BeginExpansion
\mathbb{R}
%EndExpansion
^{2}\mid a_{i}^{-}\left( -\varepsilon ^{-\theta }x_{3}\right)
<z_{1}<a_{i}^{+}\left( -\varepsilon ^{-\theta }x_{3}\right) , \\ 
\text{\quad }a_{j}^{-}\left( -\varepsilon ^{-\theta }x_{3}\right)
<z_{2}<a_{j}^{+}\left( -\varepsilon ^{-\theta }x_{3}\right) \text{, }%
x_{3}\in \left( -h,0\right)%
\end{array}%
\right\} .  \label{equ14}
\end{equation}

For every $\psi \in C^{1}\left( Z_{\varepsilon ,ij}\left( \omega \right)
\right) $ such that $\psi =0$ on the lateral boundary of $Z_{\varepsilon
,ij}\left( \omega \right) $, one has inside $Z_{\varepsilon ,ij}\left(
\omega \right) $%
\begin{equation*}
\left( \psi \left( s,z_{2},x_{3}\right) \right) ^{2}=\left(
\dint\nolimits_{a_{i}^{-}}^{s}\frac{\partial \psi }{\partial z_{1}}%
dz_{1}\right) ^{2}\leq q_{i}\left( -\varepsilon ^{-\theta }x_{3}\right)
\dint\nolimits_{a_{i}^{-}}^{a_{i}^{+}}\left( \frac{\partial \psi }{\partial
z_{1}}\right) ^{2}dz_{1}.
\end{equation*}

Thanks to the hypothesis (\ref{equ5}), one has the following Poincar\'{e}
estimate%
\begin{equation*}
\dint\nolimits_{Z_{\varepsilon ,ij}\left( \omega \right) }\psi
^{2}dz_{1}dz_{2}\leq C\dint\nolimits_{Z_{\varepsilon ,ij}\left( \omega
\right) }\left\vert \nabla \psi \right\vert ^{2}dz_{1}dz_{2}.
\end{equation*}

Defining $z_{1}=\left( x_{1}-\varepsilon i\right) /\varepsilon $ and $%
z_{2}=\left( x_{2}-\varepsilon j\right) /\varepsilon $, one gets the first
estimate.

We replace $\Phi $ by $v_{\varepsilon }$ in the variational formulation (\ref%
{equ13}). Since the matrices $K_{\varepsilon }^{+}$ and $K_{\varepsilon
}^{-} $ are bounded, symmetric and positive definite, we deduce, using the
first part of this Lemma, that%
\begin{equation*}
\begin{array}{l}
\dint\nolimits_{\Omega _{f}^{+,\varepsilon }}\left\vert v_{\varepsilon
,d}\right\vert ^{2}dx+\dint\nolimits_{\Omega _{h,f}^{-,\varepsilon
}}\left\vert v_{\varepsilon ,d}\right\vert ^{2}dx+\mu \varepsilon
^{2}\dint\nolimits_{Y_{\varepsilon }\left( \omega \right) }\left\vert \nabla
v_{\varepsilon ,s}\right\vert ^{2}dx+\gamma \dint\nolimits_{\Gamma
_{0,\varepsilon }^{+}}\left\vert \left( v_{\varepsilon ,s}\right) _{\tau
}\right\vert ^{2}dx^{\prime } \\ 
\quad +\gamma \dint\nolimits_{\Gamma _{h,\varepsilon }^{-}}\left\vert \left(
v_{\varepsilon ,s}\right) _{\tau }\right\vert ^{2}dx^{\prime }\leq C\left(
\dint\nolimits_{\Omega _{f}^{+,\varepsilon }}g^{+}\cdot v_{\varepsilon
,d}dx+\dint\nolimits_{\Omega _{h,f}^{-,\varepsilon }}g^{-}\cdot
v_{\varepsilon ,d}dx\right) ,%
\end{array}%
\end{equation*}%
hence, using Cauchy-Schwarz inequality%
\begin{equation*}
\dint\nolimits_{\Omega _{f}^{\varepsilon }\left( \omega \right) }\left\vert
v_{\varepsilon }\right\vert ^{2}dx\leq C\left( \dint\nolimits_{\Omega
_{f}^{+,\varepsilon }}\left\vert v_{\varepsilon }\right\vert ^{2}dx\right)
^{1/2}+C\left( \dint\nolimits_{\Omega _{h,f}^{-,\varepsilon }}\left\vert
v_{\varepsilon }\right\vert ^{2}dx\right) ^{1/2}\leq C\left(
\dint\nolimits_{\Omega _{f}^{\varepsilon }\left( \omega \right) }\left\vert
v_{\varepsilon }\right\vert ^{2}dx\right) ^{1/2}
\end{equation*}%
and%
\begin{equation*}
\varepsilon ^{2}\dint\nolimits_{Y_{\varepsilon }\left( \omega \right)
}\left\vert \nabla v_{\varepsilon }\right\vert ^{2}dx\leq C\left(
\dint\nolimits_{\Omega _{f}^{\varepsilon }\left( \omega \right) }\left\vert
v_{\varepsilon }\right\vert ^{2}dx\right) ^{1/2}.
\end{equation*}

Using a trace theorem in the normalized fissure $Z_{\varepsilon ,ij}\left(
\omega \right) $, we have%
\begin{equation*}
\dint\nolimits_{\partial Z_{\varepsilon ,ij}\left( \omega \right) }\psi
^{2}dz_{1}dz_{2}\leq C\dint\nolimits_{Z_{\varepsilon ,ij}\left( \omega
\right) }\left\vert \nabla \psi \right\vert ^{2}dz_{1}dz_{2},
\end{equation*}%
which ends the proof.
\end{proof}

We now deal with the pressure. We define the extension $\widetilde{p}%
_{\varepsilon ,d}$ of the pressure $p_{\varepsilon ,d}$ in the solid parts
of the porous media%
\begin{equation*}
\widetilde{p}_{\varepsilon ,d}=\left\{ 
\begin{array}{ll}
p_{\varepsilon ,d} & \text{in }\Omega _{f}^{+,\varepsilon }\cup \Omega
_{h,f}^{-,\varepsilon }, \\ 
\dfrac{1}{\left\vert Z^{1}\right\vert }\dint_{Z^{1}}p_{\varepsilon ,d}\left(
a_{\varepsilon }z+a_{\varepsilon }l\right) dz & \forall z\in \varepsilon
Z^{2}+l\varepsilon \subset \Omega ^{+}\cup \Omega _{h}^{-}\text{, }l\in 
\mathbb{Z}^{3}.%
\end{array}%
\right.
\end{equation*}

We define the zero mean-value pressures%
\begin{equation*}
\left\{ 
\begin{array}{rcll}
p_{\varepsilon }^{+} & = & \widetilde{p}_{\varepsilon ,d}-\dfrac{1}{%
\left\vert \Omega ^{+}\right\vert }\dint_{\Omega ^{+}}\widetilde{p}%
_{\varepsilon ,d}dx & \text{in\ }\Omega ^{+}, \\ 
p_{\varepsilon }^{-} & = & \widetilde{p}_{\varepsilon ,d}-\dfrac{1}{%
\left\vert \Omega _{h}^{-}\right\vert }\dint_{\Omega _{h}^{-}}\widetilde{p}%
_{\varepsilon ,d}dx & \text{in\ }\Omega _{h}^{-}, \\ 
\overline{p_{\varepsilon }} & = & p_{\varepsilon ,s}-\dfrac{1}{\left\vert
Y_{\varepsilon ,ij}\left( \omega \right) \right\vert }\dint_{Y_{\varepsilon
,ij}\left( \omega \right) }p_{\varepsilon ,s}dx & \text{in }Y_{\varepsilon
,ij}\left( \omega \right) \text{, }\left( i,j\right) \in I_{\varepsilon
}\left( \omega \right) .%
\end{array}%
\right.
\end{equation*}

\begin{lemma}
\label{estim-press}There exists a non-random constant $C$ independent of $%
\varepsilon $ such that%
\begin{equation*}
\begin{array}{rcccrcccccc}
\dint_{\Omega ^{+}}\left\vert \nabla p_{\varepsilon }^{+}\right\vert ^{2}dx
& \leq & C, &  & \dint_{\Omega _{h}^{-}}\left\vert \nabla p_{\varepsilon
}^{-}\right\vert ^{2}dx & \leq & C, &  & \dint_{Y_{\varepsilon }\left(
\omega \right) }\left( \overline{p_{\varepsilon }}\right) ^{2}dx & \leq & C,
\\ 
\dint_{\Omega ^{+}}\left( p_{\varepsilon }^{+}\right) ^{2}dx & \leq & C, & 
& \dint_{\Omega _{h}^{-}}\left( p_{\varepsilon }^{-}\right) ^{2}dx & \leq & 
C. &  &  &  & 
\end{array}%
\end{equation*}
\end{lemma}

\begin{proof}
We multiply (\ref{equ10})$_{1,2}$ by $\nabla p_{\varepsilon }$ and obtain%
\begin{equation*}
\begin{array}{rcl}
\dint_{\Omega _{f}^{+,\varepsilon }}\left\vert \nabla p_{\varepsilon
,d}\right\vert ^{2}dx & = & -\dint\nolimits_{\Omega _{f}^{+,\varepsilon
}}g^{+}\cdot \nabla p_{\varepsilon ,d}dx+\mu ^{+}\dint\nolimits_{\Omega
_{f}^{+,\varepsilon }}\left( K_{\varepsilon }^{+}\right) ^{-1}v_{\varepsilon
,d}\cdot \nabla p_{\varepsilon ,d}dx, \\ 
\dint\nolimits_{\Omega _{h,f}^{-,\varepsilon }}\left\vert \nabla
p_{\varepsilon ,d}\right\vert ^{2}dx & = & -\dint\nolimits_{\Omega
_{h,f}^{-,\varepsilon }}g^{-}\cdot \nabla p_{\varepsilon ,d}dx+\mu
^{-}\dint\nolimits_{\Omega _{h,f}^{-,\varepsilon }}\left( K_{\varepsilon
}^{-}\right) ^{-1}v_{\varepsilon ,d}\cdot \nabla p_{\varepsilon ,d}dx.%
\end{array}%
\end{equation*}

The boundary conditions (\ref{equ10})$_{4}$, the smoothness of $g^{\pm }$
and Lemma \ref{estim-vites} imply%
\begin{equation*}
\int_{\Omega ^{+}}\left\vert \nabla p_{\varepsilon }^{+}\right\vert
^{2}dx\leq C\text{ ; }\int_{\Omega _{h}^{-}}\left\vert \nabla p_{\varepsilon
}^{-}\right\vert ^{2}dx\leq C.
\end{equation*}

Using Poincar\'{e}-Wirtinger' inequality, we have%
\begin{equation*}
\begin{array}{rcccc}
\dint_{\Omega ^{+}}\left( p_{\varepsilon }^{+}\right) ^{2}dx & \leq & 
C\left( \Omega ^{+}\right) \dint_{\Omega ^{+}}\left\vert \nabla
p_{\varepsilon }\right\vert ^{2}dx & \leq & C, \\ 
\dint_{\Omega _{h}^{-}}\left( p_{\varepsilon }^{-}\right) ^{2}dx & \leq & 
C\left( \Omega _{h}^{-}\right) \dint_{\Omega _{h}^{-}}\left\vert \nabla
p_{\varepsilon }\right\vert ^{2}dx & \leq & C.%
\end{array}%
\end{equation*}

In order to get estimates on $\overline{p_{\varepsilon }}$, we consider the
problem%
\begin{equation*}
\left\{ 
\begin{array}{rcll}
\func{div}_{z}\left( \Phi _{\varepsilon }^{0}\right) & = & \overline{%
p_{\varepsilon }}\left( \varepsilon z_{1},\varepsilon z_{2},x_{3}\right) & 
\text{in }Z_{\varepsilon ,ij}\left( \omega \right) , \\ 
\Phi _{\varepsilon }^{0} & = & 0 & \text{on }\partial Z_{\varepsilon
,ij}\left( \omega \right) ,%
\end{array}%
\right.
\end{equation*}%
where the fissure $Z_{\varepsilon ,ij}\left( \omega \right) $ is defined in (%
\ref{equ14}). This problem has a unique solution $\Phi _{\varepsilon }^{0}$
in the space $\left\{ \left( -\Delta \right) ^{-1}\nabla w\mid w\in \mathbf{L%
}^{2}\left( Z_{\varepsilon ,ij}\left( \omega \right) \right) \right\} $ (see 
\cite{Tem} for example), such that%
\begin{equation}
\dint\nolimits_{Z_{\varepsilon ,ij}\left( \omega \right) }\left\vert \nabla
\Phi _{\varepsilon }^{0}\right\vert ^{2}dzdx_{3}\leq
C\dint\nolimits_{Z_{\varepsilon ,ij}\left( \omega \right) }\left( \overline{%
p_{\varepsilon }}\right) ^{2}dzdx_{3},  \label{equ18}
\end{equation}%
where, due to the hypothesis (\ref{equ5}), the constant $C$ is non-random
and independent of $\varepsilon $. Let us define%
\begin{equation*}
\Phi _{\varepsilon }\left( x\right) =\left( 
\begin{array}{c}
\varepsilon \left( \Phi _{\varepsilon }^{0}\right) _{1}\left( \left(
x_{1}-i\varepsilon \right) /\varepsilon ,\left( x_{2}-j\varepsilon \right)
/\varepsilon ,x_{3}\right) \\ 
\varepsilon \left( \Phi _{\varepsilon }^{0}\right) _{2}\left( \left(
x_{1}-i\varepsilon \right) /\varepsilon ,\left( x_{2}-j\varepsilon \right)
/\varepsilon ,x_{3}\right) \\ 
\left( \Phi _{\varepsilon }^{0}\right) _{3}\left( \left( x_{1}-i\varepsilon
\right) /\varepsilon ,\left( x_{2}-j\varepsilon \right) /\varepsilon
,x_{3}\right)%
\end{array}%
\right) .
\end{equation*}

Then%
\begin{equation*}
\left\{ 
\begin{array}{rcll}
\func{div}\left( \Phi _{\varepsilon }\right) \left( x\right) & = & \func{div}%
_{z}\left( \Phi _{\varepsilon }^{0}\right) \left( x\right) =\overline{%
p_{\varepsilon }}\left( x\right) & \text{in }Y_{\varepsilon ,ij}\left(
\omega \right) , \\ 
\Phi _{\varepsilon } & = & 0 & \text{on }\partial Y_{\varepsilon ,ij}\left(
\omega \right)%
\end{array}%
\right.
\end{equation*}%
and, thanks to (\ref{equ18})%
\begin{equation}
\dint\nolimits_{Y_{\varepsilon }\left( \omega \right) }\left\vert \nabla
\Phi _{\varepsilon }\right\vert ^{2}dx\leq \dint\nolimits_{Z_{\varepsilon
,ij}\left( \omega \right) }\left\vert \nabla \Phi _{\varepsilon
}^{0}\right\vert ^{2}dzdx_{3}\leq C\dint\nolimits_{Z_{\varepsilon ,ij}\left(
\omega \right) }\left( \overline{p_{\varepsilon }}\right) ^{2}dzdx_{3}=\frac{%
C}{\varepsilon ^{2}}\dint\nolimits_{Y_{\varepsilon }\left( \omega \right)
}\left( \overline{p_{\varepsilon }}\left( x\right) \right) ^{2}dx.
\label{equ19}
\end{equation}

Multiplying (\ref{equ10})$_{5}$ by $\Phi _{\varepsilon }$ in $Y_{\varepsilon
}\left( \omega \right) $, we get%
\begin{equation*}
\mu \varepsilon ^{2}\dint\nolimits_{Y_{\varepsilon }\left( \omega \right)
}\nabla v_{\varepsilon }\cdot \nabla \Phi _{\varepsilon
}dx-\dint\nolimits_{Y_{\varepsilon }\left( \omega \right) }\left( \overline{%
p_{\varepsilon }}\left( x\right) \right) ^{2}dx=0.
\end{equation*}

Using the inequality (\ref{equ19}), we have%
\begin{equation*}
\dint\nolimits_{Y_{\varepsilon }\left( \omega \right) }\left( \overline{%
p_{\varepsilon }}\left( x\right) \right) ^{2}dx\leq C\left(
\dint\nolimits_{Y_{\varepsilon }\left( \omega \right) }\left( \overline{%
p_{\varepsilon }}\left( x\right) \right) ^{2}dx\right) ^{1/2}\left(
\varepsilon ^{2}\dint\nolimits_{Y_{\varepsilon }\left( \omega \right)
}\left\vert \nabla v_{\varepsilon }\right\vert ^{2}dx\right) ^{1/2}.
\end{equation*}

Thanks to Lemma \ref{estim-vites}, we derive the last estimate in $%
Y_{\varepsilon }\left( \omega \right) $.
\end{proof}

\subsection{Convergence}

Observe the following result.

\begin{lemma}
\label{conv-vites-fiss}

\begin{enumerate}
\item For every $\varphi \in C_{0}^{1}\left( Y_{h}\right) $, we have%
\begin{equation*}
\underset{\varepsilon \rightarrow 0}{\lim }\underset{\left( i,j\right) \in
I_{\varepsilon }\left( \omega \right) }{\dsum }\dint\nolimits_{Y_{%
\varepsilon ,ij}\left( \omega \right) }\varphi \left(
x_{1},x_{2},x_{3}\right) dx=h\left\langle q^{2}\right\rangle
\dint\nolimits_{\Sigma }\varphi \left( x^{\prime },0\right) dx^{\prime }%
\text{, almost surely.}
\end{equation*}

\item Let $\left( w_{\varepsilon }\right) _{\varepsilon }$ be a sequence
such that $\sup_{\varepsilon }\int\nolimits_{Y_{\varepsilon }\left( \omega
\right) }\left( w_{\varepsilon }\right) ^{2}dx<+\infty $. There exists a
subsequence, still denoted in the same way, such that%
\begin{equation*}
\underset{\varepsilon \rightarrow 0}{\lim }\dint\nolimits_{Y_{\varepsilon
}\left( \omega \right) }w_{\varepsilon }\varphi dx=h\left\langle q^{2}\left(
0\right) \right\rangle \dint\nolimits_{\Sigma }w\left( x^{\prime },0\right)
\varphi \left( x^{\prime },0\right) dx^{\prime }\text{, }\forall \varphi \in
C_{0}\left( 
%TCIMACRO{\U{211d} }%
%BeginExpansion
\mathbb{R}
%EndExpansion
^{3}\right) \text{, almost surely.}
\end{equation*}
\end{enumerate}
\end{lemma}

\begin{proof}
1. Define $\xi _{i,\varepsilon }\left( y\right) =\varepsilon i-\varepsilon
r_{i}\left( \varepsilon ^{-\theta }t\right) -y_{1}q_{i}\left( \varepsilon
^{-\theta }t\right) $ and $\xi _{j,\varepsilon }\left( y\right) =\varepsilon
j-\varepsilon r_{j}\left( \varepsilon ^{-\theta }t\right) -y_{1}q_{j}\left(
\varepsilon ^{-\theta }t\right) $ where $r_{i}\left( \varepsilon ^{-\theta
}t\right) =r\left( \varepsilon ^{-\theta }t+\beta _{i}\left( \omega \right)
,\omega \right) $ and $q_{i}\left( \varepsilon ^{-\theta }t\right) =q\left(
\varepsilon ^{-\theta }t+\alpha _{i}\left( \omega \right) ,\omega \right) $%
.\ According to the properties of the above-defined curvilinear coordinates,
we have%
\begin{equation*}
\begin{array}{l}
\underset{\varepsilon \rightarrow 0}{\lim }\underset{\left( i,j\right) \in
I_{\varepsilon }\left( \omega \right) }{\dsum }\dint\nolimits_{Y_{%
\varepsilon ,ij}\left( \omega \right) }\varphi \left(
x_{1},x_{2},x_{3}\right) dx \\ 
\quad 
\begin{array}{ll}
= & \underset{\varepsilon \rightarrow 0}{\lim }\underset{\left( i,j\right)
\in I_{\varepsilon }\left( \omega \right) }{\dsum }\dint_{-\varepsilon
/2}^{\varepsilon /2}\dint_{-\varepsilon /2}^{\varepsilon
/2}\dint\nolimits_{0}^{h}\varphi \left( \xi _{i,\varepsilon }\left( y\right)
,\xi _{j,\varepsilon }\left( y\right) ,-t\right) q_{i}\left( \varepsilon
^{-\theta }t\right) q_{j}\left( \varepsilon ^{-\theta }t\right) dydt \\ 
= & \underset{\varepsilon \rightarrow 0}{\lim }\underset{\left( i,j\right)
\in I_{\varepsilon }\left( \omega \right) }{\dsum }\varepsilon
^{2}\dint\nolimits_{0}^{h}\varphi \left( \varepsilon i,\varepsilon
j,-t\right) q^{2}\left( \varepsilon ^{-\theta }t\right) dt \\ 
= & \underset{\varepsilon \rightarrow 0}{\lim }\dfrac{h}{\varepsilon
^{-\theta }h}\dint\nolimits_{0}^{\varepsilon ^{-\theta }h}\left(
\dint\nolimits_{\Sigma }\varphi \left( x_{1},x_{2},-\varepsilon ^{\theta
}s\right) dx_{1}dx_{2}\right) q^{2}\left( s\right) ds \\ 
= & h\left\langle q^{2}\right\rangle \dint\nolimits_{\Sigma }\varphi \left(
x^{\prime },0\right) dx^{\prime },%
\end{array}%
\end{array}%
\end{equation*}%
where we have used the ergodicity result (\ref{Ergo}).

2. The sequence of measures $\left( \nu _{\varepsilon }\right) _{\varepsilon
}$, with $\nu _{\varepsilon }=\mathbf{1}_{Y_{\varepsilon }\left( \omega
\right) }dx$, $\mathbf{1}_{A}$ being the characteristic function of the set $%
A$, thus converges in the weak sense of measures, when $\varepsilon $ goes
to $0$, to the measure $\nu =h\left\langle q^{2}\right\rangle \mathbf{1}%
_{\Sigma }\left( x^{\prime }\right) dx^{\prime }$. Using the hypothesis on $%
\left( w_{\varepsilon }\right) _{\varepsilon }$, we deduce that the sequence
of measures $\left( v_{\varepsilon }w_{\varepsilon }\right) _{\varepsilon }$
has bounded variation. Up to some subsequence, the sequence $\left(
v_{\varepsilon }w_{\varepsilon }\right) _{\varepsilon }$ thus converges to
some $\chi _{0}$ in the weak sense of measures. For every $\varphi \in
C_{0}\left( 
%TCIMACRO{\U{211d} }%
%BeginExpansion
\mathbb{R}
%EndExpansion
^{3}\right) $, one has, thanks to Fenchel's inequality%
\begin{equation*}
\dint\nolimits_{%
%TCIMACRO{\U{211d} }%
%BeginExpansion
\mathbb{R}
%EndExpansion
^{3}}\left( w_{\varepsilon }\right) ^{2}d\nu _{\varepsilon }\geq
2\dint\nolimits_{%
%TCIMACRO{\U{211d} }%
%BeginExpansion
\mathbb{R}
%EndExpansion
^{3}}w_{\varepsilon }\varphi d\nu _{\varepsilon }-\dint\nolimits_{%
%TCIMACRO{\U{211d} }%
%BeginExpansion
\mathbb{R}
%EndExpansion
^{3}}\varphi ^{2}d\nu _{\varepsilon }.
\end{equation*}

Passing to the limit, we get%
\begin{equation*}
+\infty >\underset{\varepsilon \rightarrow 0}{\lim \inf }\dint\nolimits_{%
%TCIMACRO{\U{211d} }%
%BeginExpansion
\mathbb{R}
%EndExpansion
^{3}}\left( w_{\varepsilon }\right) ^{2}d\nu _{\varepsilon }\geq
2\left\langle \chi _{0},\varphi \right\rangle -\dint\nolimits_{%
%TCIMACRO{\U{211d} }%
%BeginExpansion
\mathbb{R}
%EndExpansion
^{2}}\varphi ^{2}\left( x^{\prime },0\right) d\nu .
\end{equation*}

Thus%
\begin{equation*}
\sup \left\{ \left\langle \chi _{0},\varphi \right\rangle \mid \varphi \in
C_{0}\left( 
%TCIMACRO{\U{211d} }%
%BeginExpansion
\mathbb{R}
%EndExpansion
^{3}\right) \text{, }\dint\nolimits_{%
%TCIMACRO{\U{211d} }%
%BeginExpansion
\mathbb{R}
%EndExpansion
^{2}}\varphi ^{2}\left( x^{\prime },0\right) d\nu \leq 1\right\} <+\infty .
\end{equation*}

Using Riesz' representation theorem, we can identify $\chi _{0}$ with $w\nu $%
, for some $w\in L^{2}\left( 
%TCIMACRO{\U{211d} }%
%BeginExpansion
\mathbb{R}
%EndExpansion
^{2}\right) $.
\end{proof}

\begin{remark}
\label{Remark1}

\begin{enumerate}
\item From Lemma \ref{estim-vites}, using the above result, we deduce that,
up to some subsequence%
\begin{equation*}
\underset{\varepsilon \rightarrow 0}{\lim }\dint\nolimits_{Y_{\varepsilon
}\left( \omega \right) }v_{\varepsilon }\cdot \Phi dx=h\left\langle
q^{2}\left( 0\right) \right\rangle \dint\nolimits_{\Sigma }v_{f}\cdot \Phi
dx^{\prime }\text{, }\forall \Phi \in \mathbf{C}_{0}\left( \Sigma ;%
%TCIMACRO{\U{211d} }%
%BeginExpansion
\mathbb{R}
%EndExpansion
^{3}\right) ,
\end{equation*}%
almost surely, and from Lemma \ref{estim-press}, we deduce the existence of $%
\pi _{0}\in L^{2}\left( \Sigma \right) $ such that, up to some subsequence%
\begin{equation*}
\underset{\varepsilon \rightarrow 0}{\lim }\dint\nolimits_{Y_{\varepsilon
}\left( \omega \right) }\overline{p_{\varepsilon }}\varphi
dx=\dint\nolimits_{\Sigma }\pi _{0}\varphi d\nu =h\left\langle
q^{2}\right\rangle \dint\nolimits_{\Sigma }\pi _{0}\left( x^{\prime }\right)
\varphi \left( x^{\prime }\right) dx^{\prime }\text{, }\forall \varphi \in
C_{0}\left( \Sigma \right) ,
\end{equation*}%
almost surely, where $x^{\prime }=\left( x_{1},x_{2}\right) $.

In the rest of the paper, we will no more indicate this almost surely
convergence where there is no doubt.

\item From Lemma \ref{estim-vites} and the above computations, we deduce the
existence of $v_{0}\in \mathbf{L}^{2}\left( \Omega ;%
%TCIMACRO{\U{211d} }%
%BeginExpansion
\mathbb{R}
%EndExpansion
^{3}\right) $ such that, up to some subsequence%
\begin{equation*}
\begin{array}{rcll}
v_{\varepsilon }\mid _{\Omega _{f}^{+,\varepsilon }} & \underset{\varepsilon
\rightarrow 0}{\rightarrow } & v_{0}\mid _{\Omega ^{+}}=:v_{0,d}^{+} & \text{%
w-}\mathbf{L}^{2}\left( \Omega ^{+};%
%TCIMACRO{\U{211d} }%
%BeginExpansion
\mathbb{R}
%EndExpansion
^{3}\right) , \\ 
v_{\varepsilon }\mid _{\Omega _{h,f}^{-,\varepsilon }} & \underset{%
\varepsilon \rightarrow 0}{\rightarrow } & v_{0}\mid _{\Omega
_{h}^{-}}=:v_{0,d}^{-} & \text{w-}\mathbf{L}^{2}\left( \Omega _{h}^{-};%
%TCIMACRO{\U{211d} }%
%BeginExpansion
\mathbb{R}
%EndExpansion
^{3}\right) , \\ 
\underset{\varepsilon \rightarrow 0}{\lim }\dint\nolimits_{Y_{\varepsilon
}\left( \omega \right) }v_{\varepsilon }\cdot \Phi dx & = & h\left\langle
q^{2}\right\rangle \dint\nolimits_{\Sigma }v_{0}\cdot \Phi dx^{\prime } & 
\forall \Phi \in \mathbf{C}_{0}\left( \Sigma ;%
%TCIMACRO{\U{211d} }%
%BeginExpansion
\mathbb{R}
%EndExpansion
^{3}\right) .%
\end{array}%
\end{equation*}

\item We set $v_{0}\mid _{\Sigma }=:v_{0,f}$. For every $\varphi \in
C_{0}^{1}\left( \Sigma \right) $, one has%
\begin{equation*}
\begin{array}{ccl}
0=\underset{\varepsilon \rightarrow 0}{\lim }\underset{\left( i,j\right) \in
I_{\varepsilon }\left( \omega \right) }{\dsum }\dint\nolimits_{Y_{%
\varepsilon ,ij}\left( \omega \right) }\func{div}\left( v_{\varepsilon
}\right) \varphi dx & = & -\underset{\varepsilon \rightarrow 0}{\lim }%
\underset{\left( i,j\right) \in I_{\varepsilon }\left( \omega \right) }{%
\dsum }\dint\nolimits_{Y_{\varepsilon ,ij}\left( \omega \right)
}v_{\varepsilon }\cdot \nabla \varphi dx \\ 
& = & -h\left\langle q^{2}\right\rangle \dint\nolimits_{\Sigma }v_{0,f}\cdot
\nabla \varphi dx^{\prime } \\ 
& = & h\left\langle q^{2}\right\rangle \dint\nolimits_{\Sigma }\func{div}%
\left( v_{0,f}\right) \varphi dx^{\prime },%
\end{array}%
\end{equation*}%
thanks to the estimates of Lemma \ref{estim-vites}. Thus $\func{div}\left(
v_{0,f}\right) =0$ in $\Sigma $.

\item It is easily seen that $\func{div}\left( v_{0,d}^{+}\right) =0$ in $%
\Omega ^{+}$ and $\func{div}\left( v_{0,d}^{-}\right) =0$ in $\Omega _{h}^{-}
$. For every $\varphi \in C^{1}\left( \Omega ^{+}\right) $, one has%
\begin{equation*}
0=\underset{\varepsilon \rightarrow 0}{\lim }\dint\nolimits_{\Omega
_{f}^{+,\varepsilon }}\func{div}\left( v_{\varepsilon }\right) \varphi dx=-%
\underset{\varepsilon \rightarrow 0}{\lim }\dint\nolimits_{\Omega
_{f}^{+,\varepsilon }}v_{\varepsilon }\cdot \nabla \varphi dx+\underset{%
\varepsilon \rightarrow 0}{\lim }\dint\nolimits_{\Gamma _{0,\varepsilon
}^{+}}\left( v_{\varepsilon }\right) _{3}\varphi dx^{\prime }.
\end{equation*}

Observe that%
\begin{equation*}
\underset{\varepsilon \rightarrow 0}{\lim }\dint\nolimits_{\Gamma
_{0,\varepsilon }^{+}}\left( v_{\varepsilon }\right) _{3}\varphi dx^{\prime
}=\left\langle q^{2}\left( 0\right) \right\rangle \dint\nolimits_{\Sigma
}\left( v_{0,f}\right) _{3}\varphi dx^{\prime },
\end{equation*}%
whence $\left( v_{0,d}^{+}\right) _{3}\mid _{\Sigma \times \left\{ 0\right\}
}=\left\langle q^{2}\left( 0\right) \right\rangle \left( v_{0,f}\right) _{3}$%
. In a similar way, but working in $\Omega _{h,f}^{-,\varepsilon }$, instead
of $\Omega _{f}^{+,\varepsilon }$, we have $\left( v_{0,d}^{-}\right)
_{3}\mid _{\Sigma \times \left\{ -h\right\} }=\left\langle q^{2}\left(
0\right) \right\rangle \left( v_{0,f}\right) _{3}$.

\item From Lemmas \ref{estim-press} and \ref{conv-vites-fiss}, we get, up to
some subsequence%
\begin{equation}
\begin{array}{rcll}
p_{\varepsilon }^{+} & \underset{\varepsilon \rightarrow 0}{\rightarrow } & 
p_{0}^{+} & \text{s-}L^{2}\left( \Omega ^{+}\right) , \\ 
p_{\varepsilon }^{-} & \underset{\varepsilon \rightarrow 0}{\rightarrow } & 
p_{0}^{-} & \text{s-}L^{2}\left( \Omega _{h}^{-}\right) , \\ 
\underset{\varepsilon \rightarrow 0}{\lim }\dint\nolimits_{Y_{\varepsilon
}\left( \omega \right) }\overline{p_{\varepsilon }}\varphi dx & = & 
h\left\langle q^{2}\left( 0\right) \right\rangle \dint\nolimits_{\Sigma }\pi
_{0}\varphi dx^{\prime } & \forall \varphi \in C_{0}\left( \Sigma \right) .%
\end{array}
\label{convp}
\end{equation}
\end{enumerate}
\end{remark}

We now set%
\begin{equation*}
\mathbf{V}_{0}=\left\{ 
\begin{array}{l}
v\in \mathbf{L}^{2}\left( \Omega ^{+}\cup \Omega _{h}^{-};\mathbb{R}%
^{3}\right) \mid \func{div}\left( v\right) =0\text{ in }\Omega ^{+}\cup
\Omega _{h}^{-}\text{,} \\ 
\quad v\cdot n=0\text{ on }\Gamma ^{+}\cup \Gamma ^{-}\text{, }v_{3}\mid
_{\Sigma \times \left\{ 0\right\} }=\left\langle q^{2}\left( 0\right)
\right\rangle \left( v_{,f}\right) _{3}=v_{3}\mid _{\Sigma \times \left\{
-h\right\} }%
\end{array}%
\right\} .
\end{equation*}

Every function $v\in \mathbf{V}_{0}$ can be extended in a function of $%
\mathbf{L}^{2}\left( Y_{\varepsilon }\left( \omega \right) ;\mathbb{R}%
^{3}\right) $ independent of $x_{3}$ in $Y_{\varepsilon }\left( \omega
\right) $.

We define the appropriate notion of convergence for the problem (\ref{equ10}%
).

\begin{definition}
\label{topology}A sequence $\left( V_{\varepsilon }\right) _{\varepsilon }$,
with $V_{\varepsilon }\in \mathbf{V}_{\varepsilon }$ for every $\varepsilon $%
, $\tau _{0}$-converges to some $V\in \mathbf{V}_{0}$ if%
\begin{equation*}
\left\{ 
\begin{array}{rcll}
V_{\varepsilon }\mid _{\Omega _{f}^{+,\varepsilon }} & \underset{\varepsilon
\rightarrow 0}{\rightharpoonup } & V\mid _{\Omega ^{+}}=:V_{d}^{+} & \text{w-%
}\mathbf{L}^{2}\left( \Omega ^{+};%
%TCIMACRO{\U{211d} }%
%BeginExpansion
\mathbb{R}
%EndExpansion
^{3}\right) , \\ 
V_{\varepsilon }\mid _{\Omega _{h,f}^{-,\varepsilon }} & \underset{%
\varepsilon \rightarrow 0}{\rightharpoonup } & V\mid _{\Omega
_{h}^{-}}=:V_{d}^{-} & \text{w-}\mathbf{L}^{2}\left( \Omega _{h}^{-};%
%TCIMACRO{\U{211d} }%
%BeginExpansion
\mathbb{R}
%EndExpansion
^{3}\right) , \\ 
\underset{\varepsilon \rightarrow 0}{\lim }\dint\nolimits_{Y_{\varepsilon
}\left( \omega \right) }V_{\varepsilon }\cdot \Phi dx & = & h\left\langle
q^{2}\left( 0\right) \right\rangle \dint\nolimits_{\Sigma }V_{f}\cdot \Phi
dx^{\prime } & \forall \Phi \in \mathbf{C}_{0}\left( \Sigma ;%
%TCIMACRO{\U{211d} }%
%BeginExpansion
\mathbb{R}
%EndExpansion
^{3}\right) ,%
\end{array}%
\right.
\end{equation*}%
with $V_{f}:=V\mid _{\Sigma }$.
\end{definition}

We define the functional $F_{\varepsilon }$ on $\mathbf{L}^{2}\left( \Omega
_{f}^{\varepsilon }\left( \omega \right) ;%
%TCIMACRO{\U{211d} }%
%BeginExpansion
\mathbb{R}
%EndExpansion
^{3}\right) $ through%
\begin{equation*}
F_{\varepsilon }\left( v\right) =\left\{ 
\begin{array}{l}
\mu ^{+}\dint\nolimits_{\Omega _{f}^{+,\varepsilon }}\left( K_{\varepsilon
}^{+}\right) ^{-1}v\cdot vdx+\mu ^{-}\dint\nolimits_{\Omega
_{h,f}^{-,\varepsilon }}\left( K_{\varepsilon }^{-}\right) ^{-1}v\cdot
vdx+\mu \varepsilon ^{2}\dint\nolimits_{Y_{\varepsilon }\left( \omega
\right) }\left\vert \nabla v\right\vert ^{2}dx \\ 
\quad +\gamma \dint\nolimits_{\Gamma _{0,\varepsilon }^{+}}\left(
K_{\varepsilon }^{+}\right) ^{-1/2}v_{\tau }\cdot v_{\tau }dx^{\prime
}+\gamma \dint\nolimits_{\Gamma _{h,\varepsilon }^{-}}\left( K_{\varepsilon
}^{-}\right) ^{-1/2}v_{\tau }\cdot v_{\tau }dx^{\prime } \\ 
\hfill \text{if }v\in \mathbf{V}_{\varepsilon }, \\ 
+\infty \hfill \text{otherwise.}%
\end{array}%
\right.
\end{equation*}

In order to describe the asymptotic behaviour of this functional, we
consider the $Z$-periodic solution $\Phi _{k}^{\pm }$ of the local Darcy
systems%
\begin{equation*}
\left\{ 
\begin{array}{rlll}
\left( K^{\pm }\right) ^{-1}\Phi _{k}^{\pm }-\nabla \pi _{k}^{\pm } & = & 
e_{k} & \text{in }Z^{1}\text{, }k=1,2,3, \\ 
\func{div}\left( \Phi _{k}^{\pm }\right) & = & 0 & \text{in }Z^{1}, \\ 
\Phi _{k}^{\pm }\cdot n & = & 0 & \text{on }S,%
\end{array}%
\right.
\end{equation*}%
where $\left( e_{1},e_{2},e_{3}\right) $ is the canonical basis of $\mathbb{R%
}^{3}$. We then consider the Stokes system

\begin{equation}
\left\{ 
\begin{array}{rlll}
-\Delta \eta _{k}+\nabla \xi _{k} & = & e_{k} & \text{in }Z^{\prime }=\left(
-1/2,1/2\right) ^{2}\text{, }k=1,2, \\ 
\func{div}\left( \eta _{k}\right) & = & 0 & \text{in }Z^{\prime }, \\ 
\eta _{k} & = & 0 & \text{on }\partial Z^{\prime },%
\end{array}%
\right.  \label{equ23}
\end{equation}%
where $\left( e_{1},e_{2}\right) $ is the canonical basis of $\mathbb{R}^{2}$%
, and the local scalar problem

\begin{equation}
\left\{ 
\begin{array}{rlll}
-\Delta \eta _{0} & = & 1 & \text{in }Z^{\prime }=\left( -1/2,1/2\right)
^{2}, \\ 
\eta _{0} & = & 0 & \text{on }\partial Z^{\prime }.%
\end{array}%
\right.  \label{equ24}
\end{equation}

We define the $3\times 3$ matrices $\widehat{K}^{+}$, $\widehat{K}^{-}$, the 
$2\times 2$ matrix $K_{f}$ and the constant $k_{0}$ through%
\begin{equation}
\left\{ 
\begin{array}{rlll}
\widehat{K}_{ml}^{+} & = & \dint\nolimits_{Z^{1}}\left( \Phi _{m}^{+}\right)
_{l}dz & m,l=1,2,3, \\ 
\widehat{K}_{ml}^{-} & = & \dint\nolimits_{Z^{1}}\left( \Phi _{m}^{-}\right)
_{l}dz & m,l=1,2,3, \\ 
\left( K_{f}\right) _{ml} & = & \dint\nolimits_{Z^{\prime }}\left( \eta
_{m}\right) _{l}dz & m,l=1,2, \\ 
k_{0} & = & \dint\nolimits_{Z^{\prime }}\left\vert \nabla \eta
_{0}\right\vert ^{2}dz=\dint\nolimits_{Z^{\prime }}\eta _{0}dz. & 
\end{array}%
\right.  \label{equ25}
\end{equation}

One can prove that the matrices $\widehat{K}^{\pm }$ and $K_{f}$ are
symmetric and positive definite (see \cite{San}). Our main result of this
part reads as follows.

\begin{theorem}
\label{Theorem1}Suppose that $r$ is also a stationary random process. The
sequence $\left( F_{\varepsilon }\right) _{\varepsilon }$ $\Gamma $%
-converges in the topology $\tau _{0}$ to the functional $F_{0}$ defined
through%
\begin{equation*}
F_{0}\left( v\right) =\left\{ 
\begin{array}{l}
\mu ^{+}\dint\nolimits_{\Omega ^{+}}\left( \widehat{K}^{+}\right)
^{-1}v_{d}^{+}\cdot v_{d}^{+}dx+\mu ^{-}\dint\nolimits_{\Omega
_{h}^{-}}\left( \widehat{K}^{-}\right) ^{-1}v_{d}^{-}\cdot v_{d}^{-}dx \\ 
\quad +\mu \dfrac{h\left\langle q^{2}\right\rangle }{\left\langle
q\right\rangle ^{2}}\dint\nolimits_{\Sigma }\left( K_{f}\right) ^{-1}\left(
v_{f}\right) _{\tau }\cdot \left( v_{f}\right) _{\tau }dx^{\prime }+\mu 
\dfrac{h\left\langle q^{2}\right\rangle \left\langle 1/q^{2}\right\rangle }{%
k_{0}}\dint\nolimits_{\Sigma }\left( \left( v_{f}\right) _{3}\right)
^{2}dx^{\prime } \\ 
\quad +\left\langle q^{2}\right\rangle \gamma \dint\nolimits_{\Sigma }\left(
\left( K^{\ast +}\right) ^{-1/2}+\left( K^{\ast -}\right) ^{-1/2}\right)
\left( v_{f}\right) _{\tau }\cdot \left( v_{f}\right) _{\tau }dx^{\prime }
\\ 
\hfill \text{if }v\in \mathbf{V}_{0}, \\ 
+\infty \hfill \text{ otherwise,}%
\end{array}%
\right. 
\end{equation*}

where $K^{\ast \pm }=\int_{\left\langle a^{-}\left( 0\right) \right\rangle
}^{\left\langle a^{+}\left( 0\right) \right\rangle }\int_{\left\langle
a^{-}\left( 0\right) \right\rangle }^{\left\langle a^{+}\left( 0\right)
\right\rangle }K^{\pm }\left( z_{1},z_{2},0\right) dz_{1}dz_{2}$.
\end{theorem}

\begin{proof}
For the definition and the properties of the $\Gamma $-convergence, we refer
to \cite{Att} and \cite{Dal}.

\textbf{Upper }$\mathbf{\Gamma }$\textbf{-limit.} Choose any smooth $v\in 
\mathbf{C}^{1}\left( \overline{\Omega ^{+}\cup \Omega _{h}^{-}};\mathbb{R}%
^{3}\right) \cap \mathbf{V}_{0}$. We define $v\mid _{\Omega ^{+}}=:v_{d}^{+}$%
, $v\mid _{\Omega _{h}^{-}}=:v_{d}^{-}$, $v\mid _{\Sigma }=:v_{f}$ and build
in $Y_{\varepsilon ,ij}\left( \omega \right) $%
\begin{equation}
\left\{ 
\begin{array}{rcl}
\left( v_{\varepsilon ,f}^{0}\right) _{\tau } & = & \dfrac{\left( \left(
K_{f}\right) ^{-1}v_{f}\right) _{k}\left( i\varepsilon ,j\varepsilon
,0\right) }{h\left\langle q\right\rangle }\dint\nolimits_{-h}^{0}\eta
_{\varepsilon ,k,ij}\left( x\right) dx_{3}, \\ 
\left( v_{\varepsilon ,f}^{0}\right) _{3} & = & \dfrac{\left( v_{f}\right)
_{3}\left( i\varepsilon ,j\varepsilon ,0\right) }{hk_{0}}\dint%
\nolimits_{-h}^{0}\eta _{\varepsilon ,0,ij}\left( x\right) dx_{3},%
\end{array}%
\right.  \label{equ26}
\end{equation}%
where%
\begin{equation*}
\begin{array}{ccl}
\eta _{\varepsilon ,k,ij}\left( x\right) & = & \left( 
\begin{array}{c}
q_{i}\left( -\varepsilon ^{-\theta }x_{3}\right) \left( \eta _{k}\right)
_{1}\left( z\left( x_{1},x_{2},x_{3}\right) \right) \\ 
q_{j}\left( -\varepsilon ^{-\theta }x_{3}\right) \left( \eta _{k}\right)
_{2}\left( z\left( x_{1},x_{2},x_{3}\right) \right)%
\end{array}%
\right) , \\ 
\eta _{\varepsilon ,0,ij}\left( x\right) & = & \eta _{0}\left( z\left(
x_{1},x_{2},x_{3}\right) \right) ,%
\end{array}%
\end{equation*}%
with%
\begin{equation}
\begin{array}{rll}
z\left( x_{1},x_{2},x_{3}\right) & = & \left( 
\begin{array}{c}
\dfrac{x_{1}-i\varepsilon -\varepsilon \left( a_{i}^{-}\left( -\varepsilon
^{-\theta }x_{3}\right) +a_{i}^{+}\left( -\varepsilon ^{-\theta
}x_{3}\right) \right) /2}{\varepsilon q_{i}\left( -\varepsilon ^{-\theta
}x_{3}\right) } \\ 
\dfrac{x_{2}-j\varepsilon -\varepsilon \left( a_{j}^{-}\left( -\varepsilon
^{-\theta }x_{3}\right) +a_{j}^{+}\left( -\varepsilon ^{-\theta
}x_{3}\right) \right) /2}{\varepsilon q_{j}\left( -\varepsilon ^{-\theta
}x_{3}\right) }%
\end{array}%
\right) , \\ 
q_{i}\left( s\right) & = & q\left( s+\alpha _{i}\left( \omega \right)
,\omega \right) ,%
\end{array}
\label{zq}
\end{equation}%
$\eta _{k}$ being the solution of (\ref{equ23}), $\eta _{0}$ the solution of
(\ref{equ24}) and $K_{f}$ and $k_{0}$ being defined in (\ref{equ25}). We
then define the test-function $v_{\varepsilon }^{0}$ through%
\begin{equation}
\left\{ 
\begin{array}{rcll}
v_{\varepsilon }^{0} & = & \Phi _{j}^{+}\left( \dfrac{x}{\varepsilon }%
\right) \left( \left( \widehat{K}^{+}\right) ^{-1}v_{d}^{+}\right) _{j} & 
\text{in }\Omega _{f}^{+,\varepsilon }, \\ 
v_{\varepsilon }^{0} & = & \Phi _{j}^{-}\left( \dfrac{x}{\varepsilon }%
\right) \left( \left( \widehat{K}^{-}\right) ^{-1}v_{d}^{-}\right) _{j} & 
\text{in }\Omega _{h,f}^{-,\varepsilon }, \\ 
v_{\varepsilon }^{0} & = & v_{\varepsilon ,f}^{0} & \text{in }Y_{\varepsilon
}\left( \omega \right) .%
\end{array}%
\right.  \label{fon-test}
\end{equation}

One deduces from this construction that $v_{\varepsilon }^{0}\mid _{\Omega
^{+}}\in \mathbf{L}^{2}\left( \Omega _{f}^{+,\varepsilon };\mathbb{R}%
^{3}\right) $, $v_{\varepsilon }^{0}\mid _{\Omega _{h}^{-}}\in \mathbf{L}%
^{2}\left( \Omega _{h,f}^{-,\varepsilon };\mathbb{R}^{3}\right) $, $%
v_{\varepsilon }^{0}\mid _{Y_{\varepsilon }\left( \omega \right) }\in 
\mathbf{H}^{1}\left( Y_{\varepsilon }\left( \omega \right) ;\mathbb{R}%
^{3}\right) $, $v_{\varepsilon }^{0}=0$ on $\Lambda _{\varepsilon }\left(
\omega \right) $, $\func{div}\left( v_{\varepsilon }^{0}\right) =0$ in $%
\Omega _{f}^{+,\varepsilon }\cup \Omega _{h,f}^{-,\varepsilon }$, $%
v_{\varepsilon }^{0}\cdot n=0$ on $\partial \Omega _{f}^{\varepsilon }\left(
\omega \right) $ and%
\begin{equation*}
\func{div}\left( v_{\varepsilon }^{0}\right) =\dfrac{\left( \left(
K_{f}\right) ^{-1}v_{f}\right) _{k}\left( i\varepsilon ,j\varepsilon
,0\right) }{h\left\langle q\right\rangle }\dint\nolimits_{-h}^{0}\func{div}%
_{z}\left( \eta _{\varepsilon ,k,ij}\right) \left( x\right) dx_{3}=0\text{,
in }Y_{\varepsilon ,ij}\left( \omega \right) .
\end{equation*}

Therefore $v_{\varepsilon }^{0}\in \mathbf{V}_{\varepsilon }$.\ Moreover $%
v_{\varepsilon }^{0}$ is independent of $x_{3}$ in each fissure $%
Y_{\varepsilon ,ij}\left( \omega \right) $. Using the ergodic result (\ref%
{Ergo}) and making some computations, one easily proves that $\left(
v_{\varepsilon }^{0}\right) _{\varepsilon }$ $\tau _{0}$-converges to $v$.

We compute the limit $\lim_{\varepsilon \rightarrow 0}F_{\varepsilon }\left(
v_{\varepsilon }^{0}\right) $.\ We have%
\begin{equation*}
\begin{array}{ccl}
\underset{\varepsilon \rightarrow 0}{\lim }\mu \varepsilon
^{2}\dint\nolimits_{Y_{\varepsilon }\left( \omega \right) }\left\vert \nabla
v_{\varepsilon }^{0}\right\vert ^{2}dx & = & \underset{\varepsilon
\rightarrow 0}{\lim }\mu \varepsilon ^{2}\underset{\left( i,j\right) \in
I_{\varepsilon }\left( \omega \right) }{\dsum }\dint\nolimits_{Y_{%
\varepsilon ,ij}\left( \omega \right) }\left\vert \nabla \left(
v_{\varepsilon }^{0}\right) _{\tau }\right\vert ^{2}dx \\ 
&  & \quad +\underset{\varepsilon \rightarrow 0}{\lim }\mu \varepsilon ^{2}%
\underset{\left( i,j\right) \in I_{\varepsilon }\left( \omega \right) }{%
\dsum }\dint\nolimits_{Y_{\varepsilon ,ij}\left( \omega \right) }\left\vert
\nabla \left( v_{\varepsilon }^{0}\right) _{3}\right\vert ^{2}dx.%
\end{array}%
\end{equation*}

Using the expression (\ref{equ26})$_{1}$, we have%
\begin{equation*}
\begin{array}{l}
\underset{\varepsilon \rightarrow 0}{\lim }\mu \varepsilon ^{2}\underset{%
\left( i,j\right) \in I_{\varepsilon }\left( \omega \right) }{\dsum }%
\dint\nolimits_{Y_{\varepsilon ,ij}\left( \omega \right) }\left\vert \nabla
\left( v_{\varepsilon }^{0}\right) _{\tau }\right\vert ^{2}dx=\underset{%
\varepsilon \rightarrow 0}{\lim }\mu \varepsilon ^{2}\underset{\left(
i,j\right) \in I_{\varepsilon }\left( \omega \right) }{\dsum }\dfrac{1}{%
h^{2}\left\langle q\right\rangle ^{2}} \\ 
\quad \ \times \dint\nolimits_{Y_{\varepsilon ,ij}\left( \omega \right)
}\left( 
\begin{array}{l}
\left( \left( K_{f}\right) ^{-1}v_{f}\right) _{k}\left( i\varepsilon
,j\varepsilon ,0\right) \dint\nolimits_{-h}^{0}q_{i}\left( -\varepsilon
^{-\theta }x_{3}\right) \nabla _{z}\eta _{k}\left( z\left(
x_{1},x_{2},x_{3}\right) \right) dx_{3} \\ 
\quad \times \left( \left( K_{f}\right) ^{-1}v_{f}\right) _{l}\left(
i\varepsilon ,j\varepsilon ,0\right) \dint\nolimits_{-h}^{0}q_{i}\left(
-\varepsilon ^{-\theta }x_{3}\right) \nabla _{z}\eta _{l}\left( z\left(
x_{1},x_{2},x_{3}\right) \right) dx_{3}%
\end{array}%
\right) .%
\end{array}%
\end{equation*}

We introduce the change of variables $\left( z_{1},z_{2}\right) =z\left(
x_{1},x_{2},x_{3}\right) $, where $z\left( x_{1},x_{2},x_{3}\right) $ has
been defined in (\ref{zq})$_{1}$, and get%
\begin{equation*}
\begin{array}{l}
\underset{\varepsilon \rightarrow 0}{\lim }\mu \varepsilon ^{2}\underset{%
\left( i,j\right) \in I_{\varepsilon }\left( \omega \right) }{\dsum }%
\dint\nolimits_{Y_{\varepsilon ,ij}\left( \omega \right) }\left\vert \nabla
\left( v_{\varepsilon }^{0}\right) _{\tau }\right\vert ^{2}dx \\ 
\quad =\mu \underset{\varepsilon \rightarrow 0}{\lim }\underset{\left(
i,j\right) \in I_{\varepsilon }\left( \omega \right) }{\dsum }\dfrac{%
\varepsilon ^{2}}{\left\langle q\right\rangle ^{2}}\dint\nolimits_{-h}^{0}%
\dint\nolimits_{Z^{\prime }}\left( 
\begin{array}{l}
\left( \left( K_{f}\right) ^{-1}v_{f}\right) _{k}\left( i\varepsilon
,j\varepsilon ,0\right) \left( \left( K_{f}\right) ^{-1}v_{f}\right)
_{l}\left( i\varepsilon ,j\varepsilon ,0\right) \\ 
\quad \times \nabla _{z}\eta _{k}\left( z\right) \cdot \nabla _{z}\eta
_{l}\left( z\right) q_{i}\left( -\varepsilon ^{-\theta }x_{3}\right)
q_{j}\left( -\varepsilon ^{-\theta }x_{3}\right) dzdx_{3}%
\end{array}%
\right) .%
\end{array}%
\end{equation*}

Using the ergodicity result (\ref{Ergo}) and the definition (\ref{equ25})$%
_{1}$ of $K_{f}$, the above limit is equal to%
\begin{equation*}
\begin{array}{l}
\mu \dfrac{h\left\langle q^{2}\right\rangle }{\left\langle q\right\rangle
^{2}}\dint\nolimits_{\Sigma }\left( \left( \left( K_{f}\right)
^{-1}v_{f}\right) _{k}\left( \left( K_{f}\right) ^{-1}v_{f}\right)
_{l}\right) \left( x^{\prime }\right) dx^{\prime }\dint\nolimits_{Z^{\prime
}}\nabla _{z}\eta _{k}\left( z\right) \cdot \nabla _{z}\eta _{l}\left(
z\right) dz \\ 
\quad =\mu \dfrac{h\left\langle q^{2}\right\rangle }{\left\langle
q\right\rangle ^{2}}\dint\nolimits_{\Sigma }\left( K_{f}\right) ^{-1}\left(
v_{f}\right) _{\tau }\cdot \left( v_{f}\right) _{\tau }dx^{\prime },%
\end{array}%
\end{equation*}%
because $v_{f}$ is independent of $x_{3}$ in $Y_{h}$. Using a similar
argument, we have%
\begin{equation*}
\underset{\varepsilon \rightarrow 0}{\lim }\mu \varepsilon ^{2}\underset{%
\left( i,j\right) \in I_{\varepsilon }\left( \omega \right) }{\dsum }%
\dint\nolimits_{Y_{\varepsilon ,ij}\left( \omega \right) }\left\vert \nabla
\left( v_{\varepsilon }^{0}\right) _{3}\right\vert ^{2}dx=\mu h\dfrac{%
\left\langle q^{2}\right\rangle \left\langle 1/q^{2}\right\rangle }{k_{0}}%
\dint\nolimits_{\Sigma }\left( \left( v_{f}\right) _{3}\right)
^{2}dx^{\prime }.
\end{equation*}

On the other hand, observe that, for every $\psi \in \mathbf{C}%
_{c}^{1}\left( \Omega ;\mathbb{R}^{3}\right) $ we have%
\begin{equation*}
\begin{array}{rll}
\underset{\varepsilon \rightarrow 0}{\lim }\dint\nolimits_{\Omega
_{f}^{+,\varepsilon }}v_{\varepsilon }^{0}\cdot \psi dx & = & 
\dint\nolimits_{\Omega ^{+}}\left( \left( \widehat{K}^{+}\right)
^{-1}v_{d}^{+}\right) _{j}\dint_{Z^{1}}\Phi _{j}^{+}\left( z\right) dz\cdot
\psi dx \\ 
& = & \dint\nolimits_{\Omega ^{+}}v_{d}^{+}\cdot \psi dx, \\ 
\underset{\varepsilon \rightarrow 0}{\lim }\dint\nolimits_{\Gamma
_{0,\varepsilon }^{+}}K_{\varepsilon }^{+}\left( x\right) \cdot \psi \left(
x\right) dx & = & \underset{\varepsilon \rightarrow 0}{\lim }\underset{%
\left( i,j\right) \in I_{\varepsilon }\left( \omega \right) }{\dsum }%
\varepsilon ^{2}\dint_{a_{i}^{-}\left( 0\right) }^{a_{i}^{+}\left( 0\right)
}\dint_{a_{j}^{-}\left( 0\right) }^{a_{j}^{+}\left( 0\right) }K^{+}\left(
y^{\prime },0\right) dy^{\prime }\cdot \psi \left( i\varepsilon
,j\varepsilon ,0\right) \\ 
& = & \dint\nolimits_{\Sigma }\left( \dint_{\left\langle a^{-}\left(
0\right) \right\rangle }^{\left\langle a^{+}\left( 0\right) \right\rangle
}\dint_{\left\langle a^{-}\left( 0\right) \right\rangle }^{\left\langle
a^{+}\left( 0\right) \right\rangle }K^{+}\left( z_{1},z_{2},0\right)
dz_{1}dz_{2}\right) \cdot \psi \left( x^{\prime },0\right) dx^{\prime }.%
\end{array}%
\end{equation*}

We thus obtain%
\begin{equation*}
\begin{array}{l}
\underset{\varepsilon \rightarrow 0}{\lim }\left( 
\begin{array}{l}
\mu ^{+}\dint\nolimits_{\Omega ^{+}}\left( K_{\varepsilon }^{+}\right)
^{-1}v_{\varepsilon }^{0}\cdot v_{\varepsilon }^{0}dx+\mu
^{-}\dint\nolimits_{\Omega _{h}^{-}}\left( K_{\varepsilon }^{-}\right)
^{-1}v_{\varepsilon }^{0}\cdot v_{\varepsilon }^{0}dx \\ 
\quad +\gamma \dint\nolimits_{\Gamma _{0,\varepsilon }^{+}}\left(
K_{\varepsilon }^{+}\right) ^{-1/2}\left( v_{\varepsilon }^{0}\right) _{\tau
}\cdot \left( v_{\varepsilon }^{0}\right) _{\tau }dx^{\prime }+\gamma
\dint\nolimits_{\Gamma _{h,\varepsilon }^{-}}\left( K_{\varepsilon
}^{-}\right) ^{-1/2}\left( v_{\varepsilon }^{0}\right) _{\tau }\cdot \left(
v_{\varepsilon }^{0}\right) _{\tau }dx^{\prime }%
\end{array}%
\right) \\ 
\quad 
\begin{array}{cl}
= & \mu ^{+}\dint\nolimits_{\Omega ^{+}}\left( \widehat{K}^{+}\right)
^{-1}v_{d}^{+}\cdot v_{d}^{+}dx+\mu ^{-}\dint\nolimits_{\Omega
_{h}^{-}}\left( \widehat{K}^{-}\right) ^{-1}v_{d}^{-}\cdot v_{d}^{-}dx \\ 
& \quad +\left\langle q^{2}\right\rangle \gamma \dint\nolimits_{\Sigma
}\left( \left( K^{\ast +}\right) ^{-1/2}+\left( K^{\ast -}\right)
^{-1/2}\right) \left( v_{f}\right) _{\tau }\cdot \left( v_{f}\right) _{\tau
}dx^{\prime },%
\end{array}%
\end{array}%
\end{equation*}%
whence $\lim_{\varepsilon \rightarrow 0}F_{\varepsilon }\left(
v_{\varepsilon }^{0}\right) =F_{0}\left( v\right) $.

For every $v\in \mathbf{V}_{0}$, there exists a sequence $\left(
v_{m}\right) _{m}\subset \mathbf{C}^{1}\left( \overline{\Omega ^{+}\cup
\Omega _{h}^{-}};\mathbb{R}^{3}\right) \cap \mathbf{V}_{0}$ such that%
\begin{equation}
v_{m}\underset{m\rightarrow +\infty }{\rightarrow }v\text{, s-}\mathbf{L}%
^{2}\left( \Omega ^{+}\cup \Omega _{h}^{-};\mathbb{R}^{3}\right) .
\label{vm}
\end{equation}

Building the sequence $\left( \left( v_{m}\right) _{\varepsilon }^{0}\right)
_{\varepsilon }$ associated to $v_{m}$ through (\ref{fon-test}), the
sequence $\left( \left( v_{m}\right) _{\varepsilon }^{0}\right)
_{\varepsilon }$ $\tau _{0}$-converges to $v_{m}$, and using the above
computations for smooth functions, we have: $\lim_{\varepsilon \rightarrow
0}F_{\varepsilon }\left( \left( v_{m}\right) _{\varepsilon }^{0}\right)
=F_{0}\left( v_{m}\right) $. Hence $\lim_{m\rightarrow \infty
}\lim_{\varepsilon \rightarrow 0}F_{\varepsilon }\left( \left( v_{m}\right)
_{\varepsilon }^{0}\right) =F_{0}\left( v\right) $. Using the
diagonalization argument of \cite[Corollary 1.18]{Att}, there exists a
sequence $\left( v_{\varepsilon }^{0}\right) _{\varepsilon }$, $%
v_{\varepsilon }^{0}=\left( v_{m\left( \varepsilon \right) }\right)
_{\varepsilon }^{0}$ ($m\left( \varepsilon \right) \rightarrow _{\varepsilon
\rightarrow 0}+\infty $), such that $\left( v_{\varepsilon }^{0}\right)
_{\varepsilon }$ $\tau _{0}$-converges to $v$, and $\lim \sup_{\varepsilon
\rightarrow 0}F_{\varepsilon }\left( v_{\varepsilon }^{0}\right) \leq
F_{0}\left( v\right) $.

\textbf{Lower }$\mathbf{\Gamma }$\textbf{-limit.} Let $\left( v_{\varepsilon
}^{1}\right) _{\varepsilon }$ be a sequence such that $v_{\varepsilon
}^{1}\in \mathbf{V}_{\varepsilon }$ for every $\varepsilon $, and $\left(
v_{\varepsilon }^{1}\right) _{\varepsilon }$ $\tau _{0}$-converges to $v$.
We write the subdifferential inequality%
\begin{equation}
\mu \varepsilon ^{2}\dint\nolimits_{Y_{\varepsilon }\left( \omega \right)
}\left\vert \nabla v_{\varepsilon }^{1}\right\vert ^{2}dx\geq \mu
\varepsilon ^{2}\dint\nolimits_{Y_{\varepsilon }\left( \omega \right)
}\left\vert \nabla \left( v_{m}\right) _{\varepsilon }^{0}\right\vert
^{2}dx+2\mu \varepsilon ^{2}\dint\nolimits_{Y_{\varepsilon }\left( \omega
\right) }\nabla \left( v_{m}\right) _{\varepsilon }^{0}\cdot \left( \nabla
v_{\varepsilon }^{1}-\nabla \left( v_{m}\right) _{\varepsilon }^{0}\right)
dx,  \label{equ29}
\end{equation}%
where $\left( \left( v_{m}\right) _{\varepsilon }^{0}\right) _{\varepsilon }$
is the sequence associated to $v_{m}$ through (\ref{fon-test}) and where the
sequence $\left( v_{m}\right) _{m}$ satisfies the conditions (\ref{vm}).
Observe that%
\begin{equation*}
\dint\nolimits_{Y_{\varepsilon ,ij}\left( \omega \right) }\nabla \left(
v_{m}\right) _{\varepsilon }^{0}\cdot \left( \nabla v_{\varepsilon
}^{1}-\nabla \left( v_{m}\right) _{\varepsilon }^{0}\right)
dx=-\dint\nolimits_{Y_{\varepsilon ,ij}\left( \omega \right) }\Delta \left(
v_{m}\right) _{\varepsilon }^{0}\cdot \left( v_{\varepsilon }^{1}-\left(
v_{m}\right) _{\varepsilon }^{0}\right) dx,
\end{equation*}%
because $v_{\varepsilon }^{1}-\left( v_{m}\right) _{\varepsilon }^{0}=0$ on $%
\partial Y_{\varepsilon }\left( \omega \right) \backslash \left( \Gamma
_{0,\varepsilon }^{+}\cup \Gamma _{h,\varepsilon }^{-}\right) \left( \omega
\right) $, and $\frac{\partial \left( v_{m}\right) _{\varepsilon }^{0}}{%
\partial x_{3}}\mid _{x_{3}=0}=\frac{\partial \left( v_{m}\right)
_{\varepsilon }^{0}}{\partial x_{3}}\mid _{x_{3}=-h}=0$.\ Then, using the
ergodicity result (\ref{Ergo}), we have%
\begin{equation*}
\begin{array}{l}
\underset{\varepsilon \rightarrow 0}{\lim }2\mu \varepsilon
^{2}\dint\nolimits_{Y_{\varepsilon }\left( \omega \right) }\nabla \left(
v_{m}\right) _{\varepsilon }^{0}\cdot \left( \nabla v_{\varepsilon
}^{1}-\nabla \left( v_{m}\right) _{\varepsilon }^{0}\right) dx \\ 
\quad 
\begin{array}{cl}
= & -\mu \dfrac{h\left\langle q^{2}\right\rangle \left\langle
1/q^{2}\right\rangle }{\left\langle q\right\rangle }\dint\nolimits_{\Sigma
}\dint\nolimits_{Z^{\prime }}\Delta _{z}\eta _{k}\left( z\right) \left(
\left( K_{f}\right) ^{-1}\left( v_{m}\right) _{\tau }\right) _{k}\cdot
\left( v-v_{m}\right) _{\tau }\left( x^{\prime }\right) dzdx^{\prime } \\ 
& \quad -\mu \dfrac{h\left\langle q^{2}\right\rangle \left\langle
1/q^{2}\right\rangle }{k_{0}}\dint\nolimits_{\Sigma
}\dint\nolimits_{Z^{\prime }}\Delta _{z}\eta _{0}\left( z\right) \left(
v_{m}\right) _{3}\left( v-v_{m}\right) _{3}\left( x^{\prime }\right)
dzdx^{\prime },%
\end{array}%
\end{array}%
\end{equation*}%
whence%
\begin{equation*}
\underset{m\rightarrow \infty }{\lim }\underset{\varepsilon \rightarrow 0}{%
\lim }2\mu \varepsilon ^{2}\dint\nolimits_{Y_{\varepsilon }\left( \omega
\right) }\nabla \left( v_{m}\right) _{\varepsilon }^{0}\cdot \left( \nabla
v_{\varepsilon }^{1}-\nabla \left( v_{m}\right) _{\varepsilon }^{0}\right)
dx=0.
\end{equation*}

Recalling the inequality (\ref{equ29}) and the computations built in the
above case of smooth functions, we have%
\begin{equation*}
\begin{array}{l}
\underset{\varepsilon \rightarrow 0}{\lim \inf }\mu \varepsilon
^{2}\dint\nolimits_{Y_{\varepsilon }\left( \omega \right) }\left\vert \nabla
v_{\varepsilon }^{1}\right\vert ^{2}dx \\ 
\quad \geq \mu \dfrac{h\left\langle q^{2}\right\rangle }{\left\langle
q\right\rangle ^{2}}\dint\nolimits_{\Sigma }\left( K_{f}\right) ^{-1}\left(
v_{f}\right) _{\tau }\cdot \left( v_{f}\right) _{\tau }dx^{\prime }+\mu h%
\dfrac{\left\langle q^{2}\right\rangle }{k_{0}}\left\langle
1/q^{2}\right\rangle \dint\nolimits_{\Sigma }\left( \left( v_{f}\right)
_{3}\right) ^{2}dx^{\prime }.%
\end{array}%
\end{equation*}

Thus, computing in an easy way the $\lim \inf $ of the other terms in $%
F_{\varepsilon }\left( v_{\varepsilon }^{1}\right) $, we obtain: $\lim
\inf_{\varepsilon }F_{\varepsilon }\left( v_{\varepsilon }^{1}\right) \geq
F_{0}\left( v\right) $.
\end{proof}

Let us write the problem associated to the limit functional $F_{0}$.

\begin{corollary}
\label{corollary1}The solution $\left( v_{\varepsilon },p_{\varepsilon
}\right) $ of the problem (\ref{equ10})-(\ref{equ11}) verifies the following
properties:

\begin{itemize}
\item $\left( v_{\varepsilon }\right) _{\varepsilon }$ $\tau _{0}$-converges
to $v_{0}\in \mathbf{V}_{0}$, and set $v_{0}\mid _{\Omega ^{+}}=v_{0,d}^{+}$%
, $v_{0}\mid _{\Omega _{h}^{-}}=v_{0,d}^{-}$, $v_{0}\mid _{\Sigma }=v_{0,f}$.

\item $\left( p_{\varepsilon }^{+}\right) _{\varepsilon }$ converges to $%
p_{0}^{+}$, s-$L^{2}\left( \Omega ^{+}\right) $, $\left( p_{\varepsilon
}^{-}\right) _{\varepsilon }$ converges to $p_{0}^{-}$, s-$L^{2}\left(
\Omega _{h}^{-}\right) $ and (\ref{convp})$_{3}$ holds true.

\item $v_{0,d}^{+}$, $v_{0,d}^{-}$, $v_{0,f}$, $p_{0}^{+}$, $p_{0}^{-}$ and $%
\pi _{0}$ are solutions of the following problems:

\item[i)] in the regions $\Omega ^{+}$ and $\Omega _{h}^{-}$, one has the
Darcy laws%
\begin{equation}
\left\{ 
\begin{array}{rlll}
\mu ^{+}\left( \widehat{K}^{+}\right) ^{-1}v_{0,d}^{+}-\nabla p_{0}^{+} & =
& g^{+} & \text{in }\Omega ^{+}, \\ 
\func{div}\left( v_{0,d}^{+}\right) & = & 0 & \text{in }\Omega ^{+}, \\ 
\mu ^{-}\left( \widehat{K}^{-}\right) ^{-1}v_{0,d}^{-}-\nabla p_{0}^{-} & =
& g^{-} & \text{in }\Omega _{h}^{-}, \\ 
\func{div}\left( v_{0,d}^{-}\right) & = & 0 & \text{in }\Omega _{h}^{-}, \\ 
\left( v_{0,d}^{+}\right) _{3}\mid _{\Sigma \times \left\{ 0\right\} } & = & 
\left\langle q^{2}\right\rangle \left( v_{0,f}\right) _{3} & \text{on }%
\Sigma \times \left\{ 0\right\} , \\ 
\left( v_{0,d}^{-}\right) _{3}\mid _{\Sigma \times \left\{ -h\right\} } & =
& \left\langle q^{2}\right\rangle \left( v_{0,f}\right) _{3} & \Sigma \times
\left\{ -h\right\} ,%
\end{array}%
\right.  \label{Darcy}
\end{equation}

\item[ii)] on $\Sigma $, the velocity $\left( v_{0,f}\right) _{3}$ is given
through%
\begin{equation}
\left( v_{0,f}\right) _{3}\left( x^{\prime }\right) =\left( p_{0}^{+}\left(
x^{\prime },0\right) -p_{0}^{-}\left( x^{\prime },-h\right) \right) \dfrac{%
k_{0}}{\mu h\left\langle q^{2}\left( 0\right) \right\rangle \left\langle
1/q^{2}\left( 0\right) \right\rangle }  \label{v0f3}
\end{equation}%
and the tangential velocity $\left( v_{0,f}\right) _{\tau }$ satisfies the
modified Darcy law%
\begin{equation*}
\left\{ 
\begin{array}{rlll}
\dfrac{\mu }{\left\langle q\right\rangle ^{2}}\left( K_{f}\right)
^{-1}\left( v_{0,f}\right) _{\tau }+\dfrac{\gamma }{h}\left( \left( K^{\ast
+}\right) ^{-1/2}+\left( K^{\ast -}\right) ^{-1/2}\right) \left(
v_{0,f}\right) _{\tau }+\nabla \pi _{0} & = & 0 & \text{in }\Sigma , \\ 
\func{div}\left( v_{0,f}\right) _{\tau } & = & 0 & \text{in }\Sigma , \\ 
\left( v_{0,f}\right) _{\tau }\cdot n & = & 0 & \text{on }\partial \Sigma .%
\end{array}%
\right.
\end{equation*}
\end{itemize}
\end{corollary}

\begin{proof}
Thanks to the properties of the $\Gamma $-convergence, the sequence $\left(
v_{\varepsilon }\right) _{\varepsilon }$ $\tau _{0}$-converges to $v_{0}\in 
\mathbf{V}_{0}$ and $\lim_{\varepsilon \rightarrow 0}F_{\varepsilon }\left(
v_{\varepsilon }\right) =F_{0}\left( v_{0}\right) $, where $v_{0}$ is the
minimizer of the problem%
\begin{equation*}
\underset{v\in \mathbf{V}^{0}}{\inf }\left( F_{0}\left( v\right)
-2\dint\nolimits_{\Omega ^{+}}g^{+}\cdot vdx-2\dint\nolimits_{\Omega
_{h}^{-}}g^{-}\cdot vdx\right) .
\end{equation*}

For every $V\in \mathbf{V}_{0}$, we have the following identity%
\begin{equation*}
\begin{array}{l}
\mu ^{+}\dint\nolimits_{\Omega ^{+}}\left( \widehat{K}^{+}\right)
^{-1}v_{0,d}\cdot Vdx+\mu ^{-}\dint\nolimits_{\Omega _{h}^{-}}\left( 
\widehat{K}^{-}\right) ^{-1}v_{0,d}\cdot Vdx \\ 
\quad +\mu h\dfrac{\left\langle q^{2}\right\rangle }{\left\langle
q\right\rangle ^{2}}\dint\nolimits_{\Sigma }\left( K_{f}\right) ^{-1}\left(
v_{0,f}\right) _{\tau }\cdot V_{\tau }dx^{\prime }+\mu h\dfrac{\left\langle
q^{2}\right\rangle }{k_{0}}\left\langle 1/q^{2}\right\rangle
\dint\nolimits_{\Sigma }\left( v_{0,f}\right) _{3}V_{3}dx^{\prime } \\ 
\quad +\left\langle q^{2}\right\rangle \gamma \dint\nolimits_{\Sigma }\left(
\left( K^{\ast +}\right) ^{-1/2}+\left( K^{\ast -}\right) ^{-1/2}\right)
\left( v_{0,f}\right) _{\tau }\cdot V_{\tau }dx^{\prime } \\ 
\quad =\dint\nolimits_{\Omega ^{+}}g^{+}\cdot Vdx+\dint\nolimits_{\Omega
_{h}^{-}}g^{-}\cdot Vdx.%
\end{array}%
\end{equation*}

We infer the existence of a pressure $p_{0}^{+}$ (resp.\ $p_{0}^{-}$, $\pi
_{0}$) in $\Omega ^{+}$ (resp. $\Omega _{h}^{-}$, $\Sigma $) such that%
\begin{equation*}
\dint\nolimits_{\Omega ^{+}}\nabla p_{0}^{+}\cdot Vdx+\dint\nolimits_{\Omega
_{h}^{-}}\nabla p_{0}^{-}\cdot Vdx+h\dint\nolimits_{\Sigma }\nabla \pi
_{0}\cdot V_{\tau }dx^{\prime }+\mu h\dfrac{\left\langle q^{2}\right\rangle
\left\langle 1/q^{2}\right\rangle }{k_{0}}\dint\nolimits_{\Sigma }\left(
v_{0,f}\right) _{3}V_{3}dx^{\prime }=0,
\end{equation*}%
which implies, because $V$ (resp. $V_{\tau }$) is divergence-free in $\Omega
^{+}\cup \Omega _{h}^{-}$ (resp. $\Sigma $)%
\begin{equation*}
\dint\nolimits_{\Sigma }\left( -\left( p_{0}^{+}\left( x^{\prime },0\right)
-p_{0}^{-}\left( x^{\prime },-h\right) \right) +\mu h\dfrac{\left\langle
q^{2}\right\rangle \left\langle 1/q^{2}\right\rangle }{k_{0}}\left(
v_{0,f}\right) _{3}\right) V_{3}dx^{\prime }+h\dint\nolimits_{\partial
\Sigma }\pi _{0}n\cdot V_{\tau }d\sigma .
\end{equation*}

This gives the result.
\end{proof}

\section{Study of the transport problem}

Let us now consider the transport problem (\ref{eq9}).\ In this section, we
will describe the asymptotic behaviour of the solution $u_{\varepsilon }$ of
(\ref{eq9}), when $\varepsilon $ goes to 0, distinguishing between the cases 
$\mathcal{R}=0$ and $\mathcal{R}\neq 0$.

\subsection{Existence of a weak solution and a priori estimates}

We define the space%
\begin{equation*}
H_{\Gamma ^{+}\cup \Gamma ^{-}}^{1}\left( \Omega _{f}^{\varepsilon }\left(
\omega \right) \right) =\left\{ u\in H^{1}\left( \Omega _{f}^{\varepsilon
}\left( \omega \right) \right) \mid u=0\text{ on }\Gamma ^{+}\cup \Gamma
^{-}\right\} .
\end{equation*}

\begin{lemma}
\label{estim-contam}

\begin{enumerate}
\item The problem (\ref{eq9}) has a unique weak solution $u_{\varepsilon
}\in H_{\Gamma ^{+}\cup \Gamma ^{-}}^{1}\left( \Omega _{f}^{\varepsilon
}\left( \omega \right) \right) $ which is nonnegative in $\Omega
_{f}^{\varepsilon }\left( \omega \right) $.

\item There exists a non-random constant $C$ which is independent of $%
\varepsilon $ such that%
\begin{equation*}
\int_{\Omega _{f}^{\varepsilon }\left( \omega \right) }\left( u_{\varepsilon
}\right) ^{2}dx\leq C\text{ ; }\int_{\Omega _{f}^{\varepsilon }\left( \omega
\right) }\left\vert \nabla u_{\varepsilon }\right\vert ^{2}dx\leq C\text{.}
\end{equation*}

\item There exists a linear and bounded extension operator $P^{\varepsilon
}:H^{1}\left( \Omega _{f}^{\varepsilon }\left( \omega \right) \right)
\rightarrow H^{1}\left( \Omega \right) $ and two non-random positive
constants $C_{1}$ and $C_{2}$ such that%
\begin{equation*}
\begin{array}{rlll}
P^{\varepsilon }u_{\varepsilon } & = & u_{\varepsilon } & \text{in }\Omega
_{f}^{\varepsilon }\left( \omega \right) , \\ 
\dint_{\Omega }\left\vert P^{\varepsilon }u_{\varepsilon }\right\vert ^{2}dx
& \leq & C_{1}\dint_{\Omega _{f}^{\varepsilon }\left( \omega \right)
}\left\vert u_{\varepsilon }\right\vert ^{2}dx, &  \\ 
\dint_{\Omega }\left\vert \nabla P^{\varepsilon }u_{\varepsilon }\right\vert
^{2}dx & \leq & C_{2}\dint_{\Omega _{f}^{\varepsilon }\left( \omega \right)
}\left\vert \nabla u_{\varepsilon }\right\vert ^{2}dx. & 
\end{array}%
\end{equation*}
\end{enumerate}
\end{lemma}

\begin{proof}
1. Using the standard variational methods, one proves that the problem (\ref%
{eq9}) has a unique solution $u_{\varepsilon }\in H_{\Gamma ^{+}\cup \Gamma
^{-}}^{1}\left( \Omega _{f}^{\varepsilon }\left( \omega \right) \right) $.
Multiplying (\ref{eq9}) par $\left( u_{\varepsilon }\right) ^{-}=\min \left(
0,u_{\varepsilon }\right) $ and using Green's formula, one has%
\begin{equation*}
\begin{array}{l}
D\dint_{\Omega _{f}^{\varepsilon }\left( \omega \right) }\left\vert \nabla
\left( u_{\varepsilon }\right) ^{-}\right\vert ^{2}dx+\dint_{\Omega
_{f}^{\varepsilon }\left( \omega \right) }\left( v_{\varepsilon }\cdot
\nabla \left( u_{\varepsilon }\right) ^{-}\right) \left( u_{\varepsilon
}\right) ^{-}dx+\mathcal{R}\dint_{\Omega _{f}^{\varepsilon }\left( \omega
\right) }\left( \left( u_{\varepsilon }\right) ^{-}\right) ^{2}dx \\ 
\quad =\dint_{\Omega _{f}^{+,\varepsilon }}f\left( u_{\varepsilon }\right)
^{-}dx\leq 0,%
\end{array}%
\end{equation*}%
because $f$ is nonnegative.\ Because $\func{div}\left( v_{\varepsilon
}\right) =0$ in$\ \Omega _{f}^{\varepsilon }\left( \omega \right) $ and $%
v_{\varepsilon }\cdot n=0$ on $\partial \Omega _{f}^{\varepsilon }\left(
\omega \right) $, one has%
\begin{equation*}
\int_{\Omega _{f}^{\varepsilon }\left( \omega \right) }\left( v_{\varepsilon
}\cdot \nabla u_{\varepsilon }\right) u_{\varepsilon }dx=-\int_{\Omega
_{f}^{\varepsilon }\left( \omega \right) }\left( v_{\varepsilon }\cdot
\nabla u_{\varepsilon }\right) u_{\varepsilon }dx=0.
\end{equation*}

We deduce that $\int_{\Omega _{f}^{\varepsilon }\left( \omega \right)
}\left\vert \nabla \left( u_{\varepsilon }\right) ^{-}\right\vert ^{2}dx+%
\mathcal{R}\int_{\Omega _{f}^{\varepsilon }\left( \omega \right) }\left(
\left( u_{\varepsilon }\right) ^{-}\right) ^{2}dx\leq 0$, thus $\left(
u_{\varepsilon }\right) ^{-}=0$, in $\Omega _{f}^{\varepsilon }\left( \omega
\right) $, hence $u_{\varepsilon }$ is nonnegative in $\Omega
_{f}^{\varepsilon }\left( \omega \right) $.

2. As already observed, we have%
\begin{equation*}
D\int_{\Omega _{f}^{\varepsilon }\left( \omega \right) }\left\vert \nabla
u_{\varepsilon }\right\vert ^{2}dx+\mathcal{R}\int_{\Omega _{f}^{\varepsilon
}\left( \omega \right) }\left( u_{\varepsilon }\right) ^{2}dx=\int_{\Omega
_{f}^{+,\varepsilon }}fu_{\varepsilon }dx.
\end{equation*}

For $\mathcal{R}\neq 0$, one deduces from this equality, using
Cauchy-Schwarz' inequality, that $\int_{\Omega _{f}^{\varepsilon }\left(
\omega \right) }\left( u_{\varepsilon }\right) ^{2}dx\leq C$ and $%
\int_{\Omega _{f}^{\varepsilon }\left( \omega \right) }\left\vert \nabla
u_{\varepsilon }\right\vert ^{2}dx\leq C$.

In the case $\mathcal{R}=0$, one can prove, using \cite[Lemma 3.4]{All},
that there exists a non-random constant $C$ independent of $\varepsilon $
such that%
\begin{equation*}
\int_{\Omega _{f}^{\varepsilon }\left( \omega \right) }\left( u_{\varepsilon
}\right) ^{2}dx\leq C\int_{\Omega _{f}^{\varepsilon }\left( \omega \right)
}\left\vert \nabla u_{\varepsilon }\right\vert ^{2}dx.
\end{equation*}

Thus, using the above equality, we get the desired estimates.

3. This is a particular case of the result given in \cite{Acerb}.
\end{proof}

We will still denote by $u_{\varepsilon }$ its extension $P^{\varepsilon
}u_{\varepsilon }$ to the whole $\Omega $. From Lemmas \ref{estim-contam}
and \ref{conv-vites-fiss}, we deduce the existence of $u_{0}^{+}\in
H_{\Gamma ^{+}}^{1}\left( \Omega ^{+}\right) $ and $u_{0}^{-}\in H_{\Gamma
^{-}}^{1}\left( \Omega _{h}^{-}\right) $, such that, up to some subsequence%
\begin{equation}
\left\{ 
\begin{array}{rcll}
u_{\varepsilon }\mid _{\Omega ^{+}} & \underset{\varepsilon \rightarrow 0}{%
\rightharpoonup } & u_{0}^{+} & \text{w-}H_{\Gamma ^{+}}^{1}\left( \Omega
^{+}\right) , \\ 
u_{\varepsilon }\mid _{\Omega _{h}^{-}} & \underset{\varepsilon \rightarrow 0%
}{\rightharpoonup } & u_{0}^{-} & \text{w-}H_{\Gamma ^{-}}^{1}\left( \Omega
_{h}^{-}\right) , \\ 
\underset{\varepsilon \rightarrow 0}{\lim }\dint\nolimits_{Y_{\varepsilon
}\left( \omega \right) }u_{\varepsilon }\varphi dx & = & h\left\langle
q^{2}\left( 0\right) \right\rangle \dint\nolimits_{\Sigma }u_{0}^{+}\left(
x^{\prime },0\right) \varphi \left( x^{\prime }\right) dx^{\prime } & 
\forall \varphi \in C_{0}\left( \Sigma \right) .%
\end{array}%
\right.  \label{conv-solut}
\end{equation}

We intend to describe the problems satisfied by $u_{0}^{+}$ and $u_{0}^{-}$,
in their respective domains. We now define the notion of convergence
associated to sequences satisfying the above convergences.

\begin{definition}
A sequence $\left( U_{\varepsilon }\right) _{\varepsilon }$, with $%
U_{\varepsilon }\in H_{\Gamma ^{+}\cup \Gamma ^{-}}^{1}\left( \Omega
_{f}^{\varepsilon }\left( \omega \right) \right) $ for every $\varepsilon $, 
$\tau _{1}$-converges to $U$, with $U\mid _{\Omega ^{+}}=:U^{+}\in H_{\Gamma
^{+}}^{1}\left( \Omega ^{+}\right) $ and $U\mid _{\Omega _{h}^{-}}=:U^{-}\in
H_{\Gamma ^{-}}^{1}\left( \Omega _{h}^{-}\right) $, if the convergences (\ref%
{conv-solut}) are satisfied, replacing $u_{\varepsilon }$ by $U_{\varepsilon
}$.
\end{definition}

\subsection{The asymptotic behaviour in the case $\mathcal{R}=0$}

In this subsection, we deal with the case $\mathcal{R}=0$. Using the
boundary conditions (\ref{eq9})$_{2,3}$, we consider the variational
formulation of the problem (\ref{eq9})%
\begin{equation*}
\forall u\in H_{\Gamma ^{+}\cup \Gamma ^{-}}^{1}\left( \Omega
_{f}^{\varepsilon }\left( \omega \right) \right) :D\dint_{\Omega
_{f}^{\varepsilon }\left( \omega \right) }\nabla u_{\varepsilon }\cdot
\nabla udx+\dint_{\Omega _{f}^{\varepsilon }\left( \omega \right) }\left(
v_{\varepsilon }\cdot \nabla u_{\varepsilon }\right) udx=\dint_{\Omega
_{f}^{+,\varepsilon }}fudx,
\end{equation*}%
where $v_{\varepsilon }$ is the velocity of the fluid flow, that is the
solution of (\ref{equ10})-(\ref{equ11}). We consider the $Z$-periodic
solution $b_{j}$ of the cell problem%
\begin{equation}
\left\{ 
\begin{array}{rllll}
\Delta b_{j} & = & 0 & \text{in }Z^{1} & j=1,2,3, \\ 
\left( \nabla b_{j}+e_{j}\right) \cdot n & = & 0 & \text{on }S & 
\end{array}%
\right.  \label{cel-pb}
\end{equation}%
and the problem%
\begin{equation*}
\left\{ 
\begin{array}{rllll}
\Delta c_{m} & = & 0 & \text{in }Z^{\prime } & m=1,2, \\ 
\left( \nabla c_{m}+e_{m}\right) \cdot n & = & 0 & \text{on }\partial
Z^{\prime }. & 
\end{array}%
\right.
\end{equation*}

We define the tensors $\widehat{D}$ and $D^{\ast }$ through%
\begin{equation*}
\left\{ 
\begin{array}{rll}
\widehat{D}_{ij} & = & D\left( \left\vert Z^{1}\right\vert \delta
_{ij}+\dint_{Z^{1}}\dfrac{\partial b_{j}}{\partial z_{i}}dz\right) , \\ 
D_{ml}^{\ast } & = & D\left( \delta _{ml}+\dint_{Z^{\prime }}\dfrac{\partial
c_{l}}{\partial z_{m}}dz^{\prime }\right) .%
\end{array}%
\right.
\end{equation*}

Let $\chi _{+}^{\varepsilon }$ (resp. $\chi _{-}^{\varepsilon }$) be the
characteristic function of $\Omega _{f}^{+,\varepsilon }\left( \omega
\right) $ (resp. $\Omega _{h,f}^{-,\varepsilon }\left( \omega \right) $). We
have the following result.

\begin{lemma}
\label{coef-dif}One has, up to some subsequence:

\begin{enumerate}
\item $\chi _{+}^{\varepsilon }D\nabla u_{\varepsilon }\rightharpoonup
_{\varepsilon \rightarrow 0}\widehat{D}\nabla u_{0}^{+}$,\ in $\mathbf{L}%
^{2}\left( \Omega ^{+};\mathbb{R}^{3}\right) $-weak,

\item $\chi _{-}^{\varepsilon }D\nabla u_{\varepsilon }\rightharpoonup
_{\varepsilon \rightarrow 0}\widehat{D}\nabla u_{0}^{-}$,\ in $\mathbf{L}%
^{2}\left( \Omega _{h}^{-};\mathbb{R}^{3}\right) $-weak,

\item $\lim_{\varepsilon \rightarrow 0}\int\nolimits_{Y_{\varepsilon }\left(
\omega \right) }D\nabla _{\tau }u_{\varepsilon }\cdot \varphi
dx=h\left\langle q^{2}\left( 0\right) \right\rangle \int\nolimits_{\Sigma
}D^{\ast }\nabla _{\tau }u_{0}^{+}\left( x^{\prime },0\right) \cdot \varphi
dx^{\prime }$, $\forall \varphi \in \mathbf{C}_{c}^{\infty }\left( \Sigma ;%
\mathbb{R}^{2}\right) $, where $\nabla _{\tau }u_{0}^{+}=\left( \frac{%
\partial u_{0}^{+}}{\partial x_{1}},\frac{\partial u_{0}^{+}}{\partial x_{2}}%
\right) $.
\end{enumerate}
\end{lemma}

\begin{proof}
1. Let $\varphi \in C_{c}^{\infty }\left( \Omega ^{+}\right) $, $%
b_{j}^{\varepsilon }=\varepsilon b_{j}\left( x/\varepsilon \right) $.
Multiplying (\ref{eq9}) (for $\mathcal{R}=0$) by $\chi _{+}^{\varepsilon
}b_{j}^{\varepsilon }\varphi $, we get%
\begin{equation*}
\dint_{\Omega ^{+}}\chi _{+}^{\varepsilon }D\nabla u_{\varepsilon }\cdot
\left( \varphi \nabla b_{j}^{\varepsilon }+b_{j}^{\varepsilon }\nabla
\varphi \right) dx+\dint_{\Omega ^{+}}\chi _{+}^{\varepsilon }v_{\varepsilon
}\cdot \nabla u_{\varepsilon }b_{j}^{\varepsilon }\varphi dx=\dint_{\Omega
^{+}}\chi _{+}^{\varepsilon }fb_{j}^{\varepsilon }\varphi dx,
\end{equation*}%
from which we deduce that%
\begin{equation}
\underset{\varepsilon \rightarrow 0}{\lim }\dint_{\Omega ^{+}}\chi
_{+}^{\varepsilon }D\nabla u_{\varepsilon }\varphi \cdot \nabla
b_{j}^{\varepsilon }dx=0.  \label{grad1}
\end{equation}

Observe now that, through (\ref{cel-pb})%
\begin{equation*}
\begin{array}{rcl}
\dint_{\Omega ^{+}}\chi _{+}^{\varepsilon }D\nabla \left( \varphi
u_{\varepsilon }\right) \cdot \left( \nabla b_{j}^{\varepsilon
}+e_{j}\right) dx & = & D\dint_{\Omega ^{+}}\chi _{+}^{\varepsilon }\left(
\nabla b_{j}^{\varepsilon }+e_{j}\right) \varphi \cdot \nabla u_{\varepsilon
}dx \\ 
&  & \quad +D\dint_{\Omega ^{+}}\chi _{+}^{\varepsilon }\left( \nabla
b_{j}^{\varepsilon }+e_{j}\right) u_{\varepsilon }\cdot \nabla \varphi dx \\ 
& = & 0.%
\end{array}%
\end{equation*}

Thus, taking into account (\ref{grad1}), one has, up to some subsequence%
\begin{equation*}
\begin{array}{rll}
\underset{\varepsilon \rightarrow 0}{\lim }\dint_{\Omega ^{+}}\chi
_{+}^{\varepsilon }D\varphi \left( e_{j}\cdot \nabla u_{\varepsilon }\right)
dx & = & -\underset{\varepsilon \rightarrow 0}{\lim }D\dint_{\Omega
^{+}}\chi _{+}^{\varepsilon }\left( \nabla b_{j}^{\varepsilon }+e_{j}\right)
u_{\varepsilon }\cdot \nabla \varphi dx \\ 
& = & -D\underset{i=1}{\overset{3}{\dsum }}\left( \left\vert
Z^{1}\right\vert \delta _{ij}+\dint_{Z^{1}}\dfrac{\partial b_{j}}{\partial
z_{i}}dz\right) \dint_{\Omega ^{+}}\dfrac{\partial \varphi }{\partial x_{i}}%
u_{0}^{+}dx \\ 
& = & D\underset{i=1}{\overset{3}{\dsum }}\left( \left\vert Z^{1}\right\vert
\delta _{ij}+\dint_{Z^{1}}\dfrac{\partial b_{j}}{\partial z_{i}}dz\right)
\dint_{\Omega ^{+}}\dfrac{\partial u_{0}^{+}}{\partial x_{i}}\varphi dx.%
\end{array}%
\end{equation*}

2. In a similar way than above, we get, for every $\varphi \in C_{c}^{\infty
}\left( \Omega _{h}^{-}\right) $%
\begin{equation*}
\underset{\varepsilon \rightarrow 0}{\lim }\dint_{\Omega _{h}^{-}}\chi
_{-}^{\varepsilon }D\varphi \left( e_{j}\cdot \nabla u_{\varepsilon }\right)
dx=D\underset{i=1}{\overset{3}{\dsum }}\left( \left\vert Z^{1}\right\vert
\delta _{ij}+\dint_{Z^{1}}\dfrac{\partial b_{j}}{\partial z_{i}}dz\right)
\dint_{\Omega _{h}^{-}}\frac{\partial u_{0}^{-}}{\partial x_{i}}\varphi dx.
\end{equation*}

3. We define the quantity $c_{m}^{\varepsilon }$, $m=1,2$, through%
\begin{equation*}
c_{m}^{\varepsilon }\left( x\right) =\varepsilon \left\langle q\left(
0\right) \right\rangle c_{m}\left( 
\begin{array}{c}
\dfrac{x_{1}-i\varepsilon -\varepsilon \left( a_{i}^{-}\left( -\varepsilon
^{-\theta }x_{3}\right) +a_{i}^{+}\left( -\varepsilon ^{-\theta
}x_{3}\right) \right) /2}{\varepsilon q_{i}\left( -\varepsilon ^{-\theta
}x_{3}\right) } \\ 
\dfrac{x_{2}-j\varepsilon -\varepsilon \left( a_{j}^{-}\left( -\varepsilon
^{-\theta }x_{3}\right) +a_{j}^{+}\left( -\varepsilon ^{-\theta
}x_{3}\right) \right) /2}{\varepsilon q_{j}\left( -\varepsilon ^{-\theta
}x_{3}\right) }%
\end{array}%
\right) ,
\end{equation*}

Then, using curvilinear coordinates, the ergodic result (\ref{Ergo}) and
making some computations as before, we get the result.
\end{proof}

Our main result in this subsection reads as follows.

\begin{theorem}
\label{theorem2}The sequence $\left( u_{\varepsilon }\right) _{\varepsilon }$%
, where $u_{\varepsilon }$ is the solution of (\ref{eq9}), converges in the
topology $\tau _{1}$ to the solution $\left( u_{0}^{+},u_{0}^{-}\right) $ of
the variational formulation%
\begin{equation}
\begin{array}{l}
\forall \left( u^{+},u^{-}\right) \in H_{\Gamma ^{+}}^{1}\left( \Omega
^{+}\right) \times H_{\Gamma ^{-}}^{1}\left( \Omega _{h}^{-}\right)
:\dint_{\Omega ^{+}}\widehat{D}\nabla u_{0}^{+}\cdot \nabla
u^{+}dx+\dint_{\Omega _{h}^{-}}\widehat{D}\nabla u_{0}^{-}\cdot \nabla
u^{-}dx \\ 
\quad +\dint\nolimits_{\Omega ^{+}}\left( v_{0,d}^{+}\cdot \nabla
u_{0}^{+}\right) u^{+}dx+\dint\nolimits_{\Omega _{h}^{-}}\left(
v_{0,d}^{-}\cdot \nabla u_{0}^{-}\right) u^{-}dx \\ 
\quad +h\left\langle q^{2}\left( 0\right) \right\rangle
\dint\nolimits_{\Sigma }D^{\ast }\nabla _{\tau }u_{0}^{+}\cdot \nabla _{\tau
}u^{+}dx^{\prime }-h\left\langle q^{2}\left( 0\right) \right\rangle
\dint\nolimits_{\Sigma }\left( \left( v_{0}\right) _{\tau }\cdot \nabla
_{\tau }u^{+}\right) u_{0}^{+}dx^{\prime } \\ 
\quad +\dfrac{D}{h\left\langle 1/q^{2}\right\rangle }\dint\nolimits_{\Sigma
}\left( u_{0}^{+}\left( u^{-}-u^{+}\right) -u_{0}^{-}\left(
u^{-}-u^{+}\right) \exp \left( \dfrac{p_{0}^{+}-p_{0}^{-}}{D\langle
q^{2}\rangle \left\langle 1/q^{2}\right\rangle }\dfrac{k_{0}}{\mu }\right)
\right) dx^{\prime } \\ 
\qquad =\left\vert Z^{1}\right\vert \dint_{\Omega ^{+}}fu^{+}dx,%
\end{array}
\label{varform}
\end{equation}%
where $v_{0,d}^{+}$ and $v_{0,d}^{-}$ are the limit velocities appearing in
Remark \ref{Remark1} and $p_{0}^{+}$ and $p_{0}^{-}$ are the pressures
appearing in Corollary \ref{corollary1}.
\end{theorem}

Before starting the proof of Theorem \ref{theorem2}, let us introduce the
constant "vertical" velocity $\left( v_{\varepsilon ,ij}\right) _{3}$ in the
fissure $Y_{\varepsilon ,ij}\left( \omega \right) $, $\left( i,j\right) \in
I_{\varepsilon }\left( \omega \right) $, defined as%
\begin{equation}
\left( v_{\varepsilon ,ij}\right) _{3}=\left( p_{0,\varepsilon
,ij}^{+}-p_{0,\varepsilon ,ij}^{-}\right) \dfrac{k_{0}}{\mu h\left\langle
q^{2}\right\rangle \left\langle 1/q^{2}\right\rangle },  \label{veloci}
\end{equation}%
where $p_{0,\varepsilon ,ij}^{+}=p_{0}^{+}\left( i\varepsilon ,j\varepsilon
,0\right) $, $p_{0,\varepsilon ,ij}^{-}=p_{0}^{-}\left( i\varepsilon
,j\varepsilon ,-h\right) $, $p_{0}^{+}$ and $p_{0}^{-}$ being the pressures
defined in Corollary \ref{corollary1} (compare to (\ref{v0f3})). Inside the
fissure $Y_{\varepsilon ,ij}\left( \omega \right) $, for every $\left(
i,j\right) \in I_{\varepsilon }\left( \omega \right) $, we define, for every 
$u\in C^{2}\left( \overline{\Omega }\right) $ satisfying $u=0$ on $\Gamma
=\Gamma ^{+}\cup \Gamma ^{-}$%
\begin{equation}
\begin{array}{l}
\bar{u}_{\varepsilon ,ij}\left( x_{1},x_{2},x_{3}\right) =u^{+}\left(
x_{1},x_{2},0\right) \\ 
\quad +\dfrac{\left( u^{-}\left( x_{1},x_{2},-h\right) -u^{+}\left(
x_{1},x_{2},0\right) \right) \dint\nolimits_{x_{3}}^{0}\dfrac{1}{q_{i}\left(
-\varepsilon ^{-\theta }t\right) q_{j}\left( -\varepsilon ^{-\theta
}t\right) }\exp \left( -\dfrac{t\left( v_{\varepsilon ,ij}\right) _{3}}{D}%
\right) dt}{\dint\nolimits_{-h}^{0}\dfrac{1}{q_{i}\left( -\varepsilon
^{-\theta }t\right) q_{j}\left( -\varepsilon ^{-\theta }t\right) }\exp
\left( -\dfrac{t\left( v_{\varepsilon ,ij}\right) _{3}}{D}\right) dt},%
\end{array}
\label{uepsijf}
\end{equation}%
where $q_{i}\left( s\right) $ is defined in (\ref{zq})$_{2}$. We finally
define the test-function $u_{0,\varepsilon }$ through%
\begin{equation}
u_{0,\varepsilon }\left( x\right) =\left\{ 
\begin{array}{ll}
u^{+}\left( x\right) & \text{in }\Omega ^{+}, \\ 
\bar{u}_{\varepsilon ,ij}\left( x\right) & \text{in }Y_{\varepsilon
,ij}\left( \omega \right) \text{, }\forall \left( i,j\right) \in
I_{\varepsilon }\left( \omega \right) , \\ 
u^{-}\left( x\right) & \text{in }\Omega _{h}^{-}.%
\end{array}%
\right.  \label{equ42}
\end{equation}

The properties of this test-function are gathered in the following result.

\begin{lemma}
\label{proposition1}

\begin{enumerate}
\item One has:%
\begin{equation}
\left\{ 
\begin{array}{rlll}
-D\dfrac{\partial }{\partial x_{3}}\left( q_{i}q_{j}\left( -\varepsilon
^{-\theta }x_{3}\right) \dfrac{\partial \bar{u}_{\varepsilon ,ij}}{\partial
x_{3}}\right) \left( x\right) \quad &  &  &  \\ 
-\left( q_{i}q_{j}\right) \left( -\varepsilon ^{-\theta }x_{3}\right) \left(
v_{\varepsilon ,ij}\right) _{3}\dfrac{\partial \bar{u}_{\varepsilon ,ij}}{%
\partial x_{3}}\left( x\right) & = & 0 & \text{in }Y_{\varepsilon ,ij}\left(
\omega \right) , \\ 
\bar{u}_{\varepsilon ,ij}\left( x_{1},x_{2},0\right) & = & u^{+}\left(
x_{1},x_{2},0\right) , & \text{on }\Gamma _{0,\varepsilon ,ij}^{+}\left(
\omega \right) , \\ 
\bar{u}_{\varepsilon ,ij}\left( x_{1},x_{2},-h\right) & = & u^{-}\left(
x_{1},x_{2},-h\right) & \text{on }\Gamma _{h,\varepsilon ,ij}^{-}\left(
\omega \right) .%
\end{array}%
\right.  \label{equaub}
\end{equation}

\item For every $\varepsilon >0$, $u_{0,\varepsilon }\in H_{\Gamma ^{+}\cup
\Gamma ^{-}}^{1}\left( \Omega _{\varepsilon }\left( \omega \right) \right) $.

\item The sequence $\left( u_{0,\varepsilon }\right) _{\varepsilon }$ $\tau
_{1}$-converges to $u$.
\end{enumerate}
\end{lemma}

\begin{proof}
1. This is an immediate consequence of the definition (\ref{uepsijf}) of $%
\bar{u}_{\varepsilon ,ij}$.

2. This is an immediate consequence of the construction (\ref{equ42}) of $%
u_{0,\varepsilon }$, in $\Omega ^{+}$ and in $\Omega _{h}^{-}$, and through
the "boundary conditions" (\ref{equaub})$_{2,3}$ satisfied by $%
u_{0,\varepsilon }$ at the ends of the fissure $Y_{\varepsilon ,ij}\left(
\omega \right) $.

3. From this construction, we deduce that $\left( u_{0,\varepsilon }\mid
_{\Omega ^{+}}\right) _{\varepsilon }$ (resp. $\left( u_{0,\varepsilon }\mid
_{\Omega _{h}^{-}}\right) _{\varepsilon }$) converges to $u^{+}$ (resp. $%
u^{-}$) in $H^{1}\left( \Omega ^{+}\right) $-strong (resp. $H^{1}\left(
\Omega _{h}^{-}\right) $-strong).

Moreover, for every $\varphi \in C_{0}\left( 
%TCIMACRO{\U{211d} }%
%BeginExpansion
\mathbb{R}
%EndExpansion
^{3}\right) $, we define%
\begin{equation*}
A_{\varepsilon }=\dint\nolimits_{Y_{\varepsilon }\left( \omega \right)
}\varphi \dfrac{\dint\nolimits_{x_{3}}^{0}\dfrac{1}{q_{i}\left( -\varepsilon
^{-\theta }t\right) q_{j}\left( -\varepsilon ^{-\theta }t\right) }\exp
\left( -\dfrac{t\left( v_{\varepsilon ,ij}\right) _{3}}{D}\right) dt}{%
\dint\nolimits_{-h}^{0}\dfrac{1}{q_{i}\left( -\varepsilon ^{-\theta
}t\right) q_{j}\left( -\varepsilon ^{-\theta }t\right) }\exp \left( -\dfrac{%
t\left( v_{\varepsilon ,ij}\right) _{3}}{D}\right) dt}dx.
\end{equation*}

One has%
\begin{equation*}
\begin{array}{lll}
\underset{\varepsilon \rightarrow 0}{\lim }A_{\varepsilon } & = & \underset{%
\varepsilon \rightarrow 0}{\lim }\underset{\left( i,j\right) \in
I_{\varepsilon }\left( \omega \right) }{\dsum }\dfrac{1}{\dint%
\nolimits_{-h}^{0}\dfrac{1}{q_{i}\left( -\varepsilon ^{-\theta }t\right)
q_{j}\left( -\varepsilon ^{-\theta }t\right) }\exp \left( -\dfrac{t\left(
v_{\varepsilon ,ij}\right) _{3}}{D}\right) dt} \\ 
&  & \quad \times \dint\nolimits_{-h}^{0}\varphi \left( \varepsilon
i,\varepsilon j,x_{3}\right) \left( 
\begin{array}{l}
\dint\nolimits_{x_{3}}^{0}\dfrac{1}{q_{i}\left( -\varepsilon ^{-\theta
}t\right) q_{j}\left( -\varepsilon ^{-\theta }t\right) } \\ 
\quad \times \exp \left( -\dfrac{t\left( v_{\varepsilon ,ij}\right) _{3}}{D}%
\right) dt%
\end{array}%
\right) q^{2}\left( -\varepsilon ^{-\theta }x_{3}\right) dx_{3}%
\end{array}%
\end{equation*}%
and%
\begin{equation*}
\begin{array}{lll}
\underset{\varepsilon \rightarrow 0}{\lim }A_{\varepsilon } & = & \underset{%
\varepsilon \rightarrow 0}{\lim }\underset{\left( i,j\right) \in
I_{\varepsilon }\left( \omega \right) }{\dsum }\dfrac{1}{\dint%
\nolimits_{-h}^{0}\dfrac{1}{q_{i}\left( -\varepsilon ^{-\theta }t\right)
q_{j}\left( -\varepsilon ^{-\theta }t\right) }\exp \left( -\dfrac{t\left(
v_{\varepsilon ,ij}\right) _{3}}{D}\right) dt} \\ 
&  & \quad \times \varepsilon ^{\theta }\dint\nolimits_{-h\varepsilon
^{-\theta }}^{0}\varphi \left( \varepsilon i,\varepsilon j,\varepsilon
^{\theta }\xi \right) \left( 
\begin{array}{l}
\dint\nolimits_{\varepsilon ^{\theta }\xi }^{0}\dfrac{1}{q_{i}\left(
-\varepsilon ^{-\theta }t\right) q_{j}\left( -\varepsilon ^{-\theta
}t\right) } \\ 
\quad \times \exp \left( -\dfrac{t\left( v_{\varepsilon ,ij}\right) _{3}}{D}%
\right) dt%
\end{array}%
\right) q^{2}\left( \xi \right) d\xi \\ 
& = & 0.%
\end{array}%
\end{equation*}

Thus%
\begin{equation*}
\begin{array}{l}
\underset{\varepsilon \rightarrow 0}{\lim }\dint\nolimits_{Y_{\varepsilon
}\left( \omega \right) }\varphi \dfrac{u^{-}\left( x_{1},x_{2},-h\right)
-u^{+}\left( x_{1},x_{2},0\right) }{\dint\nolimits_{-h}^{0}\dfrac{1}{%
q_{i}\left( -\varepsilon ^{-\theta }t\right) q_{j}\left( -\varepsilon
^{-\theta }t\right) }\exp \left( -\dfrac{t\left( v_{\varepsilon ,ij}\right)
_{3}}{D}\right) dt} \\ 
\quad \times \dint\nolimits_{x_{3}}^{0}\dfrac{1}{q_{i}\left( -\varepsilon
^{-\theta }t\right) q_{j}\left( -\varepsilon ^{-\theta }t\right) }\exp
\left( -\dfrac{t\left( v_{\varepsilon ,ij}\right) _{3}}{D}\right) dtdx=0,%
\end{array}%
\end{equation*}%
from which we deduce that%
\begin{equation*}
\underset{\varepsilon \rightarrow 0}{\lim }\dint\nolimits_{Y_{\varepsilon
}\left( \omega \right) }\varphi u_{0,\varepsilon }dx=h\left\langle
q^{2}\left( 0\right) \right\rangle \dint\nolimits_{\Sigma }u^{+}\left(
x^{\prime },0\right) \varphi \left( x^{\prime },0\right) dx^{\prime }\text{.}
\end{equation*}

Thus the sequence $\left( u_{0,\varepsilon }\right) _{\varepsilon }$ $\tau
_{1}$-converges to $u$.
\end{proof}

\begin{proof}[Proof of Theorem \protect\ref{theorem2}]
Thanks to the boundary conditions (\ref{equ10})$_{4,7}$, we have%
\begin{equation*}
\dint\nolimits_{\Omega _{f}^{\varepsilon }\left( \omega \right) }\left(
v_{\varepsilon }\cdot \nabla u_{\varepsilon }\right) u_{0,\varepsilon
}dx=-\dint\nolimits_{\Omega _{f}^{\varepsilon }\left( \omega \right) }\left(
v_{\varepsilon }\cdot \nabla u_{0,\varepsilon }\right) u_{\varepsilon }dx.
\end{equation*}

Using Lemma \ref{proposition1} and the "compensated compactness" result (see 
\cite{Mur}), we immediately deduce the following limits%
\begin{equation*}
\begin{array}{rll}
\underset{\varepsilon \rightarrow 0}{\lim }\dint\nolimits_{\Omega
^{+}}\left( v_{\varepsilon }\cdot \nabla u_{\varepsilon }\right)
u_{0,\varepsilon }dx & = & \dint\nolimits_{\Omega ^{+}}\left(
v_{0,d}^{+}\cdot \nabla u_{0}^{+}\right) u^{+}dx, \\ 
\underset{\varepsilon \rightarrow 0}{\lim }\dint\nolimits_{\Omega
_{h}^{-}}\left( v_{\varepsilon }\cdot \nabla u_{\varepsilon }\right)
u_{0,\varepsilon }dx & = & \dint\nolimits_{\Omega _{h}^{-}}\left(
v_{0,d}^{-}\cdot \nabla u_{0}^{-}\right) u^{-}dx,%
\end{array}%
\end{equation*}%
$u_{0,\varepsilon }$ being independent of $\varepsilon $ in $\Omega ^{+}\cup
\Omega _{h}^{-}$. We then write%
\begin{equation*}
\begin{array}{l}
D\dint\nolimits_{Y_{\varepsilon }\left( \omega \right) }\nabla
u_{\varepsilon }\cdot \nabla u_{0,\varepsilon
}dx-\dint\nolimits_{Y_{\varepsilon }\left( \omega \right) }\left(
v_{\varepsilon }\cdot \nabla u_{0,\varepsilon }\right) u_{\varepsilon }dx \\ 
\quad 
\begin{array}{ll}
= & D\dint\nolimits_{Y_{\varepsilon }\left( \omega \right) }\nabla _{\tau
}u_{\varepsilon }\cdot \nabla _{\tau }u_{0,\varepsilon
}dx-\dint\nolimits_{Y_{\varepsilon }\left( \omega \right) }\left( \left(
v_{\varepsilon }\right) _{\tau }\cdot \nabla _{\tau }u_{0,\varepsilon
}\right) u_{\varepsilon }dx \\ 
& \quad +D\dint\nolimits_{Y_{\varepsilon }\left( \omega \right) }\dfrac{%
\partial u_{\varepsilon }}{\partial x_{3}}\dfrac{\partial u_{0,\varepsilon }%
}{\partial x_{3}}dx-\dint\nolimits_{Y_{\varepsilon }\left( \omega \right)
}\left( v_{\varepsilon ,s}\right) _{3}\dfrac{\partial u_{0,\varepsilon }}{%
\partial x_{3}}u_{\varepsilon }dx.%
\end{array}%
\end{array}%
\end{equation*}

A direct computation gives%
\begin{equation*}
\begin{array}{l}
\underset{\varepsilon \rightarrow 0}{\lim }\left(
D\dint\nolimits_{Y_{\varepsilon }\left( \omega \right) }\dfrac{\partial
u_{\varepsilon }}{\partial x_{3}}\dfrac{\partial u_{0,\varepsilon }}{%
\partial x_{3}}dx-\dint\nolimits_{Y_{\varepsilon }\left( \omega \right)
}\left( \left( v_{\varepsilon ,s}\right) _{3}\dfrac{\partial
u_{0,\varepsilon }}{\partial x_{3}}\right) u_{\varepsilon }dx\right) \\ 
\quad 
\begin{array}{ll}
= & \underset{\varepsilon \rightarrow 0}{\lim }\underset{\left( i,j\right)
\in I_{\varepsilon }\left( \omega \right) }{\dsum }\dint\nolimits_{Y_{%
\varepsilon ,ij}\left( \omega \right) }\left( D\dfrac{\partial
u_{\varepsilon }}{\partial x_{3}}\dfrac{\partial \bar{u}_{\varepsilon ,ij}}{%
\partial x_{3}}q_{i}q_{j}-\left( v_{\varepsilon ,s}\right) _{3}\dfrac{%
\partial \bar{u}_{\varepsilon ,ij}}{\partial x_{3}}u_{\varepsilon
}q_{i}q_{j}\right) dx \\ 
= & \underset{\varepsilon \rightarrow 0}{\lim }\left( 
\begin{array}{l}
\underset{\left( i,j\right) \in I_{\varepsilon }\left( \omega \right) }{%
\dsum }\dint\nolimits_{Y_{\varepsilon ,ij}\left( \omega \right) }\left( -D%
\dfrac{\partial }{\partial x_{3}}\left( q_{i}q_{j}\dfrac{\partial \bar{u}%
_{\varepsilon ,ij}}{\partial x_{3}}\right) -\left( v_{\varepsilon ,s}\right)
_{3}q_{i}q_{j}\dfrac{\partial \bar{u}_{\varepsilon ,ij}}{\partial x_{3}}%
\right) u_{\varepsilon }dx \\ 
\quad -D\underset{\left( i,j\right) \in I_{\varepsilon }\left( \omega
\right) }{\dsum }\dint\nolimits_{\Gamma _{0,\varepsilon ,ij}^{+}\left(
\omega \right) }\dfrac{\partial \bar{u}_{\varepsilon ,ij}}{\partial x_{3}}%
\mid _{x_{3}=0}q_{i}q_{j}\left( 0\right) u_{\varepsilon }dx^{\prime } \\ 
\quad +D\underset{\left( i,j\right) \in I_{\varepsilon }\left( \omega
\right) }{\dsum }\dint\nolimits_{\Gamma _{h,\varepsilon ,ij}^{-}\left(
\omega \right) }\left( \dfrac{\partial \bar{u}_{\varepsilon ,ij}}{\partial
x_{3}}\mid _{x_{3}=-h}\right) q_{i}q_{j}\left( -h\varepsilon ^{-\theta
}\right) u_{\varepsilon }dx^{\prime }%
\end{array}%
\right) .%
\end{array}%
\end{array}%
\end{equation*}

Hence%
\begin{equation*}
\begin{array}{l}
\underset{\varepsilon \rightarrow 0}{\lim }\left(
D\dint\nolimits_{Y_{\varepsilon }\left( \omega \right) }\dfrac{\partial
u_{\varepsilon }}{\partial x_{3}}\dfrac{\partial u_{0,\varepsilon }}{%
\partial x_{3}}dx-\dint\nolimits_{Y_{\varepsilon }\left( \omega \right)
}\left( \left( v_{\varepsilon ,s}\right) _{3}\dfrac{\partial
u_{0,\varepsilon }}{\partial x_{3}}\right) u_{\varepsilon }dx\right) \\ 
\quad =\underset{\varepsilon \rightarrow 0}{\lim }\left( 
\begin{array}{l}
\underset{\left( i,j\right) \in I_{\varepsilon }\left( \omega \right) }{%
\dsum }\dint\nolimits_{Y_{\varepsilon ,ij}\left( \omega \right) }\left(
\left( v_{\varepsilon ,ij}\right) _{3}-\left( v_{\varepsilon ,s}\right)
_{3}\right) \dfrac{\partial \bar{u}_{\varepsilon ,ij}}{\partial x_{3}}%
q_{i}q_{j}u_{\varepsilon }dx \\ 
\quad -D\underset{\left( i,j\right) \in I_{\varepsilon }\left( \omega
\right) }{\dsum }\dint\nolimits_{\Gamma _{0,\varepsilon ,ij}^{+}\left(
\omega \right) }\left( q_{i}q_{j}\left( 0\right) \dfrac{\partial \bar{u}%
_{\varepsilon ,ij}}{\partial x_{3}}\mid _{x_{3}=0}u_{\varepsilon
}^{+}\right) dx^{\prime } \\ 
\quad -D\underset{\left( i,j\right) \in I_{\varepsilon }\left( \omega
\right) }{\dsum }\dint\nolimits_{\Gamma _{h,\varepsilon ,ij}^{-}\left(
\omega \right) }\left( \dfrac{\partial \bar{u}_{\varepsilon ,ij}}{\partial
x_{3}}\mid _{x_{3}=-h}\right) q_{i}q_{j}\left( -h\varepsilon ^{-\theta
}\right) u_{\varepsilon }^{-}dx^{\prime }%
\end{array}%
\right) .%
\end{array}%
\end{equation*}

Using (\ref{equaub})$_{1}$, we have%
\begin{equation*}
\begin{array}{l}
\underset{\varepsilon \rightarrow 0}{\lim }\left(
D\dint\nolimits_{Y_{\varepsilon }\left( \omega \right) }\dfrac{\partial
u_{\varepsilon }}{\partial x_{3}}\dfrac{\partial u_{0,\varepsilon }}{%
\partial x_{3}}dx-\dint\nolimits_{Y_{\varepsilon }\left( \omega \right)
}\left( \left( v_{\varepsilon ,s}\right) _{3}\dfrac{\partial
u_{0,\varepsilon }}{\partial x_{3}}\right) u_{\varepsilon }dx\right) \\ 
\quad 
\begin{array}{ll}
= & D\underset{\varepsilon \rightarrow 0}{\lim }\underset{\left( i,j\right)
\in I_{\varepsilon }\left( \omega \right) }{\dsum }\dint\nolimits_{\Gamma
_{0,\varepsilon ,ij}^{+}\left( \omega \right) }\dfrac{\left( u^{-}\left(
x^{\prime },-h\right) -u^{+}\left( x^{\prime },0\right) \right) }{%
\dint\nolimits_{-h}^{0}\dfrac{1}{q_{i}\left( -\varepsilon ^{-\theta
}t\right) q_{j}\left( -\varepsilon ^{-\theta }t\right) }\exp \left( -\dfrac{%
t\left( v_{\varepsilon ,ij}\right) _{3}}{D}\right) dt}u_{\varepsilon
}^{+}\left( x^{\prime },-h\right) dx^{\prime } \\ 
& -D\underset{\varepsilon \rightarrow 0}{\lim }\underset{\left( i,j\right)
\in I_{\varepsilon }\left( \omega \right) }{\dsum }\dint\nolimits_{\Gamma
_{h,\varepsilon ,ij}^{-}\left( \omega \right) }\dfrac{\left( u^{-}\left(
x^{\prime },-h\right) -u^{+}\left( x^{\prime },0\right) \right) \exp \left( 
\dfrac{h\left( v_{\varepsilon ,ij}\right) _{3}}{D}\right) }{%
\dint\nolimits_{-h}^{0}\dfrac{1}{q_{i}\left( -\varepsilon ^{-\theta
}t\right) q_{j}\left( -\varepsilon ^{-\theta }t\right) }\exp \left( -\dfrac{%
t\left( v_{\varepsilon ,ij}\right) _{3}}{D}\right) dt}u_{\varepsilon
}^{-}\left( x^{\prime },0\right) dx^{\prime }.%
\end{array}%
\end{array}%
\end{equation*}

Introducing the change of variables $s=-\varepsilon ^{-\theta }t$, we get%
\begin{equation*}
\begin{array}{l}
\dint\nolimits_{-h}^{0}\dfrac{1}{q_{i}\left( -\varepsilon ^{-\theta
}t\right) q_{j}\left( -\varepsilon ^{-\theta }t\right) }\exp \left( -\dfrac{%
t\left( v_{\varepsilon ,ij}\right) _{3}}{D}\right) dt \\ 
\quad =\dfrac{h}{\varepsilon ^{-\theta }h}\dint\nolimits_{0}^{\varepsilon
^{-\theta }h}\dfrac{1}{q_{i}\left( s\right) q_{j}\left( s\right) }\exp
\left( \varepsilon ^{\theta }s\dfrac{\left( v_{\varepsilon ,ij}\right) _{3}}{%
D}\right) ds\underset{\varepsilon \rightarrow 0}{\rightarrow }h\left\langle
1/q^{2}\right\rangle ,%
\end{array}%
\end{equation*}%
using the ergodicity property (\ref{Ergo}). Thus, using the proof of Lemma %
\ref{conv-vites-fiss}, we get%
\begin{equation*}
\begin{array}{l}
\underset{\varepsilon \rightarrow 0}{\lim }\left(
D\dint\nolimits_{Y_{\varepsilon }\left( \omega \right) }\dfrac{\partial
u_{\varepsilon }}{\partial x_{3}}\dfrac{\partial u_{0,\varepsilon }}{%
\partial x_{3}}dx-\dint\nolimits_{Y_{\varepsilon }\left( \omega \right)
}\left( \left( v_{\varepsilon ,s}\right) _{3}\dfrac{\partial
u_{0,\varepsilon }}{\partial x_{3}}\right) u_{\varepsilon }dx\right) \\ 
\quad =\dfrac{D}{h\left\langle 1/q^{2}\right\rangle }\dint\nolimits_{\Sigma
}\left( u_{0}^{+}\left( u^{-}-u^{+}\right) -u_{0}^{-}\left(
u^{-}-u^{+}\right) \exp \left( \dfrac{p_{0}^{+}-p_{0}^{-}}{D\langle
q^{2}\rangle \left\langle 1/q^{2}\right\rangle }\dfrac{k_{0}}{\mu }\right)
\right) dx^{\prime }.%
\end{array}%
\end{equation*}

We now compute, using Lemma \ref{coef-dif}%
\begin{equation*}
\begin{array}{l}
\underset{\varepsilon \rightarrow 0}{\lim }\left(
D\dint\nolimits_{Y_{\varepsilon }\left( \omega \right) }\nabla _{\tau
}u_{\varepsilon }\cdot \nabla _{\tau }u_{0,\varepsilon
}dx-\dint\nolimits_{Y_{\varepsilon }\left( \omega \right) }\left( \left(
v_{\varepsilon }\right) _{\tau }\cdot \nabla _{\tau }u_{0,\varepsilon
}\right) u_{\varepsilon }dx\right) \\ 
\quad 
\begin{array}{ll}
= & \underset{\varepsilon \rightarrow 0}{\lim }\left(
\dint\nolimits_{Y_{\varepsilon }\left( \omega \right) }D^{\ast }\nabla
_{\tau }u_{\varepsilon }\cdot \left( \nabla _{\tau }u\right) _{0,\varepsilon
}dx-\dint\nolimits_{Y_{\varepsilon }\left( \omega \right) }\left( \left(
v_{\varepsilon }\right) _{\tau }\cdot \left( \nabla _{\tau }u\right)
_{0,\varepsilon }\right) u_{\varepsilon }dx\right) \\ 
= & h\left\langle q^{2}\left( 0\right) \right\rangle \dint\nolimits_{\Sigma
}D^{\ast }\nabla _{\tau }u_{0}^{+}\cdot \left( \nabla _{\tau }u^{+}\right)
dx^{\prime }-h\left\langle q^{2}\left( 0\right) \right\rangle
\dint\nolimits_{\Sigma }\left( \left( v_{0}\right) _{\tau }\cdot \left(
\nabla _{\tau }u^{+}\right) \right) u_{0}^{+}dx^{\prime },%
\end{array}%
\end{array}%
\end{equation*}%
which leads to the limit variational formulation (\ref{varform}).
\end{proof}

The problem associated to this limit variational formulation (\ref{varform})
is given in the following Corollary.

\begin{corollary}
\label{corollary2}The sequence $\left( u_{\varepsilon }\right) _{\varepsilon
}$, where $u_{\varepsilon }$ is the solution of (\ref{eq9}), $\tau _{1}$%
-converges to the solution $u_{0}$ of the problem%
\begin{equation}
\left\{ 
\begin{array}{rlll}
-\func{div}\left( \widehat{D}\nabla u_{0}^{+}\right) +v_{0,d}^{+}\cdot
\nabla u_{0}^{+} & = & \left\vert Z^{1}\right\vert f & \text{in }\Omega ^{+},
\\ 
-\func{div}\left( \widehat{D}\nabla u_{0}^{-}\right) +v_{0,d}^{-}\cdot
\nabla u_{0}^{-}\quad & = & 0 & \text{in }\Omega _{h}^{-}, \\ 
-\widehat{D}\nabla u^{+}\cdot e_{3}\quad &  &  &  \\ 
-h\left\langle q^{2}\left( 0\right) \right\rangle \func{div}_{\tau }\left(
D^{\ast }\nabla u_{0}^{+}\right) \quad &  &  &  \\ 
+h\left\langle q^{2}\left( 0\right) \right\rangle \left( v_{0,f}\right)
_{\tau }\cdot \nabla _{\tau }u_{0}^{+} & = & \dfrac{D}{h\left\langle
1/q^{2}\left( 0\right) \right\rangle }\left( 
\begin{array}{l}
u_{0}^{+}\left( .,0\right) \\ 
-u_{0}^{-}\left( .,-h\right) A%
\end{array}%
\right) & \text{on }\Gamma _{0}^{+}, \\ 
D^{\ast }\nabla _{\tau }u_{0}^{+}\cdot n_{\tau } & = & 0 & \text{on }%
\partial \Sigma , \\ 
\widehat{D}\nabla u_{0}^{-}\cdot e_{3} & = & \dfrac{-D}{h\left\langle
1/q^{2}\left( 0\right) \right\rangle }\left( 
\begin{array}{l}
u_{0}^{+}\left( .,0\right) \\ 
-u_{0}^{-}\left( .,-h\right) A%
\end{array}%
\right) & \text{on }\Gamma _{h}^{-}, \\ 
u_{0}^{+} & = & 0 & \text{on }\Gamma ^{+}, \\ 
u_{0}^{-} & = & 0 & \text{on }\Gamma ^{-},%
\end{array}%
\right.  \label{equ33}
\end{equation}%
where $A=\exp \left( \frac{p_{0}^{+}-p_{0}^{-}}{D\langle q^{2}\rangle
\langle 1/q^{2}\rangle }\frac{k_{0}}{\mu }\right) $, $u_{0}\mid _{\Omega
^{+}}=:u_{0}^{+}$ and $u_{0}\mid _{\Omega _{h}^{-}}=:u_{0}^{-}$.
\end{corollary}

\begin{proof}
This is an immediate consequence of the limit variational formulation (\ref%
{varform}).
\end{proof}

\begin{remark}
\label{Remark2}Consider the case of a dispersive contaminant with a
diffusion coefficient $D\left( x,\omega \right) $ defined through $D\left(
x,\omega \right) =D_{mol}+D_{disp}\left( x,\omega \right) $ with%
\begin{equation*}
D_{disp}\left( x,\omega \right) =\left\{ 
\begin{array}{ll}
D_{disp}\left( x\right) & \text{in }\Omega ^{+}\cup \Omega _{h}^{-}, \\ 
D_{disp}\left( x_{1},x_{2},-\varepsilon ^{-\theta }x_{3}+\alpha _{ij}\left(
\omega \right) ,\omega \right) & \text{in }Y_{\varepsilon ,ij}\left( \omega
\right) ,%
\end{array}%
\right.
\end{equation*}%
where $\left( \alpha _{ij}\left( \omega \right) \right) _{i,j\in \mathbb{Z}}$
is a sequence of random variables such that $\left\vert \alpha _{ij}\left(
\omega \right) \right\vert \leq C$,\ $\forall i,j\in \mathbb{Z}$, with
probability 1, $C$ being some non-random constant. We suppose that $D$ is
continuous with respect to the variable $x$ and, with probability 1, $%
d_{0}\leq D\left( x,\omega \right) \leq d_{1}$, where $d_{0}$ and $d_{1}$
are positive and non-random constants. We suppose that $D_{disp}$ is a
stationary random process.

Let $u\in C^{2}\left( \overline{\Omega }\right) $ be such that $u=0$ on $%
\Gamma $. We build the modified test-function $\bar{u}_{\varepsilon ,ij}$
inside the fissure $Y_{\varepsilon ,ij}\left( \omega \right) $%
\begin{equation*}
\begin{array}{l}
\bar{u}_{\varepsilon ,ij}\left( x_{1},x_{2},x_{3}\right) =u^{+}\left(
x_{1},x_{2},0\right) \\ 
\quad +\dfrac{\left( u^{-}\left( x_{1},x_{2},-h\right) -u^{+}\left(
x_{1},x_{2},0\right) \right) \dint\nolimits_{x_{3}}^{0}\dfrac{\exp \left(
-t\left( v_{\varepsilon ,ij}\right) _{3}\right) }{q_{i}\left( -\varepsilon
^{-\theta }t\right) q_{j}\left( -\varepsilon ^{-\theta }t\right) D_{\ast
}\left( -\varepsilon ^{-\theta }t\right) }dt}{\dint\nolimits_{-h}^{0}\dfrac{%
\exp \left( -t\left( v_{\varepsilon ,ij}\right) _{3}\right) }{q_{i}\left(
-\varepsilon ^{-\theta }t\right) q_{j}\left( -\varepsilon ^{-\theta
}t\right) D_{\ast }\left( -\varepsilon ^{-\theta }t\right) }dt},%
\end{array}%
\end{equation*}%
where $\left( D_{\ast }\right) \left( s\right) =\left(
D_{mol}+D_{disp}\left( i\varepsilon ,j\varepsilon ,s+\alpha _{ij}\left(
\omega \right) ,\omega \right) \right) $. Implementing this test-function in
the above process, one gets at the limit a problem similar to (\ref{equ33}),
except that $\left\langle 1/q^{2}\left( 0\right) \right\rangle $ is now
replaced by $\left\langle 1/D_{\ast }\left( .\right) q^{2}\left( 0\right)
\right\rangle $, where $\left\langle 1/\left( D_{\ast }q^{2}\left( 0\right)
\right) \right\rangle $ is the mathematical expectation of $1/\left(
q^{2}\left( t\right) D_{\ast }\left( .,t\right) \right) $, with respect to
the measure probability $P$, and $b_{j}$ is replaced by the $Z$-periodic
solution $b_{j}$ of the problem%
\begin{equation*}
\left\{ 
\begin{array}{rllll}
\func{div}\left( D\left( x\right) \left( e_{j}+\nabla b_{j}\right) \right) & 
= & 0 & \text{in }Z^{1} & j=1,2,3, \\ 
\left( \nabla b_{j}+e_{j}\right) \cdot n & = & 0 & \text{on }S, & 
\end{array}%
\right.
\end{equation*}%
and $c_{m}$ by%
\begin{equation*}
\left\{ 
\begin{array}{rllll}
\func{div}\left( D\left( x,\omega \right) \left( e_{m}+\nabla c_{m}\right)
\right) & = & 0 & \text{in }Z^{\prime } & m=1,2, \\ 
\left( \nabla c_{m}+e_{m}\right) \cdot n & = & 0 & \text{on }\partial
Z^{\prime }. & 
\end{array}%
\right.
\end{equation*}
\end{remark}

\subsection{The asymptotic behaviour in the case of a reactive contaminant ($%
\mathcal{R}>0$)}

In this subsection, we consider the reaction-diffusion equation (\ref{eq9})
with first-order reaction, that is with $\mathcal{R}>0$ (see for example 
\cite{Bear}, \cite{Horn}). We denote $w_{\varepsilon ,ij}$ the solution of
the differential equation%
\begin{equation}
\left\{ 
\begin{array}{rll}
-D\dfrac{\partial }{\partial x_{3}}\left( \left( q_{i}q_{j}\right) \left(
-\varepsilon ^{-\theta }x_{3}\right) \dfrac{\partial w_{\varepsilon ,ij}}{%
\partial x_{3}}\right) \left( x_{3}\right) -\left( q_{i}q_{j}\right) \left(
-\varepsilon ^{-\theta }x_{3}\right) v_{\varepsilon ,ij}\dfrac{\partial
w_{\varepsilon ,ij}}{\partial x_{3}}\left( x_{3}\right) &  &  \\ 
+\mathcal{R}\left( q_{i}q_{j}\right) \left( -\varepsilon ^{-\theta
}x_{3}\right) w_{\varepsilon ,ij}\left( x_{3}\right) & = & 0, \\ 
w_{\varepsilon ,ij}\left( 0\right) & = & 1, \\ 
w_{\varepsilon ,ij}^{\prime }\left( 0\right) & = & 0%
\end{array}%
\right.  \label{equ44}
\end{equation}%
and $z_{\varepsilon ,ij}$ the solution of%
\begin{equation}
\left\{ 
\begin{array}{rll}
-D\dfrac{\partial }{\partial x_{3}}\left( \left( q_{i}q_{j}\right) \left(
-\varepsilon ^{-\theta }x_{3}\right) \dfrac{\partial z_{\varepsilon ,ij}}{%
\partial x_{3}}\right) \left( x_{3}\right) -\left( q_{i}q_{j}\right) \left(
-\varepsilon ^{-\theta }x_{3}\right) v_{\varepsilon ,ij}\dfrac{\partial
z_{\varepsilon ,ij}}{\partial x_{3}}\left( x_{3}\right) &  &  \\ 
+\mathcal{R}\left( q_{i}q_{j}\right) \left( -\varepsilon ^{-\theta
}x_{3}\right) z_{\varepsilon ,ij}\left( x_{3}\right) & = & 0, \\ 
z_{\varepsilon ,ij}\left( 0\right) & = & 0, \\ 
z_{\varepsilon ,ij}^{\prime }\left( 0\right) & = & \dfrac{1}{%
q_{i}q_{j}\left( 0\right) },%
\end{array}%
\right.  \label{equ45}
\end{equation}%
where $v_{\varepsilon ,ij}$ is the velocity defined in (\ref{veloci}). We
have the following estimates.

\begin{proposition}
\label{proposition2}There exist non-random positive constants $C_{0}$ and $%
C_{1}$ independent of $\varepsilon $ and of $i$ and $j$, such that:

\begin{enumerate}
\item $\forall \varepsilon >0$, $\forall \left( i,j\right) \in
I_{\varepsilon }\left( \omega \right) $, $\forall x_{3}\in \left[ -h,0\right]
:1\leq w_{\varepsilon ,ij}\left( x_{3}\right) \leq C_{1}$ and $%
-C_{0}^{-1}x_{3}\leq z_{\varepsilon ,ij}\left( x_{3}\right) \leq C_{1}$.

\item $\forall \left( i,j\right) \in I_{\varepsilon }\left( \omega \right) $%
, $\forall x_{3}\in \left[ -h,0\right] $%
\begin{equation*}
\left\{ 
\begin{array}{rll}
\underset{\varepsilon \rightarrow 0}{\lim }\left\vert w_{\varepsilon
,ij}\left( x_{3}\right) -\cosh \left( \widehat{\mathcal{R}}x_{3}\right)
\right\vert & = & 0, \\ 
\underset{\varepsilon \rightarrow 0}{\lim }\left\vert z_{\varepsilon
,ij}\left( x_{3}\right) -\dfrac{\left\langle 1/q^{2}\left( 0\right)
\right\rangle }{\widehat{\mathcal{R}}}\sinh \left( \widehat{\mathcal{R}}%
x_{3}\right) \right\vert & = & 0, \\ 
\underset{\varepsilon \rightarrow 0}{\lim }\left\vert w_{\varepsilon
,ij}^{\prime }\left( x_{3}\right) \left( q_{i}q_{j}\right) \left(
-\varepsilon ^{-\theta }x_{3}\right) \exp \left( \dfrac{x_{3}v_{\varepsilon
,ij}}{D}\right) -\dfrac{\widehat{\mathcal{R}}}{\left\langle 1/q^{2}\left(
0\right) \right\rangle }\sinh \left( \widehat{\mathcal{R}}x_{3}\right)
\right\vert & = & 0, \\ 
\underset{\varepsilon \rightarrow 0}{\lim }\left\vert z_{\varepsilon
,ij}^{\prime }\left( x_{3}\right) \left( q_{i}q_{j}\right) \left(
-\varepsilon ^{-\theta }x_{3}\right) \exp \left( \dfrac{x_{3}v_{\varepsilon
,ij}}{D}\right) -\cosh \left( \widehat{\mathcal{R}}x_{3}\right) \right\vert
& = & 0,%
\end{array}%
\right.
\end{equation*}%
where $\widehat{\mathcal{R}}=\sqrt{\mathcal{R}\left\langle q^{2}\left(
0\right) \right\rangle \left\langle 1/q^{2}\left( 0\right) \right\rangle /D}$%
.
\end{enumerate}
\end{proposition}

\begin{proof}
1. Multiplying the equations (\ref{equ44}) and (\ref{equ45}) by $\exp \left(
x_{3}v_{\varepsilon ,ij}/D\right) $ and integrating by parts the first term
of these equations, we obtain the Voltera type integral equations%
\begin{equation}
\left\{ 
\begin{array}{rll}
w_{\varepsilon ,ij}\left( x_{3}\right) & = & \dfrac{\mathcal{R}}{D}%
\dint_{0}^{x_{3}}\left( q_{i}q_{j}\right) \left( -\varepsilon ^{-\theta
}s\right) \exp \left( s\dfrac{v_{\varepsilon ,ij}}{D}\right) w_{\varepsilon
,ij}\left( s\right) \\ 
&  & \quad \times \left( \dint_{s}^{x_{3}}\dfrac{\exp \left( -\zeta \dfrac{%
v_{\varepsilon ,ij}}{D}\right) }{\left( q_{i}q_{j}\right) \left(
-\varepsilon ^{-\theta }\zeta \right) }d\zeta \right) ds+1, \\ 
z_{\varepsilon ,ij}\left( x_{3}\right) & = & \dfrac{\mathcal{R}}{D}\left(
\dint_{0}^{x_{3}}\left( q_{i}q_{j}\right) \left( -\varepsilon ^{-\theta
}s\right) \exp \left( s\dfrac{v_{\varepsilon ,ij}}{D}\right) z_{\varepsilon
,ij}\left( s\right) \right) \\ 
&  & \quad \times \left( \dint_{s}^{x_{3}}\dfrac{\exp \left( -\zeta \dfrac{%
v_{\varepsilon ,ij}}{D}\right) }{\left( q_{i}q_{j}\right) \left(
-\varepsilon ^{-\theta }\zeta \right) }d\zeta \right) ds+\dint_{0}^{x_{3}}%
\dfrac{\exp \left( -\zeta \dfrac{v_{\varepsilon ,ij}}{D}\right) }{\left(
q_{i}q_{j}\right) \left( -\varepsilon ^{-\theta }\zeta \right) }d\zeta ,%
\end{array}%
\right.  \label{Volt}
\end{equation}%
(which can be solved by the method of successive approximations). Taking
into account the hypothesis (\ref{equ5}), we obtain the first point of the
Proposition.

2. Consider the integral equations%
\begin{equation*}
\left\{ 
\begin{array}{rll}
w_{ij}\left( x_{3}\right) & = & \dfrac{\mathcal{R}}{D}\left\langle
q^{2}\left( 0\right) \right\rangle \left\langle 1/q^{2}\left( 0\right)
\right\rangle \dint_{0}^{x_{3}}\left( x_{3}-s\right) w_{ij}\left( s\right)
ds+1, \\ 
z_{ij}\left( x_{3}\right) & = & \dfrac{\mathcal{R}}{D}\left\langle
q^{2}\left( 0\right) \right\rangle \left\langle 1/q^{2}\left( 0\right)
\right\rangle \dint_{0}^{x_{3}}\left( x_{3}-s\right) z_{ij}\left( s\right)
ds+x_{3}\left\langle 1/q^{2}\left( 0\right) \right\rangle ,%
\end{array}%
\right.
\end{equation*}%
whose solutions are $w_{ij}\left( x_{3}\right) =\cosh \left( \widehat{%
\mathcal{R}}x_{3}\right) $ and $z_{ij}\left( x_{3}\right) =\left\langle
1/q^{2}\left( 0\right) \right\rangle \sinh \left( \widehat{\mathcal{R}}%
x_{3}\right) /\widehat{\mathcal{R}}$ respectively. The differences $%
w_{\varepsilon ,ij}-w_{ij}$ and $z_{\varepsilon ,ij}-z_{ij}$ satisfy
integral equations of Voltera type. Using the ergodic result (\ref{Ergo}),
we prove the convergences%
\begin{equation*}
\underset{\varepsilon \rightarrow 0}{\lim }\left\vert w_{\varepsilon
,ij}\left( x_{3}\right) -w_{ij}\left( x_{3}\right) \right\vert =\underset{%
\varepsilon \rightarrow 0}{\lim }\left\vert z_{\varepsilon ,ij}\left(
x_{3}\right) -z_{ij}\left( x_{3}\right) \right\vert =0,
\end{equation*}%
uniformly with respect to $x_{3}\in \left[ -h,0\right] $.

The last estimates for the derivatives follow from the derivation of the
equations (\ref{Volt}).
\end{proof}

Let $u\in C^{2}\left( \overline{\Omega }\right) $ be such that $u=0$ on $%
\Gamma ^{+}\cup \Gamma ^{-}$. We here define the test-function $\bar{u}%
_{\varepsilon ,ij}$ inside the fissure $Y_{\varepsilon ,ij}\left( \omega
\right) $ through%
\begin{equation*}
\bar{u}_{\varepsilon ,ij}\left( x\right) =u^{+}\left( x_{1},x_{2},0\right)
w_{\varepsilon ,ij}\left( x_{3}\right) +\frac{u^{-}\left(
x_{1},x_{2},-h\right) -u^{+}\left( x_{1},x_{2},0\right) w_{\varepsilon
,ij}\left( -h\right) }{z_{\varepsilon ,ij}\left( -h\right) }z_{\varepsilon
,ij}\left( x_{3}\right) .
\end{equation*}

We finally define the test-function $u_{0,\varepsilon }$ in the same way as (%
\ref{equ42}). It is easily proved that the sequence $\left( u_{0,\varepsilon
}\right) _{\varepsilon }$ $\tau _{1}$-converges to $u$. On the other hand,
we compute%
\begin{equation*}
\left\{ 
\begin{array}{rll}
u_{0,\varepsilon }\left( x^{\prime },0\right) & = & u^{+}\left(
x_{1},x_{2},0\right) , \\ 
\bar{u}_{\varepsilon ,ij}\left( x^{\prime },-h\right) & = & u^{-}\left(
x_{1},x_{2},-h\right) , \\ 
\dfrac{\partial u_{0,\varepsilon }}{\partial x_{3}}\mid _{x_{3}=0} & = & 
\dfrac{u^{-}\left( x_{1},x_{2},-h\right) -u^{+}\left( x_{1},x_{2},0\right)
w_{\varepsilon ,ij}\left( -h\right) }{z_{\varepsilon ,ij}\left( -h\right) }%
\dfrac{1}{q_{i}q_{j}\left( 0\right) }, \\ 
\dfrac{\partial u_{0,\varepsilon }}{\partial x_{3}}\mid _{x_{3}=-h} & = & 
u^{+}\left( x_{1},x_{2},0\right) w_{\varepsilon ,ij}^{\prime }\left(
-h\right) \\ 
&  & \quad +\dfrac{u^{-}\left( x_{1},x_{2},-h\right) -u^{+}\left(
x_{1},x_{2},0\right) w_{\varepsilon ,ij}\left( -h\right) }{z_{\varepsilon
,ij}\left( -h\right) }z_{\varepsilon ,ij}^{\prime }\left( -h\right) .%
\end{array}%
\right.
\end{equation*}

We have the following result.

\begin{lemma}
One has:%
\begin{equation*}
\begin{array}{l}
\underset{\varepsilon \rightarrow 0}{\lim }\underset{\left( i,j\right) \in
I_{\varepsilon }\left( \omega \right) }{\dsum }\dint\nolimits_{\Gamma
_{0,\varepsilon ,ij}^{+}\left( \omega \right) }\dfrac{\partial
u_{0,\varepsilon }}{\partial x_{3}}\mid _{x_{3}=0}\left( q_{i}q_{j}\right)
\left( 0\right) u_{\varepsilon }^{+}d\sigma \\ 
\text{\quad }=\dfrac{-1}{\left\langle 1/q^{2}\left( 0\right) \right\rangle }%
\dint_{\Sigma }\dfrac{\widehat{\mathcal{R}}}{\sinh \left( \widehat{\mathcal{R%
}}h\right) }\left( u^{-}-u^{+}\cosh \left( \widehat{\mathcal{R}}h\right)
\right) u_{0}^{+}dx^{\prime }, \\ 
\underset{\varepsilon \rightarrow 0}{\lim }\underset{\left( i,j\right) \in
I_{\varepsilon }\left( \omega \right) }{\dsum }\dint\nolimits_{\Gamma
_{h,\varepsilon ,ij}^{-}\left( \omega \right) }\dfrac{\partial
u_{0,\varepsilon }}{\partial x_{3}}\mid _{x_{3}=-h}\left( q_{i}q_{j}\right)
\left( h\varepsilon ^{-\theta }\right) u_{\varepsilon }^{-}d\sigma \\ 
\text{\quad }=\dfrac{-1}{\left\langle 1/q^{2}\left( 0\right) \right\rangle }%
\dint_{\Sigma }\dfrac{\widehat{\mathcal{R}}}{\sinh \left( \widehat{\mathcal{R%
}}h\right) }\left( u^{-}\cosh \left( \widehat{\mathcal{R}}h\right)
-u^{+}\right) \exp \left( \dfrac{p_{0}^{+}-p_{0}^{-}}{D\langle
1/q^{2}\rangle }\dfrac{k_{0}}{\mu }\right) u_{0}^{-}dx^{\prime },%
\end{array}%
\end{equation*}
\end{lemma}

\begin{proof}
Using the estimates of Proposition \ref{proposition2}, we can replace, when $%
\varepsilon $ is small enough, $w_{\varepsilon ,ij}\left( -h\right) $ by $%
\cosh \left( \widehat{\mathcal{R}}h\right) $, $z_{\varepsilon ij}\left(
-h\right) $ by $-\left\langle 1/q^{2}\left( 0\right) \right\rangle \sinh
\left( \widehat{\mathcal{R}}h\right) /\widehat{\mathcal{R}}$, $%
w_{\varepsilon ,ij}^{\prime }\left( -h\right) $ by the quantity $-\frac{\exp
\left( hv_{\varepsilon ,ij}/D\right) }{\left( q_{i}q_{j}\right) \left(
\varepsilon ^{-\theta }h\right) }\frac{\widehat{\mathcal{R}}}{\left\langle
1/q^{2}\left( 0\right) \right\rangle }\sinh \left( \widehat{\mathcal{R}}%
h\right) $ and $z_{\varepsilon ,ij}^{\prime }\left( -h\right) $ by the
quantity $\frac{\exp \left( hv_{\varepsilon ,ij}/D\right) }{\left(
q_{i}q_{j}\right) \left( \varepsilon ^{-\theta }h\right) }\cosh \left( 
\widehat{\mathcal{R}}h\right) $.\ Thus%
\begin{equation*}
\begin{array}{l}
\dfrac{\partial u_{0,\varepsilon }}{\partial x_{3}}\mid _{x_{3}=-h}\left(
q_{i}q_{j}\right) \left( h\varepsilon ^{-\theta }\right) \\ 
\begin{array}{ll}
= & \left( 
\begin{array}{l}
u^{+}\left( x_{1},x_{2},0\right) w_{\varepsilon ,ij}^{\prime }\left(
-h\right) \\ 
\quad +\dfrac{u^{-}\left( x_{1},x_{2},-h\right) -u^{+}\left(
x_{1},x_{2},0\right) w_{\varepsilon ,ij}\left( -h\right) }{z_{\varepsilon
,ij}\left( -h\right) }z_{\varepsilon ,ij}^{\prime }\left( -h\right)%
\end{array}%
\right) \left( q_{i}q_{j}\right) \left( h\varepsilon ^{-\theta }\right) \\ 
\underset{\varepsilon \rightarrow 0}{\sim } & \left( 
\begin{array}{l}
-u^{+}\left( x_{1},x_{2},0\right) \dfrac{\exp \left( hv_{\varepsilon
,ij}/D\right) }{\left( q_{i}q_{j}\right) \left( \varepsilon ^{-\theta
}h\right) }\dfrac{\widehat{\mathcal{R}}}{\left\langle 1/q^{2}\left( 0\right)
\right\rangle }\sinh \left( \widehat{\mathcal{R}}h\right) \\ 
-\widehat{\mathcal{R}}\dfrac{u^{-}\left( x_{1},x_{2},-h\right) -u^{+}\left(
x_{1},x_{2},0\right) \cosh \left( \widehat{\mathcal{R}}h\right) }{%
\left\langle 1/q^{2}\left( 0\right) \right\rangle \sinh \left( \widehat{%
\mathcal{R}}h\right) } \\ 
\quad \times \dfrac{\exp \left( hv_{\varepsilon ,ij}/D\right) }{\left(
q_{i}q_{j}\right) \left( \varepsilon ^{-\theta }h\right) }\cosh \left( 
\widehat{\mathcal{R}}h\right)%
\end{array}%
\right) \left( q_{i}q_{j}\right) \left( h\varepsilon ^{-\theta }\right) \\ 
= & \dfrac{\widehat{\mathcal{R}}}{\left\langle 1/q^{2}\left( 0\right)
\right\rangle }\exp \left( hv_{\varepsilon ,ij}/D\right) \left( 
\begin{array}{l}
-u^{+}\left( x_{1},x_{2},0\right) \sinh \left( \widehat{\mathcal{R}}h\right)
\\ 
-\dfrac{u^{-}\left( x_{1},x_{2},-h\right) -u^{+}\left( x_{1},x_{2},0\right)
\cosh \left( \widehat{\mathcal{R}}h\right) }{\sinh \left( \widehat{\mathcal{R%
}}h\right) } \\ 
\quad \times \cosh \left( \widehat{\mathcal{R}}h\right)%
\end{array}%
\right) \\ 
= & \dfrac{\widehat{\mathcal{R}}}{\left\langle 1/q^{2}\left( 0\right)
\right\rangle \sinh \left( \widehat{\mathcal{R}}h\right) }\exp \left(
hv_{\varepsilon ,ij}/D\right) \\ 
& \quad \times \left( 
\begin{array}{l}
u^{+}\left( x_{1},x_{2},0\right) \left( -\left( \sinh \right) ^{2}\left( 
\widehat{\mathcal{R}}h\right) +\left( \cosh \right) ^{2}\left( \widehat{%
\mathcal{R}}h\right) \right) \\ 
-u^{-}\left( x_{1},x_{2},-h\right) \cosh \left( \widehat{\mathcal{R}}h\right)%
\end{array}%
\right) \\ 
= & \dfrac{\widehat{\mathcal{R}}}{\left\langle 1/q^{2}\left( 0\right)
\right\rangle \sinh \left( \widehat{\mathcal{R}}h\right) }\exp \left(
hv_{\varepsilon ,ij}/D\right) \left( u^{+}\left( x_{1},x_{2},0\right)
-u^{-}\left( x_{1},x_{2},-h\right) \cosh \left( \widehat{\mathcal{R}}%
h\right) \right) ,%
\end{array}%
\end{array}%
\end{equation*}%
whence%
\begin{equation*}
\begin{array}{l}
\underset{\varepsilon \rightarrow 0}{\lim }\underset{\left( i,j\right) \in
I_{\varepsilon }\left( \omega \right) }{\dsum }\dint\nolimits_{\Gamma
_{h,\varepsilon ,ij}^{-}\left( \omega \right) }\dfrac{\partial
u_{0,\varepsilon }}{\partial x_{3}}\mid _{x_{3}=-h}\left( q_{i}q_{j}\right)
\left( h\varepsilon ^{-\theta }\right) u_{\varepsilon }^{-}d\sigma \\ 
\quad =\dfrac{-1}{\left\langle 1/q^{2}\left( 0\right) \right\rangle }%
\dint_{\Sigma }\dfrac{\widehat{\mathcal{R}}}{\sinh \left( \widehat{\mathcal{R%
}}h\right) }\left( u^{-}\cosh \left( \widehat{\mathcal{R}}h\right)
-u^{+}\right) \exp \left( \dfrac{p_{0}^{+}-p_{0}^{-}}{D\langle
1/q^{2}\rangle }\dfrac{k_{0}}{\mu }\right) u_{0}^{-}dx^{\prime }.%
\end{array}%
\end{equation*}

Similarly, we get%
\begin{equation*}
\begin{array}{l}
\underset{\varepsilon \rightarrow 0}{\lim }\underset{\left( i,j\right) \in
I_{\varepsilon }\left( \omega \right) }{\dsum }\dint\nolimits_{\Gamma
_{0,\varepsilon ,ij}^{+}\left( \omega \right) }\dfrac{\partial
u_{0,\varepsilon }}{\partial x_{3}}\mid _{x_{3}=0}\left( q_{i}q_{j}\right)
\left( 0\right) u_{\varepsilon }^{+}d\sigma \\ 
\begin{array}{ll}
= & \underset{\varepsilon \rightarrow 0}{\lim }\underset{\left( i,j\right)
\in I_{\varepsilon }\left( \omega \right) }{\dsum }\dint\nolimits_{\Gamma
_{0,\varepsilon ,ij}^{+}\left( \omega \right) }\dfrac{u^{-}\left(
x_{1},x_{2},-h\right) -u^{+}\left( x_{1},x_{2},0\right) w_{\varepsilon
,ij}\left( -h\right) }{z_{\varepsilon ,ij}\left( -h\right) }\dfrac{1}{%
q_{i}q_{j}\left( 0\right) }\left( q_{i}q_{j}\right) \left( 0\right)
u_{\varepsilon }^{+}d\sigma \\ 
= & \underset{\varepsilon \rightarrow 0}{\lim }\underset{\left( i,j\right)
\in I_{\varepsilon }\left( \omega \right) }{\dsum }\dint\nolimits_{\Gamma
_{0,\varepsilon ,ij}^{+}\left( \omega \right) }\dfrac{u^{-}\left(
x_{1},x_{2},-h\right) -u^{+}\left( x_{1},x_{2},0\right) w_{\varepsilon
,ij}\left( -h\right) }{z_{\varepsilon ,ij}\left( -h\right) }u_{\varepsilon
}^{+}d\sigma \\ 
= & \underset{\varepsilon \rightarrow 0}{\lim }\underset{\left( i,j\right)
\in I_{\varepsilon }\left( \omega \right) }{\dsum }\dint\nolimits_{\Gamma
_{0,\varepsilon ,ij}^{+}\left( \omega \right) }-\widehat{\mathcal{R}}\dfrac{%
u^{-}\left( x_{1},x_{2},-h\right) -u^{+}\left( x_{1},x_{2},0\right) \cosh
\left( \widehat{\mathcal{R}}h\right) }{\left\langle 1/q^{2}\left( 0\right)
\right\rangle \sinh \left( \widehat{\mathcal{R}}h\right) }u_{\varepsilon
}^{+}d\sigma \\ 
= & \dfrac{-1}{\left\langle 1/q^{2}\left( 0\right) \right\rangle }%
\dint_{\Sigma }\dfrac{\widehat{\mathcal{R}}}{\sinh \left( \widehat{\mathcal{R%
}}h\right) }\left( u^{-}-u^{+}\cosh \left( \widehat{\mathcal{R}}h\right)
\right) u_{0}^{+}dx^{\prime },%
\end{array}%
\end{array}%
\end{equation*}%
which leads to the desired limit.
\end{proof}

Now, using the same methods as in the above subsection, we obtain the
following Theorem.

\begin{theorem}
The sequence $\left( u_{\varepsilon }\right) _{\varepsilon }$, where $%
u_{\varepsilon }$ is the solution of (\ref{eq9}), $\tau _{1}$-converges to
the solution $u_{0}$ of the problem%
\begin{equation}
\left\{ 
\begin{array}{rcll}
-\func{div}\left( \widehat{D}\nabla u_{0}^{+}\right) +v_{0,d}^{+}\cdot
\nabla u_{0}^{+}+\mathcal{R}u_{0}^{+} & = & \left\vert Z^{1}\right\vert f & 
\text{in }\Omega ^{+}, \\ 
-\func{div}\left( \widehat{D}\nabla u_{0}^{-}\right) +v_{0,d}^{-}\cdot
\nabla u_{0}^{-}+\mathcal{R}u_{0}^{-} & = & 0 & \text{in }\Omega _{h}^{-},
\\ 
-\widehat{D}\nabla u_{0}^{+}.e_{3}-h\left\langle q^{2}\left( 0\right)
\right\rangle \func{div}_{\tau }\left( D^{\ast }\nabla u_{0}^{+}\right) &  & 
&  \\ 
+h\left\langle q^{2}\left( 0\right) \right\rangle \left( v_{0,f}\right)
_{\tau }\cdot \nabla _{\tau }u_{0}^{+} & = & \left( 
\begin{array}{l}
\dfrac{D}{\left\langle 1/q^{2}\left( 0\right) \right\rangle }\dfrac{\widehat{%
\mathcal{R}}}{\sinh \left( \widehat{\mathcal{R}}h\right) } \\ 
\times \left( u_{0}^{+}\cosh \left( \widehat{\mathcal{R}}h\right)
-u_{0}^{-}A\right)%
\end{array}%
\right) & \text{on }\Gamma _{0}^{+}, \\ 
D^{\ast }\nabla _{\tau }u_{0}^{+}\cdot n_{\tau } & = & 0 & \text{on }%
\partial \Sigma , \\ 
\widehat{D}\nabla u_{0}^{-}\cdot e_{3} & = & \left( 
\begin{array}{l}
\dfrac{D}{\left\langle 1/q^{2}\left( 0\right) \right\rangle }\dfrac{\widehat{%
\mathcal{R}}}{\sinh \left( \widehat{\mathcal{R}}h\right) } \\ 
\times \left( u_{0}^{-}\cosh \left( \widehat{\mathcal{R}}h\right)
A-u_{0}^{+}\right)%
\end{array}%
\right) & \text{on }\Gamma _{h}^{-}, \\ 
u_{0}^{+} & = & 0 & \text{on }\Gamma ^{+}, \\ 
u_{0}^{-} & = & 0 & \text{on }\Gamma ^{-},%
\end{array}%
\right.  \label{asympt}
\end{equation}%
with $A=\exp \left( \frac{p_{0}^{+}-p_{0}^{-}}{D\left\langle
q^{2}\right\rangle \left\langle 1/q^{2}\right\rangle }\frac{k_{0}}{\mu }%
\right) $, $u_{0}^{+}:=u_{0}\mid _{\Omega ^{+}}$ and $u_{0}^{-}:=u_{0}\mid
_{\Omega _{h}^{-}}$.
\end{theorem}

\end{document}